\documentclass[12 pt]{amsart}

\usepackage{etex}
\usepackage[shortlabels]{enumitem}
\usepackage{amsmath}
\usepackage{amscd}
\usepackage{amsthm}
\usepackage{amsfonts}
\usepackage{amssymb}
\usepackage{eucal}
\usepackage{mathrsfs}
\usepackage{morefloats}
\usepackage{thm-restate}
\usepackage{hyperref}
\usepackage[margin=1in]{geometry}
\usepackage{tikz}
\usepackage{tikz-cd, comment, fancyvrb}
\usepackage{todonotes}

\theoremstyle{plain}
\newtheorem{theorem}{Theorem}[section]
\newtheorem{proposition}[theorem]{Proposition}
\newtheorem{lemma}[theorem]{Lemma}
\newtheorem{corollary}[theorem]{Corollary}

\newtheorem*{question}{Question}

\theoremstyle{definition}
\newtheorem{definition}[theorem]{Definition}
\newtheorem{example}[theorem]{Example}

\theoremstyle{definition}
\newtheorem{condition}[theorem]{Conditions}

\newtheorem*{theorem*}{Theorem}
\newtheorem*{proposition*}{Proposition}
\newtheorem*{lemma*}{Lemma}

\theoremstyle{remark}
\newtheorem{remark}{Remark}[section]

\numberwithin{equation}{section}

\newcommand{\FF}{\mathbb F}
\newcommand{\QQ}{\mathbb Q}
\newcommand{\OO}{\mathcal O}
\newcommand{\ZZ}{\mathbb Z}

\newcommand{\on}{\operatorname}
\newcommand{\ol}{\overline}
\newcommand{\wt}{\widetilde}
\newcommand{\wh}{\widehat}
\makeatletter
\newcommand*{\defeq}{\mathrel{\rlap{%
                     \raisebox{0.3ex}{$\m@th\cdot$}}%
                     \raisebox{-0.3ex}{$\m@th\cdot$}}%
                     =}
\makeatother

\newcommand{\pp}{\mathfrak p}

\newcommand{\id}{\on{id}}
\newcommand\ab{\text{ab}}

\newcommand{\GSp}{\operatorname{GSp}}
\newcommand*{\nolink}[1]{%
  \begin{NoHyper}#1\end{NoHyper}%
}

\DeclareMathOperator\Sp{Sp}
\DeclareMathOperator\GL{GL}
\DeclareMathOperator\Gal{Gal}

\DeclareMathOperator\quo{Quo}
\DeclareMathOperator\frob{Frob}

\DeclareMathOperator\charpoly{ch}

\DeclareMathOperator\mult{mult}
\DeclareMathOperator\End{End}

\DeclareMathOperator\cyc{cyc}

\newcommand\nc{\newcommand}
\nc\renc{\renewcommand}

\newcommand\bn{{\mathbb N}}
\newcommand\bc{{\mathbb C}}
\newcommand\br{{\mathbb R}}
\newcommand\bq{{\mathbb Q}}
\newcommand\bp{{\mathbb P}}

\newcommand\bz{{\mathbb Z}}

\newcommand\fp{{\mathfrak p}}

\newcommand\so{{\mathscr O}}

\newcommand\scv{\mathscr V}

\newcommand \ra{\rightarrow}

\newcommand \spec{\text{Spec }}

\newcommand \mg{{\mathscr M_g}}
\newcommand \ag{{\mathscr A_g}}
\newcommand \coarseag{{A_g}}
\newcommand \ug{{\mathscr U_g}}

\newcommand \trigonal[1]{\mathscr T^{#1}}
\newcommand \hyperelliptic[1]{\mathscr H_{#1}}
\newcommand \geometricprimes[1]{\mathcal T_{#1}}
\newcommand\gr[1]{\on{Gr}_{#1}}

\makeatletter
\newcommand{\customlabel}[2]{%
   \protected@write \@auxout {}{\string \newlabel {#1}{{#2}{\thepage}{#2}{#1}{}} }%
   \hypertarget{#1}{#2}
}

\DeclareMathOperator\PGL{PGL}
\DeclareMathOperator\occ{Occ}
\DeclareMathOperator\PSp{PSp}

\DeclareMathOperator\galSlope{\sigma}
\DeclareMathOperator\galMin{\tau}
\newcommand{\goodGaloisSet}{T}

\newcommand\nonEtalePrimes{S'}

\newcommand{\N}{\mathbf{N}}
\newcommand{\gee}{(\bz/2\bz)^g \rtimes S_g}
\newcommand{\mc}{\mathcal}
\newcommand{\mf}{\mathfrak}
\newcommand{\wall}{H}

\newcommand{\mono}{H}

\newcommand{\zh}{\wh{\bz}}

\def\arraystretch{1.3}

\title[Surjectivity of Galois Representations in Families of Abelian Varieties]{Surjectivity of Galois Representations in \\ Rational Families of Abelian Varieties}

\date{\today}
\author[Aaron Landesman]{Aaron Landesman}

\author[Ashvin A. Swaminathan]{Ashvin A. Swaminathan}

\author[James Tao]{James Tao}

\author[Yujie Xu]{Yujie Xu}

\AtEndDocument{\bigskip{\footnotesize%
  \textsc{Department of Mathematics, Stanford University, \mbox{Stanford, CA 94305}} \par
  \textit{E-mail address}, Aaron Landesman: \texttt{aaronlandesman@stanford.edu} \par
  \addvspace{\medskipamount}
  \textsc{Department of Mathematics, Princeton University, \mbox{Princeton, NJ 08544}} \par
  \textit{E-mail address}, Ashvin A. Swaminathan: \texttt{ashvins@math.princeton.edu} \par
  \addvspace{\medskipamount}
  \textsc{Department of Mathematics, Massachusetts Institute of Technology, \mbox{Cambridge, MA 02139}} \par
\textit{E-mail address}, James Tao: \texttt{jamestao@mit.edu}
\par
  \addvspace{\medskipamount}
  \textsc{Department of Mathematics, Harvard University, \mbox{Cambridge, MA 02138}} \par
  \textit{E-mail address}, Yujie Xu: \texttt{yujiex@math.harvard.edu}
}}

\begin{document}

\begin{abstract}
	In this article, we show that for any non-isotrivial family of abelian varieties over a rational base with big monodromy, those members that have adelic Galois representation with image as large as possible form a density-$1$ subset. Our results can be applied to a number of interesting families of abelian varieties, such as rational families dominating the moduli of Jacobians of hyperelliptic curves, trigonal curves, or plane curves.
As a consequence, we prove that for any dimension $g \geq 3$, there are infinitely many abelian varieties over $\mathbb Q$ with adelic
Galois representation having image equal to all of $\operatorname{GSp}_{2g}(\widehat{\mathbb Z})$.
\end{abstract}


\maketitle

\section{Introduction and Statement of Results}
\label{section:introduction}

\subsection{Background}
One of the most significant breakthroughs in the theory of Galois representations came in 1972, when Serre proved the Open Image Theorem for elliptic curves in his seminal paper~\cite{causalrelationship}.
Serre's theorem states that for any elliptic curve $E$ over a number field $K$ without complex multiplication, the image of the associated \emph{adelic} Galois representation $\rho_E$ is an open subgroup of the general symplectic group $\GSp_2(\wh{\ZZ})$.\footnote{Recall that $\GSp_2(\wh{\ZZ})= \mathrm{GL}_2(\wh\ZZ)$; here, we prefer to use the less common symplectic notation so as to highlight the analogy between the elliptic curve case and that of higher dimensional abelian varieties.} The Open Image Theorem not only gives rise to many important \mbox{corollaries --} from the simple consequence that the image of $\rho_E$ has finite index in $\GSp_2(\wh{\ZZ})$, to the intriguing result that the density of supersingular primes of $E$ is $0$ -- but recently, within the past two decades, the theorem has also inspired a body of research concerning the following question:

\begin{question}
How large can the image of the adelic Galois representation associated to an elliptic curve be, and how often do elliptic curves attain this largest possible Galois image?
\end{question}

The first major result addressing the above question was achieved by Duke in~\cite{duke:elliptic-curves-with-no-exceptional-primes}. He proved that for  ``most'' elliptic curves $E$ over $\QQ$ in the standard family with Weierstrass equation $y^2 = x^3 + ax + b$, the image of the \emph{mod-$\ell$ reduction} of $\rho_E$ is all of $\GSp_2(\ZZ/\ell \ZZ)$ for every prime number $\ell$;
here and in what follows, ``most'' means a density-$1$ subset of curves ordered by na\"{i}ve height.
Duke's result does not imply, however, that $\rho_E$ surjects onto $\GSp_2(\wh{\ZZ})$ for most $E$. In fact, as
Serre observes in~\cite{causalrelationship}, the image of $\rho_E$ has index divisible by $2$ in $\GSp_2(\wh{\ZZ})$ for every elliptic curve $E/\QQ$. Nonetheless, Jones proves in~\cite[Theorem 4]{josofabank} that most elliptic curves $E$ in the standard family over $\QQ$ have \emph{adelic} Galois representations with image as large as possible (i.e., with index $2$ in $\GSp_2(\wh{\ZZ})$).

The obstruction to having surjective adelic Galois representation faced by elliptic curves over $\QQ$ does not occur over other number fields. In~\cite[Theorem 1.5]{greasy}, Greicius constructed the first explicit example of an elliptic curve over a number field with Galois image equal to all of $\GSp_2(\wh{\ZZ})$. Greicius' example is not the only elliptic curve with this property: in~\cite[Theorem 1.2]{zywina2010elliptic}, Zywina employs the above result of Jones to show that most elliptic curves in the standard family over a number field $K \neq \QQ$ have Galois image equal to all of $\GSp_2(\wh{\ZZ})$ as long as $K \cap \QQ^{\cyc} = \QQ$, where $\QQ^{\cyc}$ is the maximal cyclotomic extension of $\QQ$. Subsequently, in~\cite[Theorem 1.15]{zywina2010hilbert}, Zywina achieves an intriguing generalization of this result: using a variant of Hilbert's Irreducibility Theorem, he shows that most members of \emph{every} non-isotrivial rational family of elliptic curves over \emph{any} number field have Galois image as large as possible given the constraints imposed by the arithmetic and geometric properties of the family. Further results over $\mathbb Q$ were
obtained in~\cite{grant:a-formula-for-the-number-of-elliptic-curves-with-exceptional-primes},~\cite{cojocaruH:uniform-results-for-serres-theorem-for-elliptic-curves}, and~\cite{cojocaruGJ:one-parameter-families-of-elliptic-curves}
(see~\cite[p.~6]{zywina2010hilbert} for a more detailed overview).

Given that the above question is so well-studied in the context of elliptic curves, it is natural to ask whether any of the aforementioned theorems extend to abelian varieties of higher dimension. As it happens, explicit examples of curves whose Jacobians have maximal Galois image have been constructed: it follows from the results of~\cite{dooleyfat} and~\cite{seaweed} that one can algorithmically write down equations of abelian surfaces and three-folds over $\mathbb Q$ with Galois image as large as possible. 
 Moreover, there are several results showing that in a family of abelian varieties, ``most'' fibers lying over closed points of the base have Galois image with finite index in the Galois image
of the family.
For instance, in \cite{cadoret2015open}, (see also \cite{cadoret2015integral}) the author
shows that the set of fibers lying over $K$-points of the base for which the associated Galois image does \emph{not }have finite
index in that of the family is a thin set. Furthermore, in \cite{cadoretuniform-i} and \cite{cadoretuniform-ii},
the authors show that when the base of the family is a curve, the set of fibers lying over $K$-points of the base
(and more generally closed points of bounded degree)
for which the associated Galois image does \emph{not} have finite
index in that of the family is a finite set. However, we are not aware of any results in the literature describing the density of higher-dimensional abelian varieties whose adelic Galois representations 
\mbox{have maximal image} (as opposed to merely having finite index) in that of the family.

\subsection{Main Result}\label{weaintevergonnaberoyals}

The primary objective of this article is to prove that an analogue of Zywina's result for rational families of elliptic curves in~\cite[Theorem 1.15]{zywina2010hilbert} holds for abelian varieties of arbitrary dimension, subject to a mild hypothesis on the \emph{monodromy} (i.e., Galois image) of the family under consideration. Before stating our theorems, we must establish some of the requisite notation; we expatiate upon this and other important background material in Section~\ref{subsection:setup}, where precise definitions are provided.

Let $K$ be a number field with fixed algebraic closure $\ol{K}$, let $U \subset \mathbb P^r_K$ be a dense open subscheme, and let $A \rightarrow U$ be a family of $g$-dimensional principally polarized abelian varieties (henceforth, PPAVs). Let $\mono_A \subset \GSp_{2g}(\widehat{\mathbb Z})$ be the monodromy of the family
and let $\mono_{A_u} \subset \mono_A$ be the monodromy of the fiber $A_u$ over $u \in U$. Finally, to facilitate our enumeration of PPAVs, let $\on{Ht}\colon\mathbb P^r(\overline K) \rightarrow \mathbb R_{>0}$
denote the absolute multiplicative height on projective space,\footnote{See~\cite[Section B.2, p.~174]{afraidofheights} for the definition.}
and define a height function $\| - \|$ on the lattice $\mathcal O^r_K$
sending $\left( t_1, \ldots, t_r \right) \mapsto \max_{\sigma,i}|\sigma(t_i)|$,
where $\sigma$ varies over all field embeddings $\sigma\colon K \hookrightarrow \mathbb C$.
Our main result is stated as follows:
\begin{theorem} \label{theorem:main}
	Let $B, n$ be arbitrary positive real numbers, and suppose that the rational family $A \to U$ is non-isotrivial and has big monodromy, meaning that $\mono_A$ is open in $\GSp_{2g}(\zh)$. Let $\delta_\QQ$ be the index of the closure of the commutator subgroup of $\mono_A$ in $\mono_A \cap \Sp_{2g}(\zh)$, and let $\delta_K = 1$ for $K \neq \QQ$. Then $[\mono_A : \mono_{A_u}] \geq \delta_K$ for all $u \in U(K)$, and we have the following asymptotic statements:
			\[
				\frac{|\{u \in U(K) \cap \mathcal{O}_K^r : \lVert u \rVert \le B,\, [\mono_A : \mono_{A_u}] = \delta_K\}|}{|\{u \in U(K) \cap \mc{O}_K^r : \lVert u \rVert \le B\}|} = 1 + O((\log B)^{-n}), \text{ and}
			\]
\[
				\frac{|\{u \in U(K) : \on{Ht}(u) \leq B,\, [\mono_A : \mono_{A_u}] = \delta_K\}|}{|\{u \in U(K) : \on{Ht}(u) \le B\}|} = 1 + O((\log B)^{-n}),
			\]
	 where the implied constants depend only on $A \to U$ and $n$.
\end{theorem}

\begin{remark}\label{remkydoo}
	Notice that Theorem~\ref{theorem:main} holds trivially in dimension $0$. In~\cite[Theorem 1.15]{zywina2010hilbert}, where the $1$-dimensional case of Theorem~\ref{theorem:main} is treated, Zywina bounds the error more sharply, by
$O( (\log B)B^{-1/2} )$ as opposed to our bound of
$O( (\log B)^{-n} )$.
In what follows, we shall primarily restrict ourselves to the case where the dimension $g$ is at least $2$.
\end{remark}

\begin{remark}
In~\cite{scoopdedoo}, Wallace studies a variant of Theorem~\ref{theorem:main} in the $2$-dimensional case.~Unfortunately, his argument relies upon a mistaken Masser-W\"{u}stholz-type result of Kawamura,~\cite[Main Theorem 2]{ifyouseekamy}. Although Wallace describes in~\cite[p.~468]{scoopdedoo} how to correct some of the errors in Kawamura's proof, the modified argument still appears to be mistaken; see~\cite[p.~27]{lombardoGL2type} for a description of one error in Kawamura's argument that Wallace does not adequately address.
Using the result stated in Appendix~\ref{lombardstreet}, written by Davide Lombardo, we are able to patch this error in Wallace's argument.
\end{remark}

\begin{remark}
	\label{remark:}
    The locus of $u \in U(K)$ with $[\mono_A : \mono_{A_u}] > \delta_K$ will not in general be Zariski-closed, so the ``sparseness'' of this locus can only be quantified by an asymptotic statement. To see why, consider the family of
	elliptic curves over $K$
	given by \mbox{the Weierstrass equations} $y^2 = x^3 + x + a$ for $a \in K$. Note that the mod-$2$ reduction of the monodromy is nontrivial for the family but is trivial for infinitely many members of the family, namely those for which the defining polynomial $x^3 + x + a$ factors completely over $K$.
\end{remark}

We now outline the proof of Theorem~\ref{theorem:main}. Hilbert's Irreducibility Theorem is the prototype for results like Theorem~\ref{theorem:main}, but it only applies in the setting of finite groups. Indeed, the phenomenon that Galois representations associated to elliptic curves over $\bq$ \emph{never} surject onto $\GSp_2(\wh{\ZZ})$ shows that Hilbert's Irreducibility Theorem cannot hold for infinite groups. However, when $A \to U$ has \emph{big monodromy}, in the sense that $H_A$ is open in $\GSp_{2g}(\zh)$, the problem is essentially reduced to showing that, for most $u \in U(K)$, the mod-$\ell$ reduction of $H_{A_u}$ contains $\GSp_{2g}(\bz / \ell \bz)$ for each sufficiently large prime $\ell$. This reduction uses an infinite version of Goursat's lemma. Since these mod-$\ell$ reductions are \emph{finite} groups, the na\"{i}ve expectation is that Hilbert's Irreducibility Theorem can be applied once for each $\ell$. Unfortunately, the sum of the resulting error terms does not \emph{a priori} converge to zero.

To overcome this problem, we divide the primes $\ell$ into three regions.
\begin{enumerate}
	\item We handle all sufficiently large primes by means of a delicate argument involving the large sieve that allows us to apply a recent result of Lombardo (namely,~\cite[Theorem 1.2]{lombardo2015explicit} and  Proposition~\ref{prop:Main}).
	\item For the smaller primes, Wallace's effective version of the Hilbert Irreducibility Theorem gives sufficiently good error terms. His approach is to complete $\phi: U \to \spec K$ to a map $\wt{\phi} : \mc{U} \to \spec \mc{O}_K$ (see Section~\ref{subsection:notation-for-families}), and then to apply the large sieve using information gleaned from the special fibers of $\wt{\phi}$. To ensure that the monodromy maps associated to special fibers of $\wt{\phi}$ capture enough information about the monodromy of the whole family, we assume the family is non-isotrivial and has big monodromy.
Our main contribution to this step is an application of the Grothendieck Specialization Theorem, which shows that Wallace's Property~\ref{property-a2}---concerning the relation between the monodromy maps associated to a geometric \emph{special} fiber and to a geometric \emph{generic} fiber---holds in a very general setting.
	\item Lastly, to handle the finitely many primes that remain, the Cohen-Serre version of the Hilbert Irreducibility Theorem suffices.
\end{enumerate}
We encourage the reader to refer to Section~\ref{subsection:outline} for a more detailed discussion of the intricate arguments outlined above.

\begin{remark}
Note that the proof strategy outlined above
is greatly influenced by
the methods that Zywina employed in~\cite{zywina2010hilbert}
to
handle the case where $g=1$ and also by unpublished work of Zureick-Brown and Zywina.
In particular, the idea
of formulating the problem in terms of monodromy groups and solving
it by
applying effective versions of Hilbert's Irreducibility Theorem and
Serre's
Open Image Theorem is largely due to them.

Zureick-Brown and Zywina were the first to state a
version of Theorem~\ref{theorem:main}.
Indeed, in a 2013 talk at the Institute for
Advanced Study, Zywina announced that he and Zureick-Brown had
proven
a result very much like Theorem~\ref{theorem:main}
using a strategy similar to that outlined above.
Following this talk,
Deligne suggested a potential way to strengthen the result by
removing the hypothesis that the family has big monodromy, and it is our understanding that Zywina
has been attempting to remove this hypothesis by following Deligne's
suggestion and that his work is still in progress. As the details of the work of Zureick-Brown and Zywina are not available, we have worked out a modified approach that utilizes recent results of
Wallace (\cite{scoopdedoo}) and Lombardo  (\cite{lombardo2015explicit}) that
had not been published at the time of Zywina's talk. In light of the above, we would like to extend a special acknowledgment to Zureick-Brown and Zywina for formulating the questions that motivated our
work
and
for introducing the ideas that inspired our proof of
Theorem~\ref{theorem:main}.
\end{remark}

\subsection{Applications} We record a number of interesting applications of our main result. These and several further applications are stated and proven in
Theorem~\ref{corollary:examples}.
\begin{theorem}[Abbreviation of Theorem~\protect{\ref{corollary:examples}}]
	\label{corollary:abbreviated-examples}
	Let $\ag$ denote the moduli stack of $g$-dimensional PPAVs, suppose $A \to U$ is a rational family, and let $V$ be
	the smallest locally closed substack of $\ag$ through which $U \rightarrow \ag$ factors.
	The conclusion of Theorem~\ref{theorem:main} holds if $V$ is normal and contains a dense open substack of any of the following loci:
	\begin{enumerate}
		\item the substack of Jacobians of hyperelliptic curves, or
		\item the substack of Jacobians of trigonal curves, or
		\item the substack of Jacobians of plane curves of degree $d$ (see Remark~\ref{remark:locus-of-plane-curves} for a more precise description of this substack), or
		\item the substack of Jacobians of all curves in $\mg$, or
		\item the moduli stack $\ag$.
	\end{enumerate}
\end{theorem}
Theorem~\ref{corollary:abbreviated-examples} has the following noteworthy corollary:

\begin{corollary}
	\label{corollary:infinitely-many-maximal-image}
	For every $g > 2$, there exist infinitely many PPAVs $A$ over $\mathbb Q$
	with the property that $\rho_A(G_\mathbb Q) = \GSp_{2g}(\widehat{\mathbb Z})$.
\end{corollary}
\begin{proof}
Let $\trigonal g(g\bmod 2) \subset \ag$ denote the locus of trigonal curves over $\mathbb Q$ of lowest Maroni invariant
(as defined at the beginning of Section~\ref{androidbeatsios}).
We have that $\trigonal g(g\bmod 2)$ is rational and normal when $g > 2$ (by
Theorem~\ref{corollary:examples}~\ref{big-maroni}) and has monodromy equal to all of $\GSp_{2g}(\widehat{\mathbb Z})$
when $g > 2$
(by Remark~\ref{remark:mmmm}).
Since $\trigonal g(g\bmod 2)$ is a dense open substack of the
locus Jacobians of trigonal curves, Theorem~\ref{corollary:abbreviated-examples}
implies that \mbox{Theorem~\ref{theorem:main}
applies to $\trigonal g(g\bmod 2)$.\qedhere}
\end{proof}

\begin{remark}
The above proof of Corollary~\ref{corollary:infinitely-many-maximal-image} is not constructive. For explicit examples of $1$-, $2$-, and $3$-dimensional PPAVs with maximal adelic Galois representations, see~\cite[Theorem 1.5]{greasy} and ~\cite[Sections 5.5.6-8]{causalrelationship},~\cite{landesman-swaminathan-tao-xu:hyperelliptic-curves}, \mbox{and~\cite[Theorem 1.1]{seaweed}, respectively.}
\end{remark}

We conclude this section with a representative example, which has incidentally enjoyed
significant discussion in the literature.

\begin{example}
	\label{example:}
	In this example, we take our family to be the Hilbert scheme $\mathscr H_4$ of plane
	curves of degree $4$ over $\mathbb Q$. There is quite a bit of earlier work concerning Galois representations associated to Jacobians of such curves. For instance, a single example of a plane quartic
	such that the adelic Galois representation associated to its Jacobian has image
	equal to $\GSp_6(\widehat{\mathbb Z})$
	is given in \cite[Theorem 1.1]{seaweed}.
	In \cite[Corollary 1.1]{anni2016residual},
	an example of a genus-$3$ hyperelliptic curve
	whose Jacobian has mod-$\ell$ monodromy equal to $\GSp_6(\mathbb Z/\ell \mathbb Z)$
	for primes $\ell \geq 3$ is constructed.
	For any $\ell \geq 13$,~\cite[Theorem 0.1]{arias2015large}
	gives
	an infinite family of $3$-dimensional PPAVs with mod-$\ell$
	monodromy equal to $\GSp_6(\mathbb Z/\ell \mathbb Z)$. All of these existence statements are subsumed by the main results of the present article: indeed, from Remark~\ref{remark:mmmm} and
	Theorem~\ref{corollary:abbreviated-examples},
	we obtain the considerably stronger statement that a density-$1$ subset of this family
	has Galois representation
	with image equal to $\GSp_6(\widehat{\mathbb Z})$.
\end{example}

The rest of this paper is organized as follows. In Section~\ref{gorilla}, we define the symplectic group and prove properties concerning its open and closed subgroups. In Section~\ref{section:background}, we introduce the basic definitions and properties associated to Galois representations of abelian varieties and families thereof. These definitions and properties are used heavily in Section~\ref{section:new-proof-of-main-theorem}, which is devoted to proving the main theorem of this article, Theorem~\ref{theorem:main}. In Section~\ref{iknewyouweretrouble}, we show that Theorem~\ref{theorem:main} can be applied to study many interesting families of PPAVs, and in so doing, we prove a result that implies Theorem~\ref{corollary:abbreviated-examples}.
Finally, in Appendix~\ref{lombardstreet}, Davide Lombardo proves a key input that we employ in Section~\ref{section:new-proof-of-main-theorem} to handle the genus-$2$ case of Theorem~\ref{theorem:main}.

\section{Definitions and Properties of Symplectic Groups} \label{gorilla}

In this section, we first detail the basic definitions and properties of symplectic groups, and we then proceed to prove a few group-theoretic lemmas that are used in our proof of the main result of this paper, Theorem~\ref{theorem:main}. The reader should feel free to proceed to Section~\ref{section:background} upon reading the statements of Propositions~\ref{theorem:adelic-surjective-subset} and~\ref{theorem:commutator-open}.

\subsection{Symplectic Groups}\label{subsection:stimpy}

Fix a commutative ring $R$, a free $R$-module $M$ of rank $2g$ for some positive integer $g$, and a non-degenerate alternating bilinear form $\langle -, - \rangle \colon M \times M \to R$. Define the {\it general symplectic group} (alternatively, the \emph{group of symplectic similitudes}) $\GSp(M)$ to be the subgroup of $\on{GL}(M)$ consisting of all $R$-automorphisms $S$ such that there exists some $m_S \in R^\times$, called the {\it multiplier} of $S$, satisfying $\langle S v, Sw \rangle = m_S \cdot \langle v, w \rangle$ for all $v, w \in M$. One readily observes that the {\it mult} map
\begin{align*}
	\mult \colon \GSp(M) & \rightarrow R^\times \\
	S & \mapsto m_S
\end{align*}
is a group homomorphism, and its kernel is the {\it symplectic group}, denoted by $\Sp(M)$.

By choosing a suitable $R$-basis for $M$, we can arrange for the corresponding matrix of the inner product $\langle -, - \rangle$ to be given by
$$\Omega_{2g} = \left[\begin{array}{c|c} 0 & \id_g \\ \hline -\id_g & 0\end{array}\right],$$
where $\id_g$ denotes the $g \times g$ identity matrix. From this choice of basis we obtain an identification $\GL(M) \simeq \GL_{2g}(R)$. We then define $\GSp_{2g}(R)$ to be the image of $\GSp(M)$ and $\Sp_{2g}(R)$ to be the image of $\Sp(M)$ under this identification. Let $\det \colon \GL_{2g}(R) \to R^\times$ be the determinant map. Since the diagram
\begin{center}
\begin{tikzcd}
\GSp(M) \arrow{r}{\sim} \arrow[swap]{rd}{\on{mult}^g} &  \GSp_{2g}(R) \arrow{d}{\on{det}} \\
& R^\times
\end{tikzcd}
\end{center}

\noindent commutes, where the diagonal map is the multiplier map raised to the $g^{\mathrm{th}}$ power, one deduces that $\GSp_{2g}(R)$ is in fact the subgroup of $\GL_{2g}(R)$ consisting of all invertible matrices $S$ satisfying $S^T \Omega_{2g} S = (\on{mult} S) \, \Omega_{2g}$ and that $\Sp_{2g}(R) = \ker(\on{mult} \colon \GSp_{2g}(R) \to R^\times)$.

Let $\on{Mat}_{2g \times 2g}(R)$ denote the space of $2g \times 2g$ matrices with entries in $R$. In subsequent subsections, we will make heavy use of the ``Lie algebra'' $\mf{sp}_{2g}(R)$, which is defined by
		\begin{align*}
			\mf{sp}_{2g}(R) &\defeq \{M \in \on{Mat}_{2g \times 2g}(R) : M^T \Omega_{2g} + \Omega_{2g} M = 0 \}.
		\end{align*}
		It is easy to see that $M^T \Omega_{2g} + \Omega_{2g}M = 0$ is equivalent to $M$ being a block matrix with $g \times g$ blocks of the form
		\[
			M = \left[\begin{array}{c|c} A & B \\ \hline C & -A^T \end{array}\right],
		\]
		where $B$ and $C$ are symmetric.

For the purpose of studying Galois representations associated to PPAVs, we will be primarily interested in the cases where the ring $R$ is the profinite completion $\wh{\ZZ}$ of $\ZZ$, the ring of $\ell$-adic integers $\ZZ_{\ell}$ for a prime number $\ell$, or the finite cyclic ring $\ZZ / m \ZZ$ for a positive integer $m$.
Note in particular that we have the identifications
\begin{equation}\label{orientation1}
\GSp_{2g}(\ZZ_\ell)   \simeq  \varprojlim_k \GSp_{2g}(\ZZ/ \ell^k \ZZ) \quad \text{and}
\end{equation}
\begin{equation}\label{orientation2}
\prod_{\text{prime } \ell} \GSp_{2g}(\ZZ_\ell) \simeq  \GSp_{2g}(\wh{\ZZ}) \simeq  \varprojlim_m \GSp_{2g}(\ZZ / m \ZZ).
\end{equation}
From~\eqref{orientation1} and~\eqref{orientation2}, we obtain the $\ell$-adic projection map $\pi_\ell \colon \GSp_{2g}(\wh{\ZZ}) \twoheadrightarrow \GSp_{2g}(\ZZ_\ell)$ and the mod-$m$ reduction map $r_m \colon \GSp_{2g}(\wh{\ZZ}) \twoheadrightarrow \GSp_{2g}(\ZZ / m \ZZ)$. Observe that~\eqref{orientation1} and~\eqref{orientation2} both hold with $\GSp_{2g}$ replaced by $\Sp_{2g}$.

\subsection{Notation}
\label{subsection:notation}

In what follows, we study subquotients of $\Sp_{2g}(\wh{\ZZ})$, $\Sp_{2g}(\ZZ_\ell)$, and $\Sp_{2g}(\ZZ/ \ell^k \ZZ)$ for $\ell$ a prime number and $k$ a positive integer. We use the following notational conventions:
\begin{itemize}
\item Let $H\subset \Sp_{2g}(\wh{\ZZ})$ be a closed subgroup.
\item Let $H_{\ell} \defeq \pi_\ell(H) \subset \Sp_{2g}(\ZZ_\ell)$ be the $\ell$-adic reduction of $H$. More generally, for any set $S$ of prime numbers, let $H_S$ denote the projection of $H$ onto $\prod_{\ell \in S} \Sp_{2g}(\wh{\ZZ})$.
\item Let $H(m) = r_{m}(H) \subset \Sp_{2g}(\ZZ/m \ZZ)$ be the mod-$m$ reduction of $H$. We often take $m = \ell^k$.
\item Let $\Gamma_{\ell^k} = \ker(\Sp_{2g}(\ZZ_\ell) \to \Sp_{2g}(\ZZ/ \ell^k \ZZ))$. Notice that the map $M \mapsto \id_{2g} + \ell^k M$ gives an isomorphism of groups
$$\mf{sp}_{2g}(\ZZ/ \ell \ZZ) \simeq \ker(\Sp_{2g}(\ZZ/\ell^{k+1} \ZZ) \to \Sp_{2g}(\ZZ/ \ell^k \ZZ))$$
for every $k \geq 1$, so we will use $\mf{sp}_{2g}(\mathbb Z/\ell\ZZ)$ to denote the above kernel.
\item For any group $G$, let $[G,G]$ be its commutator subgroup, and let $G^{\ab} \defeq G/{ {[G,G]}}$ be its abelianization.
\item For any group $G$, let $\quo(G)$ the set of isomorphism classes of finite non-abelian simple quotients of $G$, and let $\occ(G)$ be the set of isomorphism classes of finite non-abelian simple \emph{sub}quotients of $G$.
\item For any positive integer $m$, let $S_m$ denote the symmetric group on $m$ letters.
\end{itemize}

\subsection{Generalizing Goursat's Lemma}

In Sections~\ref{subsection:closed-subgroups} and~\ref{subsection:open-subgroups}, it will be crucial for us to have a theorem that allows us to express a subgroup of $\Sp_{2g}(\wh{\ZZ})$ as (roughly) the product of its $\ell$-adic projections. A natural tool for doing this is Goursat's lemma, but in much of the literature (e.g.,~\cite[Lemma 5.2.1]{ribbit} and~\cite[Lemma A.4]{zywina2010elliptic}), this result is stated for \emph{finite} products or for \emph{finite} groups. This section is devoted to proving Lemma~\ref{theorem:goursat}, which generalizes Goursat's lemma to apply in the setting that we need, namely for \emph{countable} products of \emph{profinite} groups.
	
	\begin{lemma} \label{lemma:product-quotient}
		Let $G = \prod_{i=1}^n G_i$ be a product of profinite groups. Then every finite simple quotient of $G$ is a finite simple quotient of $G_i$ for some $i$, and vice versa.
	\end{lemma}
	\begin{proof}
		Consider a finite simple quotient $\phi: G \twoheadrightarrow H$. Since each $G_i \subset G$ is normal, the image $\phi(G_i) \subset H$ is also normal. For any $i$, if $\phi(G_i)$ is larger than $\{1\}$, then it equals $H$ since $H$ is simple, and the composition
		\(
			G_i \hookrightarrow G \twoheadrightarrow H
		\)
		expresses $H$ as a quotient of $G_i$. If no such $i$ exists, then $\ker \phi = G$, contradiction. The ``vice versa'' statement is obvious.
	\end{proof}
	
	\begin{lemma}[Generalized Goursat's Lemma] \label{theorem:goursat}
		Let $A$ be a countable set, and suppose $\{G_\alpha\}_{\alpha \in A}$ is a collection of profinite groups such that, for all pairs $\alpha, \beta \in A$ with $\alpha \neq \beta$, the groups $G_\alpha$ and $G_\beta$ have no finite simple quotients in common. Let $G := \prod_{\alpha \in A} G_\alpha$, and let $\pi_\alpha : G \to G_\alpha$ be the natural projections. If $H \subset G$ is a closed subgroup with $\pi_\alpha(H) = G_\alpha$ for all $\alpha \in A$, then $H = G$.
	\end{lemma}
	\begin{proof}
		First take $A = \{1, 2\}$, so that $G = G_1 \times G_2$. The subgroup
		\(
			N_1 \times \{1\} \defeq (G_1 \times \{1\}) \cap H \subset G
		\)
		is normal because $\pi_1(H) = G_1$. This means $N_1$ is a normal subgroup of $G_1$. Similarly for the subgroup $\{1\} \times N_2$. With these definitions, the closed subgroup
		\(
			H / (N_1 \times N_2) \subset (G_1 / N_1) \times (G_2 / N_2)
		\)
		surjects onto each factor via the natural projections. We have thereby reduced to the case $N_1 = N_2 = 0$. By \cite[Lemma 5.2.1]{ribbit}, we know that $G_1 \simeq G_2$ as profinite groups. The result follows because two isomorphic profinite groups have a nontrivial finite simple quotient in common (and any quotient of $G_i / N_i$ is \emph{a priori} a quotient of $G_i$).
		
		Now take $A = \{1, 2, \ldots, n\}$ for $n \ge 3$, and suppose (by induction) that the result has been proven for $n-1$. For any $H \subset G = \prod_{i=1}^n G_i$ satisfying the hypotheses of the theorem, let $H'$ be the image of $H$ under the projection $G \twoheadrightarrow \prod_{i=1}^{n-1} G_i$. Then $H'$ satisfies the hypotheses for $n-1$, so we conclude that $H' = \prod_{i=1}^{n-1} G_i$. By Lemma~\ref{lemma:product-quotient}, the groups $\prod_{i=1}^{n-1} G_i$ and $G_n$ have no finite simple quotients in common, so the $n = 2$ case tells us that $H = G$.
		
		The only remaining case is $A = \{1, 2, \ldots\}$. Consider $H \subset G$ satisfying the hypotheses of the theorem. For each $n$, let $H_{\{1,2,\dots,n\}}$ be the image of $H$ under the projection $G \twoheadrightarrow \prod_{i=1}^{n} G_i$. By the finite case prove above, we know that $H_{\{1,2,\dots,n\}} = \prod_{i=1}^n G_i$ for each $n \ge 1$. Fix an element $g \defeq (g_i)_{i \ge 1} \subset G$, and define a sequence $\{h_1, h_2, \ldots\}$ of elements of $H$ as follows: let $h_n$ be any element of $H$ whose image in $\prod_{i=1}^n G_i$ equals $(g_1, \ldots, g_n)$. In the product topology, $h_n \to g$ as $n \to \infty$, so $g \in H$ since $H$ is closed. Since $g \in G$ was arbitrary, \mbox{we conclude that $H = G$.}
	\end{proof}

\subsection{Closed Subgroups of $\Sp_{2g}(\wh{\ZZ})$}
\label{subsection:closed-subgroups}
	
	As before, let $H \subset \Sp_{2g}(\zh)$ be a closed subgroup. The main result of this section is Proposition~\ref{theorem:adelic-surjective-subset}, which shows that properties of $H$ can be deduced from corresponding properties of the $\ell$-adic projections $H_\ell \subset \Sp_{2g}(\bz_\ell)$ as $\ell$ ranges over the prime numbers.
	We use Proposition~\ref{theorem:adelic-surjective-subset} crucially in our proof of the main theorem, Theorem~\ref{theorem:main},
	and more specifically in the proof of Proposition~\ref{lemma:ab-cyc}.


	The next lemma enables us to verify the conditions required for applying Lemma~\ref{theorem:goursat}:
	\begin{lemma}
		\label{lemma:simple-quotients-of-symplectic-group}
       If $g > 2$ or $\ell > 2$, we have $\quo(\Sp_{2g}(\ZZ_\ell)) = \{\on{PSp}_{2g}(\ZZ/ \ell \ZZ)\}$. Moreover, for all $g \geq 2$, we have $\quo(\Sp_{2g}(\mathbb Z_\ell)) \cap \quo(\Sp_{2g}(\mathbb Z_{\ell'})) = \varnothing$ if $\ell \neq \ell'$.
	\end{lemma}
	\begin{proof}
		Since $\Gamma_\ell$ is a pro-$\ell$ group, we have that
		$\quo (\Sp_{2g}(\mathbb Z_\ell)) = \quo(\Sp_{2g}(\mathbb Z/ \ell \ZZ))$.
		Furthermore, quotienting by $\{\pm \id_{2g}\}$, we have that
		$\quo(\Sp_{2g}(\mathbb Z/ \ell \ZZ)) = \quo(\Sp_{2g}(\mathbb Z/ \ell \ZZ)/\left\{ \pm \id_{2g} \right\})$. By~\cite[Theorem 3.4.1]{omeara1978symplectic}, we have $\Sp_{2g}(\mathbb Z/\ell \ZZ)/\left\{ \pm \id_{2g} \right\} = \PSp_{2g}(\ZZ/ \ell \ZZ)$ is simple for $g > 2$ or $\ell > 2$. It follows that $\quo(\Sp_{2g}(\ZZ_\ell)) = \{\on{PSp}_{2g}(\ZZ/ \ell \ZZ)\}$ in this case.

To finish the proof, note that $\quo(\Sp_{2g}(\mathbb Z_\ell)) \cap \quo(\Sp_{2g}(\mathbb Z_{\ell'})) = \varnothing$ for $g > 2$ or $\ell, \ell' > 2$ because $\PSp_{2g}(\bz / \ell \bz) \neq \PSp_{2g}(\bz / \ell' \bz)$ for $\ell \neq \ell'$ because their orders are different. The only remaining case is where $g = 2$, $\ell = 2$, and $\ell' > 2$. In this case, observe that $\PSp_{2g}(\bz / \ell' \ZZ) \notin \quo( \Sp_{2g}(\bz / 2 \ZZ))$ for $\ell' > 2$, since the order of $\PSp_{2g}(\bz / \ell \ZZ)$ exceeds that of $\Sp_{2g}(\bz / 2 \ZZ)$.
	\end{proof}
We next prove 
Proposition~\ref{theorem:truncate},
which we then use to deduce the main result of this section, Proposition~\ref{theorem:adelic-surjective-subset}.

		\begin{proposition}
		\label{theorem:truncate}
		Let $g \geq 2$ and let $H \subset \Sp_{2g}(\widehat {\mathbb Z})$ be a closed subgroup. Suppose there is a prime number $p \ge 2$ so that $H(\ell) = \Sp_{2g}(\ZZ/ \ell \ZZ)$ for all $\ell > p$. Then we have that
		\begin{equation}\label{meanttofly}
		H = H_{\{\ell \le p\}} \times \prod_{\ell > p} \Sp_{2g}(\ZZ_\ell).
		\end{equation}
	\end{proposition}
\subsubsection*{Idea of Proof}
The idea of the proof is to apply Lemma~\ref{theorem:goursat} to conclude that if the group surjects onto each factor,
then it surjects onto the product.
We verify the hypotheses of Lemma~\ref{theorem:goursat}
using Lemma~\ref{lemma:simple-quotients-of-symplectic-group} and the fact that all simple quotients of $H_{\{\ell \leq p\}}$
have smaller order than $\PSp_{2g}(\mathbb Z_\ell)$ for $\ell > p$.
	\begin{proof}
		The case where $g = 1$ is handled by~\cite[Lemma 7.6]{zywina2010hilbert}, so take $g \geq 2$. By~\cite[Theorem 1]{landesman-swaminathan-tao-xu:lifting-symplectic-group}, the fact that $H(\ell) = \Sp_{2g}(\ZZ/ \ell \ZZ)$ implies that $H_\ell = \Sp_{2g}(\ZZ_\ell)$ for all $\ell > p$.

  The proposition follows upon applying Lemma~\ref{theorem:goursat} to the product $H_{\{\ell \le p\}} \times \prod_{\ell > p} \Sp_{2g}(\ZZ_\ell)$. However, to apply it, we must check that no two of the groups $H_{\{\ell \leq p\}}$ and $\Sp_{2g}(\ZZ_\ell)$ for $\ell > p$ have any finite simple quotients in common. From~\cite[Proposition 1, part (a)]{landesman-swaminathan-tao-xu:lifting-symplectic-group}, we have that the group $\Sp_{2g}(\ZZ_\ell)$ has trivial abelianization for $\ell > 2$ and thus has no finite abelian simple quotients. Thus, it remains to verify that the sets of nonabelian simple quotients $\quo(H_{\{\ell \le p\}})$ and $\quo(\Sp_{2g}(\ZZ_\ell))$ for $\ell > p$ are all pairwise disjoint. 
Our strategy for checking this condition is to bound the sizes of the groups appearing in $\quo(H_{\{\ell \le p\}})$. First, observe that
		\[
			\quo(H_{\{\ell \le p\}}) \subset \occ\left( \prod_{\ell \le p} \Sp_{2g}(\bz_\ell)\right) = \bigcup_{\ell \le p} \occ(\Sp_{2g}(\bz_\ell)),
		\]
		where the last step follows from the first displayed equation of~\cite[p.\ IV-25]{serre1989abelian}.
But
\(
			\occ(\Sp_{2g}(\bz_\ell)) = \occ(\Gamma_\ell) \cup \occ(\Sp_{2g}(\bz / \ell \ZZ))
		\),
and $\occ(\Gamma_\ell) = \varnothing$ because $\Gamma_\ell$ is a pro-$\ell$ group, so
		\(
			\occ(\Sp_{2g}(\bz_\ell)) = \occ(\Sp_{2g}(\bz / \ell \ZZ))
		\).
		Because $\Sp_{2g}(\bz / \ell \ZZ)$ is not simple, every element of $\occ(\Sp_{2g}(\bz / \ell \ZZ))$ is bounded in size by $|\Sp_{2g}(\bz / \ell \ZZ)| / 2$, so every element of $\quo(H_{\{\ell \le p\}})$ is bounded in size by $|\Sp_{2g}(\bz / p \ZZ)| / 2$. Observing that
		\[
			\frac{1}{2} \cdot |\Sp_{2g}(\bz / p \bz)| < |\PSp_{2g}(\bz / \ell \bz)|
		\]
		for every $\ell > p$, the desired condition follows by applying Lemma~\ref{lemma:simple-quotients-of-symplectic-group}.
	\end{proof}

\begin{proposition}
		\label{theorem:adelic-surjective-subset}
		Let $G \subset \Sp_{2g}(\zh)$ be an open subgroup. There exists a positive integer $M$ such that, for every closed subgroup $H \subset G$, we have $H = G$ if and only if $H(M) = G(M)$ and $H(\ell) = \Sp_{2g}(\bz / \ell \ZZ)$ for every prime $\ell \nmid M$.
	\end{proposition}
\subsubsection*{Idea of Proof}
The idea of the proof is to find a sufficiently large $M$
so that if $H(M) = G(M)$ then $H_{\{\ell \hspace{.05cm}\nmid \hspace{.05cm}M\}} = G_{\{\ell \hspace{.05cm} \nmid\hspace{.05cm} M\}}$,
which reduces the problem to the situation of
Proposition~\ref{theorem:truncate}.

	\begin{proof}
		Again, the case where $g = 1$ is handled in~\cite[Lemma 7.6]{zywina2010hilbert}, so take $g \geq 2$. Let $p$ be any prime such that $G(\ell) = \Sp_{2g}(\bz / \ell \ZZ)$ for all primes $\ell > p$.
		Observe that the groups $\Gamma_{\ell^k}$ are open in $\Sp_{2g}(\ZZ_\ell)$ because they have finite index in $\Sp_{2g}(\ZZ_\ell)$. Since $G \subset \Sp_{2g}(\zh)$ is open, the group $G_{\{\ell \le p\}} \subset \prod_{\ell \le p} \Sp_{2g}(\ZZ_\ell)$ is open too, so there exist exponents $e(\ell) \ge 1$ with the property that
		\[
			\prod_{\ell \le p} \Gamma_{\ell^{e(\ell)}} \subset G_{\{\ell \le p \}}.
		\]
		Since the groups $\Gamma_{\ell^k}$ are finitely generated pro-$\ell$ open normal subgroups of $\GSp_{2g}(\ZZ_\ell)$, condition (ii) from~\cite[Proposition 10.6]{serre1989lectures} is satisfied. Hence,
		the equivalence of conditions (ii) and (iv) from~\cite[Proposition 10.6]{serre1989lectures} implies
		that the Frattini subgroup defined by
		\begin{align*}
		\Phi(G_{\{\ell \le p\}}) \defeq \bigcap_{\substack{S \subset G_{\{\ell \le p\}} \\ S \text{ maximal closed in  }G_{\{\ell \le p\}}}} S
		\end{align*}
		is open and normal in $G_{\{\ell \le p\}}$. This means we can find exponents $e'(\ell) \ge 1$ such that
		\[
			\prod_{\ell \le p} \Gamma_{\ell^{e'(\ell)}} \subset \Phi(G_{\{\ell \le p\}}).
		\]
		Define $M \defeq \prod_{\ell \le p} \ell^{e'(\ell)}$. Then $H(M) = G(M)$ implies that $H_{\{\ell \leq p\}} = G_{\{\ell \leq p\}}$.
		
		Now take $H$ satisfying $H(M) = G(M)$ and $H(\ell) = \Sp_{2g}(\bz / \ell \ZZ)$ for every prime $\ell \nmid M$. We have that
		\[
			H \subset G \subset H_{\{\ell \le p\}} \times \prod_{\ell > p} \Sp_{2g}(\bz_\ell).
		\]
		To show that $H = G$, we need only verify
		\[
			H = H_{\{\ell \le p\}} \times \prod_{\ell > p} \Sp_{2g}(\bz_\ell),
		\]
		which follows immediately from Proposition~\ref{theorem:truncate}.
	\end{proof}
	
	\subsection{Open Subgroups of $\GSp_{2g}(\wh{\ZZ})$}
	\label{subsection:open-subgroups}
We now return to studying the general symplectic group $\GSp_{2g}(\wh{\ZZ})$. The main result of this subsection tells us that the closure of the commutator subgroup of an open subgroup of $\GSp_{2g}(\wh{\ZZ})$ is open:

\begin{proposition} \label{theorem:commutator-open}
Let $g \geq 2$, and let $H \subset \GSp_{2g}(\zh)$ be an open subgroup. Then the closure of $[H, H]$ is an open subgroup of $\Sp_{2g}(\zh)$.
\end{proposition}

In order to prove Proposition~\ref{theorem:commutator-open}, we shall require a number of preliminary lemmas, which are stated and proven in Sections~\ref{sec1} and~\ref{sec2}.

\subsubsection{Openness Condition}\label{sec1}

The next two lemmas give us a criterion for openness in $\Sp_{2g}(\zh)$:

\begin{lemma}
		\label{lemma:open-image-in-finite-set}
		Let $S$ be a finite set of prime numbers, and let $H \subset \prod_{\ell \in S} \Sp_{2g}(\bz_\ell)$ be a closed subgroup. If each $H_\ell \subset \Sp_{2g}(\bz_\ell)$ is open, then $H \subset \prod_{\ell \in S} \Sp_{2g}(\bz_\ell)$ is open.
	\end{lemma}
	\begin{proof}
		There exists a finite-index subgroup $H' \subset H$ such that $H'(\ell)$ is trivial for every $\ell \in S$, namely the intersection of the kernels of the mod-$\ell$ reductions maps $H \to H(\ell)$. Since each $H'_\ell$ is a pro-$\ell$ group, Lemma~\ref{theorem:goursat} implies that $H' = \prod_{\ell \in S} H'_\ell$.
		Thus, $H$ contains an open subgroup and is therefore itself open.
	\end{proof}

       \begin{lemma}
		\label{theorem:adelic-open}
		Let $g \geq 2$ and let $H \subset \Sp_{2g}(\zh)$ be a closed subgroup. If $H_{\ell'}$ is open in $\Sp_{2g}(\bz_{\ell'})$ for all $\ell'$ and $H_\ell = \Sp_{2g}(\bz_\ell)$ for all but finitely many $\ell$, then $H$ is open in $\Sp_{2g}(\zh)$.
	\end{lemma}
	\begin{proof}
		Let $p$ be the largest prime with $H_{p} \neq \Sp_{2g}(\mathbb Z_{p})$. By Lemma~\ref{lemma:open-image-in-finite-set}, we have that  $H_{\{\ell \le p\}} \subset \prod_{\ell \le p} \Sp_{2g}(\mathbb Z_\ell)$ is an open subgroup.
		The result then follows from Proposition~\ref{theorem:truncate}.
	\end{proof}

\subsubsection{Two Computational Lemmas}\label{sec2}

The next two results are used in the proof of Proposition~\ref{theorem:commutator-open}. The following lemma describes the commutator of an element of $\Gamma_{\ell^m}$ with an element of $\Gamma_{\ell^n}$.

	\begin{lemma} \label{lemma:commutator-formula}
		Let $n\le m$ be positive integers, and let $\id_{2g} + \ell^n U$ and $\id_{2g} + \ell^m V$ be elements of $\GL_{2g}(\bz_\ell)$. Then we have
		\begin{align*}
		& (\id_{2g} + \ell^nU)^{-1}(\id_{2g} + \ell^mV)(\id_{2g} + \ell^n U)(\id_{2g} + \ell^m V)^{-1} \equiv \id_{2g} + \ell^{n+m} (VU - UV)\,\,\, (\on{mod}{\ell^{2n+m}}).
		\end{align*}
	\end{lemma}
	\begin{proof}
		We have
		\begin{align*}
		(\id_{2g} + \ell^mV)(\id_{2g} + \ell^n U)(\id_{2g} + \ell^mV)^{-1} &= \id_{2g} + \ell^n (\id_{2g} + \ell^mV)U(\id_{2g} + \ell^mV)^{-1} \\			
		&= \id_{2g} + \ell^n (\id_{2g} + \ell^m V)U \left(\sum_{i=0}^\infty (-1)^i \ell^{im} V^i\right) \\
		&= \id_{2g} + \ell^n \sum_{i=0}^\infty \Big[ (-1)^i \ell^{im} U V^i + (-1)^i \ell^{(i+1)m} VUV^i \Big]\\
		&= \id_{2g} + \ell^n U + \ell^{n+m} (VU - UV) (\id_{2g} + \ell^m V)^{-1}.
		\end{align*}
		Multiplying on the left by $(\id_{2g} + \ell^nU)^{-1}$ gives the desired result.
	\end{proof}

In the next proposition, we show the commutator subalgebra of $\mf{sp}_{2g}(\ZZ/\ell \ZZ)$ is sufficiently large for all primes $\ell$.

    \begin{proposition}\label{stopdrop}
We have the following results:
\begin{enumerate}
\item For all $g \geq 1$ and $\ell \geq 3$ we have $[\mf{sp}_{2g}(\ZZ/ \ell \ZZ), \mf{sp}_{2g}(\ZZ/ \ell \ZZ)] = \mf{sp}_{2g}(\ZZ/ \ell \ZZ)$.
\item For all $g \geq 1$ we have $[\mf{sp}_{2g}(\ZZ/ 4 \ZZ), \mf{sp}_{2g}(\ZZ/ 4 \ZZ)] \supset 2 \cdot \mf{sp}_{2g}(\ZZ / 2 \ZZ)$.
\end{enumerate}
\end{proposition}
\begin{proof}
	Statement (a) follows immediately from~\cite[Theorem 2.6]{eliotsteinclub}, which states that $\mf{sp}_{2g}(\ZZ/ \ell \ZZ)$ is simple for $\ell \geq 3$. It remains to prove Statement (b). For this, we compute several commutators and make deductions based on each one. For convenience, let $\mf{g} = [\mf{sp}_{2g}(\ZZ/ 4 \ZZ), \mf{sp}_{2g}(\ZZ/ 4 \ZZ)]$, let $A, D$ denote arbitrary $g \times g$ matrices, and let $B,C,E,F$ denote symmetric $g \times g$ matrices. \mbox{Since}
\begin{align}
	\label{equation:block-diagonal-commutator}
 \left[ \left[\begin{array}{c|c} A & 0 \\ \hline 0 & -A^T \end{array}\right], \left[\begin{array}{c|c} D & 0 \\ \hline 0 & -D^T \end{array}\right] \right] & = \left[\begin{array}{c|c} AD - DA & 0 \\ \hline 0 & A^TD^T - D^TA^T \end{array}\right],
 \intertext{all block-diagonal matrices in $\mf{sp}_{2g}(\ZZ/4\ZZ)$ with every diagonal entry equal to $0$ are contained in $\mf{g}$. This can be seen
by taking $A$ and $D$ to be various elementary matrices.
Furthermore,}
\label{equation:block-off-diagonal-commutator}
 \left[ \left[\begin{array}{c|c} 0 & B \\ \hline C & 0 \end{array}\right], \left[\begin{array}{c|c} 0 & E \\ \hline F & 0 \end{array}\right] \right] & = \left[\begin{array}{c|c} BF - EC & 0 \\ \hline 0 & CE - FB \end{array}\right],
 \intertext{so we can arrange that $BF-EC$ is an elementary matrix with a single nonzero entry on the diagonal.
	 Summing matrices from~\eqref{equation:block-diagonal-commutator} and~\eqref{equation:block-off-diagonal-commutator} tells us that all block-diagonal matrices are contained in $\mf{g}$. Additionally,}
\label{equation:identity-commutator}
  \left[ \left[\begin{array}{c|c} \id_g & 0 \\ \hline 0 & -\id_g \end{array}\right], \left[\begin{array}{c|c} 0 & B \\ \hline 0 & 0 \end{array}\right] \right] & = \left[\begin{array}{c|c} 0 & 2B \\ \hline 0 & 0 \end{array}\right].
  \intertext{Repeating the computation from~\eqref{equation:identity-commutator} with the other off-diagonal block nonzero implies that $2$ times any matrix in $\mf{sp}_{2g}(\ZZ / 2 \ZZ)$ whose diagonal blocks are $0$ is an element of $\mf{g}$. The desired result follows because $2 \cdot \mf{sp}_{2g}(\mathbb Z/2\ZZ)$ is contained in the subspace generated by the matrices from~\eqref{equation:block-diagonal-commutator},~\eqref{equation:block-off-diagonal-commutator}, and~\eqref{equation:identity-commutator}. \nonumber \qedhere}
\end{align}
\end{proof}

\subsubsection{Completing the Proof}

In order to prove Proposition~\ref{theorem:commutator-open}, we require the following lemma, which states that the closure of the commutator $[\Gamma_{\ell^k}, \Gamma_{\ell^k}]$ is large.

\begin{lemma} \label{proposition:commutator-effective}
	Fix $k \geq 1$. Then if $\ell \neq 2$, the closure of $[\Gamma_{\ell^k}, \Gamma_{\ell^k}]$ contains $\Gamma_{\ell^{2k}}$ and if $\ell = 2$, the closure of $[\Gamma_{\ell^k}, \Gamma_{\ell^k}]$ contains $\Gamma_{\ell^{2k+1}}$.
\end{lemma}
\begin{proof}
First suppose $\ell \geq 3$. Statement (1) of Proposition~\ref{stopdrop} implies that
for any $W' \in \mf{sp}_{2g}(\ZZ/ \ell \ZZ)$, there exist $U', V' \in \mf{sp}_{2g}(\ZZ/ \ell \ZZ)$ so that
$V'U' - U'V' = W'$.
Choosing lifts $W, U, V$ of $W', U', V'$, it follows from Lemma~\ref{lemma:commutator-formula} that for every $i$ and for every such
\begin{align*}
\id_{2g} + \ell^{2k+i} W \in \Gamma_{\ell^{2k+i}}, \quad
\id_{2g} + \ell^k U \in \Gamma_{\ell^k}, \quad \text{ and} \quad
\id_{2g} + \ell^{k+i} V \in \Gamma_{\ell^{k+i}},
\end{align*}
we have that
	\begin{align*}
		& (\id_{2g} + \ell^k U)^{-1}(\id_{2g} + \ell^{k+i} V)(\id_{2g} + \ell^k U)(\id_{2g} + \ell^{k+i} V)^{-1} \equiv \id_{2g} + \ell^{2k+i} W\,\,\, (\on{mod}{\ell^{2k+i+1}}).
		\end{align*}
Take $M_0 \in \Gamma_{\ell^{2k}}$. There exists $X_1 \in [\Gamma_{\ell^{2k}}, \Gamma_{\ell^{2k}}]$ and $M_1 \in \Gamma_{\ell^{2k+1}}$ with the property that $M_0 = X_1M_1$. Proceeding inductively in this manner, we obtain sequences $\{X_i : i = 1, 2, \dots\} \subset [\Gamma_{\ell^k}, \Gamma_{\ell^k}]$ and $\{M_i : i = 0, 1, 2, \dots\}$ with $M_i \in \Gamma_{\ell^{2k+i}}$ such that $M_i = X_{i+1}M_{i+1}$ for each $i$. Then we have the following equalities of matrices in $\on{Sp}_{2g}(\mathbb{Z}_\ell)$:
$$M_0 = \lim_{i \to \infty} \left(\prod_{j = 1}^{i} X_j \right)M_i = \prod_{j = 1}^\infty X_j.$$
It follows that $\Gamma_{\ell^{2k}}$ is contained in the closure of $[\Gamma_{\ell^k}, \Gamma_{\ell^k}]$.

Now suppose $\ell = 2$. Observe that for each $k \geq 2$ we have
$$\id_{2g} + 2^k \cdot \mf{sp}_{2g}(\ZZ/ 4 \ZZ) = \ker(\Sp_{2g}(\ZZ/ 2^{k+2} \ZZ) \to \Sp_{2g}(\ZZ/ 2^k \ZZ)).$$
It follows from Statement (2) of Proposition~\ref{stopdrop} and Lemma~\ref{lemma:commutator-formula} that for every choice of $\id_{2g} + 2^{2k+i+1} W \in \Gamma_{2^{2k+i+1}}$ and for each nonnegative integer $i$, there exist $\id_{2g} + 2^k U \in \Gamma_{2^k}$ and $\id_{2g} + 2^{k+i} V \in \Gamma_{2^{k+i}}$ with the property that
	\begin{align*}
		& (\id_{2g} + 2^k U)^{-1}(\id_{2g} + 2^{k+i} V)(\id_{2g} + 2^k U)(\id_{2g} + 2^{k+i} V)^{-1} \equiv \id_{2g} + 2^{2k+i+1} W\,\,\, (\on{mod}{\ell^{2k+i+2}}).
		\end{align*}
One may now finish the proof by applying a similar inductive argument to the one used in the case $\ell \geq 3$.
\end{proof}

We are finally in position to prove the main result of this section.
\begin{proof}[Proof of Proposition~\ref{theorem:commutator-open}]
By Lemma~\ref{theorem:adelic-open}, it suffices to prove the following two statements:
\begin{enumerate}
\item The closure of $[H,H]$ surjects onto $\Sp_{2g}(\bz_\ell)$ for all but finitely many $\ell$.
\item The closure of $[H,H]$ maps onto an open subgroup of $\Sp_{2g}(\bz_\ell)$ for each $\ell$.
\end{enumerate}
For Statement (a), notice that $H$ surjects onto $\GSp_{2g}(\bz_\ell)$ for all but finitely many $\ell$.  Note that for $\ell \geq 3$, we have $[\GSp_{2g}(\bz_\ell), \GSp_{2g}(\bz_\ell)] = \Sp_{2g}(\bz_\ell)$ because, by~\cite[Proposition 1]{landesman-swaminathan-tao-xu:lifting-symplectic-group}, we have that
$$\Sp_{2g}(\ZZ_\ell) = [\Sp_{2g}(\ZZ_\ell), \Sp_{2g}(\ZZ_\ell)] \subset [\GSp_{2g}(\bz_\ell), \GSp_{2g}(\bz_\ell)] \subset \Sp_{2g}(\bz_\ell).$$
Thus, $[H,H]$ itself surjects onto $[\GSp_{2g}(\bz_\ell), \GSp_{2g}(\bz_\ell)] = \Sp_{2g}(\bz_\ell)$ for all $\ell \geq 3$.
		
To show statement (b), we prove that the closure of $[H', H']$ is open in $\Sp_{2g}(\bz_\ell)$ for any open subgroup $H' \subset \GSp_{2g}(\bz_\ell)$. Since $H'$ is open, there exists some $k \geq 1$ such that $\Gamma_{\ell^k} \subset H'$, so by Lemma~\ref{proposition:commutator-effective}, there exists $m \geq 2k$ such that $\Gamma_{\ell^m} \subset [\Gamma_{\ell^k}, \Gamma_{\ell^k}] \subset [H', H']$. Thus, $[H', H']$ contains an open subgroup and must therefore itself be open, as desired.
\end{proof}

\section{Background on Galois Representations of PPAVs}
\label{section:background}

This section is devoted to describing the basic definitions and properties concerning Galois representations associated to families of PPAVs.
Specifically, in Section~\ref{subsection:setup}, we construct these Galois representations and provide precise definitions for the various monodromy groups discussed in Section~\ref{weaintevergonnaberoyals}. Then, in Section~\ref{subsection:notation-for-families},
we explain how a family of PPAVs over a number field $K$ may be extended to a family over the number ring $\OO_K$. The notation introduced in this section will be utilized throughout the rest of the paper.

\subsection{Defining Galois Representations for Families of PPAVs}\label{subsection:setup}

Let $K$ be a number field, and let $g \geq 0$ be an integer. Fix a base scheme $T$ (we usually take $T$ to be $\spec K$ or an open subscheme of $\spec \mathcal O_K$), and let $U$ be an integral $T$-scheme with generic point $\eta$ (we usually take $U$ to be an open subscheme of $\mathbb{P}_K^r$ or $\mathbb{P}_{\OO_K}^r$). Let $A \to U$ be a \emph{family} of $g$-dimensional PPAVs, by which we mean the following:
\begin{itemize}
\item The morphism $A \to U$ is flat, proper, and finitely presented with smooth geometrically connected fibers of dimension $g$.
\item $A$ is a group scheme over $U$, and the resulting abelian scheme is equipped with a principal polarization.
\end{itemize}
Note that $A \rightarrow U$ is automatically abelian, smooth, and projective, and further observe that the fiber $A_u$ over any point $u \in U$ is a PPAV of dimension $g$ over the \mbox{residue field $\kappa(u)$ of $u$.}

Choose a geometric generic point $\ol{\eta}$ for $U$. If $\kappa(\eta)$ has characteristic prime to $m$, the action of the \'{e}tale fundamental group $\pi_1(U,\ol{\eta})$\footnote{For a general foundational reference on the \'etale fundamental group, see \cite{noopsortSGA1Grothendieck1971}.} on the geometric generic fiber $A_{\ol{\eta}}[m]$ gives rise to a continuous linear representation whose image is constrained by the Weil pairing to lie in the general symplectic group $\GSp_{2g}(\ZZ/m \ZZ)$. We denote this \emph{mod-$m$ representation} by
\begin{equation}\label{atoll}
\rho_{A,m} \colon \pi_1(U, \ol{\eta}) \to \GSp_{2g}(\ZZ/m \ZZ).
\end{equation}
The map in~\eqref{atoll} is well-defined up to the choice of base-point $\ol{\eta}$, and choosing a different such $\ol{\eta}$ would only alter the image of $\rho_{A,m}$ by an inner automorphism of $\GSp_{2g}(\mathbb Z/m\mathbb Z)$.
For this reason, when it will not lead to confusion, we may omit
the basepoint from our notation and write $\pi_1(U)$ for $\pi_1(U, \overline \eta)$.

If $\ell$ is a prime not dividing the characteristic of $\kappa(\eta)$, then we can take the inverse limit of the mod-$\ell^k$ representations to obtain the \emph{$\ell$-adic representation}
\begin{equation}\label{itsladicguys}
\rho_{A,\ell^\infty} \colon \pi_1(U) \to \varprojlim_k{\GSp_{2g}(\ZZ/\ell^k \ZZ)}.
\end{equation}
Moreover, if $\kappa(\eta)$ has characteristic $0$, we can take the inverse limit of all the mod-$m$ representations (or equivalently the product of all the $\ell$-adic representations) to obtain an \emph{adelic} or \emph{global representation}
\begin{equation}\label{thisisthepartofme}
	\rho_A \colon \pi_1(U) \to \varprojlim_m {\GSp_{2g}(\ZZ/m \ZZ)} \simeq \GSp_{2g}(\wh{\ZZ}).
\end{equation}
\vspace*{-0.2in}
\begin{remark}
In the situation that $U = \spec K$, the choice of $\ol{\eta}$ corresponds to a choice of algebraic closure $\ol{K}$ of $K$. Taking $G_K \defeq \Gal(\ol{K}/K)$ to be the absolute Galois group, we have that $\pi_1(U, \ol{\eta}) = G_K$. This recovers the notion of a Galois representation of a PPAV over a field as a map $\rho_A: G_K \rightarrow \GSp_{2g}(\widehat{\mathbb Z})$.
\end{remark}
\vspace*{-0.1in}
\begin{remark}
	\label{remark:det-rho-is-chi}
For a commutative ring $R$, recall from the definition of the general symplectic group that we have a multiplier map $\on{mult} \colon \GSp_{2g}(R) \to R^\times$. Let $\chi_m$ be the mod-$m$ cyclotomic character, and let $\chi$ be the cyclotomic character. If $U = \spec k$, (with $k$ an arbitrary characteristic $0$ field) it follows from $G_k$-invariance of the Weil pairing that $\chi_m = \mult \circ \rho_{A,m}$ and $\chi = \mult \circ \rho_{A}$.
More generally, if $U$ is normal and integral, and $\phi: \pi_1(U) \rightarrow \pi_1(\spec K)$, then
$\chi \circ \phi = \mult \circ \rho_A$,
which holds because it holds for the generic fiber $A_\eta \rightarrow \spec K(\eta)$, and the map $\pi_1(\eta) \rightarrow \pi_1(U)$ is surjective.
\end{remark}

We now define the monodromy groups associated to the representations defined above. We call the image of $\rho_A \colon \pi_1(U) \to \GSp_{2g}(\wh{\ZZ})$ the {\it monodromy} of the family $A \to U$, and we denote it by $\mono_A$. When the base scheme is $T = \spec K$, we also define the {\it geometric monodromy}, denoted by $\mono_A^{\on{geom}}$, to be the image of the adelic representation $\rho_{A_{\overline K}}\colon \pi_1(U_{\overline K}) \rightarrow \GSp_{2g}(\widehat{\mathbb Z})$ associated to the base-changed family $A_{\ol{K}} \to U_{\ol{K}}$.
Since the cyclotomic character is trivial on $G_{\ol{K}}$, it follows that $\mono_A^{\on{geom}}$ is actually a subgroup of $\Sp_{2g}(\wh{\ZZ})$. We write $\mono_A(m)$ and $\mono_A^{\on{geom}}(m)$ for the mod-$m$ reductions of the above-defined monodromy groups.
We say $A \rightarrow U$ has big monodromy if $\mono_A$ is open
in $\GSp_{2g}(\zh)$ and $A \rightarrow U$ has big geometric monodromy if
$\mono_A^{\on{geom}}$ is open in $\Sp_{2g}(\zh)$.

In particular, for each $u \in U$, $\mono_{A_u}$ and $\mono_{A_u}^{\on{geom}}$ are the monodromy groups associated to the family $A_u \to \spec \kappa(u)$. Since $A_u$ is the pullback of $A$ along $\iota: u \rightarrow U$, $\rho_{A_u} = \iota \circ \rho_A$ and we obtain an inclusion $H_{A_u} \subset H_A$. Note that if $U$ is normal, then the map $\pi_1(\eta) \rightarrow \pi_1(U)$ is surjective, so we have that $\mono_{A_\eta} = \mono_A$.

\subsection{Extending Families over $K$ to $\OO_K$}
\label{subsection:notation-for-families}

Recall that, for a single abelian variety $A_u$ over $u = \spec K$, good reduction for $A_u$ at a prime $\mf{p} \in \Sigma_K$ implies that the Galois representation $\rho_{A_u, m} \colon G_K \to \GSp(\bz / m \bz)$ is unramified at $\mf{p}$, provided that $\mf{p}$ does not divide $m$. All but finitely many primes $\mf{p}$ are primes of good reduction for $A_u$. Similarly, for a family $A \to U$ over $\spec K$, extending the definition of this family ``across'' a prime $\mf{p} \in \Sigma_K$ reveals constraints on the monodromy of that family and its subfamilies. The purpose of this section is to explain why any family $A \to U$ can be extended across most primes in $\Sigma_K$. The constructions introduced here become particularly important in Section~\ref{subsection:applying-wallace}, where we apply the results of~\cite{scoopdedoo}. A similar treatment of these constructions can be found in~\cite[p.~460-462]{scoopdedoo}.

Retain the setting of Theorem~\ref{theorem:main}. Start with a family $A \rightarrow U$ of PPAVs over $\spec K$. Using standard spreading out techniques as in  \cite[\S8]{EGAIV.3} (see in particular \cite[8.10.5(xii)]{EGAIV.3}, \cite[9.7.7(ii)]{EGAIV.3}, and \cite[17.7.8(ii)]{EGAIV.4}), we can extend the family $A \rightarrow U$ to a family $\mathcal A \rightarrow \mathcal U$, where $\mathcal U$ is an open subscheme of $\mathbb P^r_{\mathcal O_K}$, whose generic fiber over $\spec K \rightarrow \spec \mathcal O_K$ is just $A \rightarrow U$. 
Recall from Section~\ref{subsection:setup} that the term ``family'' means that $\mathcal A \rightarrow \mathcal U$ is smooth and proper with geometrically connected fibers and that $\mathcal A$ is an abelian scheme over $\mathcal U$ with a principal polarization.
This construction is depicted in the following commutative diagram:
\[
\begin{tikzcd}
A \ar[r] \ar[d] & \mathcal{A} \ar[d] \\
U \ar[r] \ar[d] & \mathcal{U} \ar[d] \ar[r, hookrightarrow, "\text{open emb.}"] & \mathbb{P}_{\OO_K}^r \\
\spec K \ar[r] & \spec \OO_K
\end{tikzcd}
\]

Let $Z \defeq \mathbb P_K^r \setminus U$ be the locus where the original family is not defined, and let $\mathcal Z$ denote the closure of $Z$ in $\mathbb P^r_{\mathcal O_K}$. Since the bottom square in the diagram above is Cartesian, each irreducible component of $\mathbb{P}^r_{\OO_K} \setminus \mathcal{U}$ that is not contained in $\mathcal{Z}$ cannot map generically onto $\spec \OO_K$ and must therefore map to a single prime $\pp \in \Sigma_K$. Since there are finitely many irreducible components of $\mathcal{Z}$, the set $S$ of primes $\fp \in \spec \mathcal O_K$ for which $\mathbb P^r_{\mathbb F_\pp} \setminus \mathcal U_{\mathbb F_\pp} \neq \mathcal Z_{\mathbb F_\pp}$ is a finite set. The primes in $S$ can be thought of as the ``bad primes'' for the family: the smoothness of $\mathcal A \to \mathcal U$ implies that any abelian variety $A_u$ for $u \in U(K)$ will have good reduction away from the primes in $S$ and the primes lying under the (finite) intersection $\ol{\{u\}} \cap \mathcal{Z} \subset \mathbb{P}^r_{\OO_K}$.

\subsubsection{Monodromy groups of subfamilies}
Let $m \in \bz$, let $P_m \subset \Sigma_K$ be the set of primes dividing $m$, and let $\spec \OO_{P_m}$ be the complement of $P_m$ in $\spec \OO_K$. Then the base change $\mathcal U_{\OO_{P_m}}$ of $\mathcal{U}$ from $\spec \OO_K$ to $\spec \OO_{P_m}$ is the open subset of $\mathcal U$ on which $\mathcal A [m] \to \mathcal U$ is unramified and hence finite \'etale. Therefore, we obtain a finite \'etale cover $\mathcal{A}_{\OO_{P_m}}[m] \to \mathcal U_{\OO_{P_m}}$ and hence a map $\rho \colon \pi_1(\mathcal U_{\OO_{P_m}}) \to \GSp_{2g}(\bz / m\bz)$ just as in~\ref{subsection:setup}. The original family of interest can be thought of as a subfamily of this one: we have maps $U_{\ol{K}} \to U \to \mathcal{U}_{\OO_{P_m}}$, from which we obtain maps
\[
\begin{tikzcd}
	\pi_1(U_{\ol{K}}) \ar[r] & \pi_1(U) \ar[r] & \pi_1(\mathcal{U}_{\OO_{P_m}}) \ar[r, "\rho"] & \GSp_{2g}(\bz / m\bz).
\end{tikzcd}
\]
\begin{lemma} \label{lem-fill}
	The continuous map $\pi_1(U) \to \pi_1(\mathcal{U}_{\OO_{P_m}})$ is surjective.
\end{lemma}
\begin{proof}
This lemma is a consequence of~\cite[Expos\'e V, Proposition 8.2]{noopsortSGA1Grothendieck1971}; we nonetheless include a proof because it helps illustrate the constructions introduced in this section. It suffices to show that the composition of this map with any surjective continuous map $\pi_1(\mathcal{U}_{\OO_{P_m}}) \to G$ onto a finite group $G$ is surjective. According to \cite[\href{https://stacks.math.columbia.edu/tag/03SF}{Tag 03SF}]{stacks-project}, a finite quotient of the \'etale fundamental group corresponds to a connected finite Galois cover, so let $\mathcal V_m \to \mathcal U_{\OO_{P_m}}$ be the cover corresponding to our chosen surjection. By \cite[\href{https://stacks.math.columbia.edu/tag/0DV5}{Tag 0DV6}]{stacks-project}, the composed map $\pi_1(U) \to \pi_1(\mathcal{U}_{\OO_{P_m}}) \to G$ gives a $\pi_1(U)$-action on $G$ which corresponds to the pulled back cover $(\mathcal V_m)_K \to U$. The latter is connected if and only if the composed map is surjective. Since $\mathcal{V}_m$ is connected and \'etale over $\spec \OO_{P_m}$, it is irreducible, which implies that $(\mathcal{V}_m)_K$ is irreducible (its generic points correspond to those of $\mathcal V_m$), hence connected.
\end{proof}
By \cite[\href{https://stacks.math.columbia.edu/tag/0DV5}{Tag 0DV6}]{stacks-project}, the resulting monodromy representation $\pi_1(U) \to \GSp_{2g}(\bz / m \bz)$ equals that obtained from the pullback of the finite \'etale cover $\mathcal{A}_{\OO_{P_m}}[m] \to \mathcal U_{\OO_{P_m}}$ to $U$. But the pullback is just the family $A[m] \to U$, so this monodromy representation equals $\rho_{A, m}$, and its image equals $H_A(m)$. The lemma therefore implies that the image of the map $\pi_1(\mathcal{U}_{\OO_{P_m}}) \to \GSp_{2g}(\bz / m\bz)$ equals $H_A(m)$. Similarly, the map $\pi_1(U_{\ol{K}}) \to \GSp_{2g}(\bz / m \bz)$ has image equal to $H_A^{\on{geom}}(m)$.

Moreover, for $\mf{p} \in \Sigma_K$ not dividing $m$, we can also consider the subfamilies $\mathcal U_{\ol{\mathbb{F}}_\pp} \to \mathcal U_{\mathbb{F}_\pp} \to \mathcal U_{\OO_{P_m}}$ obtained by extending scalars along the maps $\OO_{P_m} \to \mathbb{F}_\pp \to \ol{\mathbb{F}}_\pp$ for some algebraic closure $\ol{\mathbb{F}}_\pp$ of $\mathbb{F}_\pp$. As before, we obtain maps
\[
\begin{tikzcd}
\pi_1(\mathcal U_{\ol{\mathbb{F}}_\pp}) \ar[r] & \pi_1(\mathcal U_{\mathbb{F}_\pp}) \ar[r] & \pi_1(\mathcal{U}_{\OO_{P_m}}) \ar[r, "\rho"] &  \GSp_{2g}(\bz / m\bz).
\end{tikzcd}
\]
We denote by $H_{A, \pp}(m)$ and $H_{A, \pp}^{\on{geom}}$ the images of the maps $\pi_1(\mathcal U_{\mathbb{F}_\pp}) \to \GSp_{2g}(\bz / m\bz)$ and $\pi_1(\mathcal U_{\ol{\mathbb{F}}_\pp}) \to \GSp_{2g}(\bz / m\bz)$, respectively.

\subsubsection{Notation for Galois \'etale covers}
As explained in the proof of Lemma~\ref{lem-fill}, finite quotients of the \'etale fundamental group correspond to connected finite Galois \'etale covers. We now fix notation for the Galois \'{e}tale covers introduced in the proof of Lemma~\ref{lem-fill} that will be used later in Section~\ref{section:new-proof-of-main-theorem} to state and verify Wallace's criteria \cite{scoopdedoo}.
\begin{itemize}
	\item Let $\mathcal{V}_m$ be the cover of $\mathcal{U}_{\OO_{P_m}}$ corresponding to the map $\pi_1(\mathcal{U}_{\OO_{P_m}}) \to \GSp_{2g}(\bz / m \bz)$.
	\item Let $V_m$ be the cover of $U$ corresponding to the map $\pi_1(U) \to \GSp_{2g}(\bz / m \bz)$.
\end{itemize}
Here, each map from $\pi_1(-)$ to a finite group gives a quotient of $\pi_1(-)$ as its image. By the reasoning of Lemma~\ref{lem-fill}, $V_m = (\mathcal{V}_m)_K$.
\begin{remark}
	The result of Lemma~\ref{lem-fill} is special to the base change $\OO_K \to K$. In general, the other maps of $\pi_1(-)$s will not be surjective, nor will the finite Galois \'etale covers $(V_m)_{\ol{K}}$, $(\mc{V}_m)_{\mathbb F_\pp}$, and $(\mc{V}_m)_{\ol{\mathbb F}_\pp}$ be connected.
\end{remark}

\section{Proof of the Theorem~\ref{theorem:main}} \label{section:new-proof-of-main-theorem}

\def\arraystretch{0.7}
\subsection{Outline of the Proof}
\label{subsection:outline}

With the view of making the proof of Theorem~\ref{theorem:main} more readily comprehensible, we now briefly describe the key aspects of the argument. We encourage the reader to refer to Figure~\ref{figure:proof-schematic} for a schematic diagram illustrating the argument.

We begin in Section~\ref{subsection:big-geometric-monodromy-equivalence}
by 
proving Proposition~\ref{proposition:big-geometric-monodromy-reduction},
showing 
that a non-isotrivial family with big 
monodromy also has big geometric monodromy.
Then, in Section~\ref{toomanynotes}, we
introduce some of the notation and standing assumptions employed in the proof.
In particular, since our family has big geometric monodromy, by
Proposition~\ref{proposition:big-geometric-monodromy-reduction},
we are able to define the constant $C$ in point (b) of Section~\ref{toomanynotes}, which will later be needed to apply
the results of \cite{scoopdedoo} (see Section~\ref{subsubsection:setup-and-statement-of-wallace}).

Then, in Section~\ref{attheskyfall}, we reduce the problem to
checking that for an appropriately chosen integer $M'$ depending on the family, most members of the family have the same mod-$M'$ image as that of the family; and (2) that for all sufficiently large primes $\ell$, most members of the family have the same mod-$\ell$ image as that of the family.

The mod-$M'$ image is dealt with in Section~\ref{subsection:cohen-serre}
using Proposition~\ref{proposition:applying-cohen-serre}, which is the Cohen-Serre version of the Hilbert Irreducibility Theorem.
For dealing with the mod-$\ell$ images,
there are two regimes of primes to consider, a medium
regime and a high regime, when $\ell$ is bigger than a suitable power of $\log B$.
We handle both of these regimes in Section~\ref{subsection:applying-wallace} by applying a result of Wallace, \cite[Theorem 3.9]{scoopdedoo}, for which we must verify the following four conditions:~\ref{assumption-4},~\ref{property-a1},~\ref{property-a2}, and~\ref{property-a3}. The rest of Section~\ref{section:new-proof-of-main-theorem} is devoted to verifying that these conditions hold in our setting.

Conditions~\ref{assumption-4} and~\ref{property-a1}, which are fairly easy to check, are treated in Sections~\ref{subsection:applying-wallace} and~\ref{ver1}. Next, condition~\ref{property-a2}
is dealt with in Section~\ref{ver2} by applying the Grothendieck Specialization Theorem in
Proposition~\ref{proposition:check-B}.
These first three conditions together essentially yield an effective
version of the Hilbert Irreducibility Theorem, which allows us to
check primes $\ell$ in the medium regime. Finally, in Section~\ref{ver3}, we verify condition~\ref{property-a3},
which allows us to dispense with primes in the high regime.
The key input to checking this condition is a recent result of Lombardo,
stated in Theorem~\ref{theorem:lombardo}.
In order to apply Lombardo's result to our setting,
as is done in Proposition~\ref{proposition:check-C}, we must verify two hypotheses and relate the na\"{i}ve height we are using to the Faltings height used in Theorem~\ref{theorem:lombardo}.
The first hypothesis is verified in Lemma~\ref{lemma:davide-1}
using \cite[Proposition 5]{ellenbergEHK:non-simple-abelian-varieties-in-a-family}.
The second hypothesis is a somewhat trickier condition, and we verify it in Lemma~\ref{lemma:davide-2} using the large sieve,
Theorem~\ref{theorem:large-sieve}.
In order to apply the large sieve, we must bound contributions at
each prime, which is done in
Proposition~\ref{proposition:good-v} using a general scheme-theoretic
result of Ekedahl~\cite[Lemma 1.2]{ekedahl1988effective}
together with Proposition~\ref{proposition:good-cover}.
We conclude the section with a brief appendix concerning the relationship between the na\"{i}ve height and the Faltings height (see Lemma~\ref{lemma:height}).

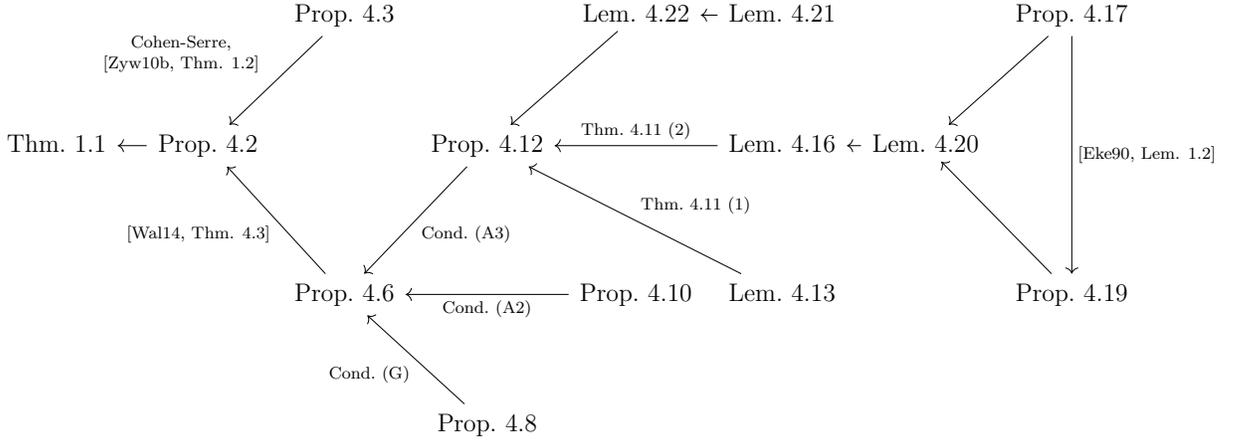
\begin{figure}
	\centering
\begin{equation}
  \nonumber
\begin{tikzpicture}[baseline= (a).base]
\node[scale=.8] (a) at (0,0){
  \begin{tikzcd}[column sep=tiny]
	  \qquad && & \text{Prop.}\nolink{~\ref{proposition:applying-cohen-serre}} \ar{ddl}[swap]{\begin{array}{c}\text{Cohen-Serre,}\\ \text{\nolink{\cite[Thm. 1.2]{zywina2010hilbert}}}\end{array}} & &  \text{Lem.}\nolink{~\ref{lemma:height}} \ar{ldd} & \text{Lem.\nolink{~\ref{lemma: height-0}}} \ar{l} && \text{Prop.\nolink{~\ref{proposition:good-cover}}} \ar{ldd}\ar{dddd}{\text{\nolink{\cite[Lem. 1.2]{ekedahl1988effective}}}}  \\
	  \\
	  \qquad \text{Thm.\nolink{~\ref{theorem:main}}} && \text{Prop.\nolink{~\ref{lemma:ab-cyc}}} \ar{ll} & & \text{Prop.\nolink{~\ref{proposition:check-C}}}  \ar{ddl}{\text{Cond.\nolink{~\ref{property-a3}}}} & & \text{Lem.\nolink{~\ref{lemma:davide-2}}} \ar{ll}[swap]{\text{Thm.\nolink{~\ref{theorem:lombardo}}\hspace{.1cm}(2)}}  & \text{Lem.\nolink{~\ref{lemma:bounding-Y-reduction}}} \ar{l} & \\
\qquad && & & & & &	  \\
\qquad && & \text{Prop.\nolink{~\ref{corollary:applying-wallace}}}\ar{uul}{\text{\nolink{\cite[Thm. 4.3]{scoopdedoo}}}} && \text{Prop.\nolink{~\ref{proposition:check-B}}}  \ar{ll}{\text{Cond.\nolink{~\ref{property-a2}}}}   & \text{Lem.\nolink{~\ref{lemma:davide-1}}}  \ar{uull}[swap]{\text{Thm.\nolink{~\ref{theorem:lombardo}}\hspace{.1cm}(1)}}  & & \text{Prop.\nolink{~\ref{proposition:good-v}}} \ar{uul} \\
&&& & &&
\\
\qquad && & & \text{Prop.\nolink{~\ref{proposition:verifying-assumptions}}}  \ar{uul}{\text{Cond.\nolink{~\ref{assumption-4}}}} & & &
 \end{tikzcd}
};
\end{tikzpicture}
\end{equation}
\caption{
A schematic diagram for the proof of the main theorem, Theorem~\ref{theorem:main}.}
\label{figure:proof-schematic}
\end{figure}
\def\arraystretch{1.3}

\subsection{Equivalence of Big Geometric Monodromy and Big Monodromy}
\label{subsection:big-geometric-monodromy-equivalence}
In the course of the proof, it will be useful to know that our given family
$A \rightarrow U$ not only has big monodromy, but also has big
geometric monodromy.
In particular, this is crucially needed to define the constant $C$ in point
$(b)$ of Section~\ref{toomanynotes}, which is used in 
applying the results of~\cite{scoopdedoo} (see Section~\ref{subsubsection:setup-and-statement-of-wallace}).
We now prove the following result, implying that our given family
has big geometric monodromy.

\begin{proposition}
	\label{proposition:big-geometric-monodromy-reduction}
	Suppose $A \rightarrow U$ is a non-isotrivial family of abelian varieties of relative dimension $g \geq 2$, with $U$ a smooth geometrically connected
	scheme over
	a number field $K$.
	Then, $A$ has big geometric monodromy if and only if it has big 
	monodromy.
\end{proposition}
\begin{proof}
We first show the easier direction: if the family $A \to U$ has big geometric 	monodromy then $A \to U$ also has big monodromy, in the sense that $\mono_A$ is open in $\GSp_{2g}(\wh{\ZZ})$. To see this, consider the exact sequence
    \[
    	\begin{tikzcd}
        	0 \ar{r} & \Sp_{2g}(\zh) \ar{r} & \GSp_{2g}(\zh) \ar{r}{\on{mult}} & \zh^\times \ar{r} & 0.
        \end{tikzcd}
    \]
    Since $\mono_A^{\on{geom}} \subset \mono_A$, the big geometric monodromy assumption tells us that $\mono_A \cap \Sp_{2g}(\zh)$ is open in $\Sp_{2g}(\zh)$. It therefore suffices to show that $\on{mult}(\mono_A)$ is open in $\zh^\times$. But $\on{mult}(\mono_A) = \chi(G_K)$, as mentioned in Remark~\ref{remark:det-rho-is-chi}, and $\chi(G_K)$ has finite index because $K/\mathbb Q$ \mbox{has finite degree.}

It only remains to prove that if the family has big monodromy
	and is non-isotrivial, it has big geometric monodromy.
	To show this, from the exact sequence
\begin{equation}
		\nonumber
		\begin{tikzcd}
			1 \ar {r} & \pi_1(U_{\overline K})  \ar {r} & \pi_1(U) \ar {r} & \pi_1(K) \ar {r} & 1
		\end{tikzcd}\end{equation}
	$\pi_1(U_{\overline K}) \subset \pi_1(U)$ is normal.
	Therefore, $\mono_A^{\on{geom}}$ is a normal subgroup of $\mono_A$, and hence also a normal subgroup of $\mono_A \cap \Sp_{2g}(\zh)$.
	Let $\psi: \Sp_{2g}(\bz) \ra \Sp_{2g}(\zh)$ denote the natural profinite completion map.	
	Since $\mono_A^{\on{geom}} \subset \mono_A \cap \Sp_{2g}(\zh)$ is normal, 
	it follows that $\psi^{-1}(\mono_A^{\on{geom}}) \subset \psi^{-1}(\mono_A \cap \Sp_{2g}(\zh))$ is normal.
	Since $\mono_A$ has finite index in $\GSp_{2g}(\zh)$, $\psi^{-1}(\mono_A \cap \Sp_{2g}(\zh))$
	has finite index in $\Sp_{2g}(\bz)$. Since $g \geq 2$ (so that $\Sp_{2g}(\bz)$ has rank at least $2$),
	by the Margulis normal subgroup theorem, 
	(see, for example \cite[Theorem 17.1.1]{morris:introduction-to-arithmetic-groups},)
	$\psi^{-1}(\mono_A^{\on{geom}})$ either has finite index in $\psi^{-1}(\mono_A \cap \Sp_{2g}(\zh))$
	or is finite.
	We will show that in the first case $A$ has big geometric monodromy and in the second case $A$ is isotrivial.

	In the case that $\psi^{-1}(\mono_A^{\on{geom}})$ has finite index in $\psi^{-1}(\mono_A \cap \Sp_{2g}(\zh))$, 
$\psi^{-1}(\mono_A^{\on{geom}})$ also has finite index in $\Sp_{2g}(\bz)$. 
Then, since $\mono_A^{\on{geom}}$ is closed,
the finite set
$\Sp_{2g}(\bz)/\psi^{-1}(\mono_A^{\on{geom}})$ is dense in the profinite
space $\Sp_{2g}(\zh)/H_A^{\on{geom}}$.
It follows that $\mono_A^{\on{geom}}$ also has finite index
in $\Sp_{2g}(\zh)$, meaning $A$ has big geometric monodromy.

	To conclude the proof, it only remains to show that if $\psi^{-1}(H_A^{\on{geom}})$ is finite, then $A$ is isotrivial.
	In this case, let $M_A^{\on{geom}}$ denote the image of the topological monodromy
	representation $\pi_1^{\on{top}}(U_{\bc}) \rightarrow \Sp_{2g}(\bz)$.
	By \cite[Expos\'e XIII, Proposition 4.6]{SGA1Grothendieck1971}, we have 
	$\pi_1(U_{\bc}) \simeq \pi_1(U_{\overline K})$, and therefore the comparison theorem tells us that $\mono_A^{\on{geom}}$
	is the profinite completion of $M_A^{\on{geom}}$.
	This implies $M_A^{\on{geom}} \subset \psi^{-1}(\mono_A^{\on{geom}})$
	and so $M_A^{\on{geom}}$ is finite. It follows that $H_A^{\on{geom}}$ is finite, being the profinite completion
	of $M_A^{\on{geom}}$.
	After a making a finite base change, we may assume
	$H_A^{\on{geom}}$ is trivial.
	Then, it is a standard fact that $A$ is isotrivial when its monodromy
	representation is trivial. For example, this follows from
	\cite{grothendieck:un-theoreme-sur-les-homomorphismes}.
\end{proof}

\subsection{Notation and Standing Assumptions}\label{toomanynotes}

Before proceeding with the proof, we set some notation and assumptions, which will remain in place for the remainder of this section.
\begin{enumerate}
\item As mentioned in Remark~\ref{remkydoo}, the case where $g=1$ is handled in~\cite[Theorem 7.1]{zywina2010hilbert}, so we will restrict our consideration to the case where $g \geq 2$.
\item 	
    Since we are assuming that $A \to U$ has big monodromy, it follows that $A \to U$ has big geometric monodromy, by Proposition~\ref{proposition:big-geometric-monodromy-reduction}. Define $C$ to be the smallest
integer bigger than $2$, depending only on $U$, with the property that for all primes $\ell > C$ we have $\mono_A^{\on{geom}}(\ell) = \Sp_{2g}(\mathbb Z/\ell \ZZ)$ and $\mono_A(\ell) = \GSp_{2g}(\mathbb Z/\ell \ZZ)$.
\item Using \cite[Proposition 6.1]{zywina2010hilbert} and the explanation
	given after the statement of
	\cite[Theorem 7.1]{zywina2010hilbert}, one readily checks that in Theorem~\ref{theorem:main}, the asymptotic statement for $K$-valued points (i.e., points in $U(K)$) can be deduced immediately from the statement for \emph{lattice points} (i.e., points in $U(K) \cap \OO_K^r$). In what follows, we will work with $K$-valued points or lattice points depending on what is most convenient.
\item Let $K^{\on{cyc}} \subset \ol{K}$ denote the maximal cyclotomic extension of $K$, and let $K^{\ab} \subset \ol{K}$ denote the maximal abelian extension of $K$.
\item In what follows, for a subgroup $H$ of a topological group $G$, let $[H,H]$ denote the \emph{closure} of the usual commutator subgroup.
\end{enumerate}

\subsection{Main Body of the Proof} \label{attheskyfall}

We begin by reducing the proof of Theorem~\ref{theorem:main} to proving Proposition~\ref{lemma:ab-cyc}.
\begin{proof}[Proof of Theorem~\ref{theorem:main} assuming Proposition~\ref{lemma:ab-cyc}]
As argued in~\cite[Proof of Theorem 7.1]{zywina2010hilbert}, for any $u \in U(K)$ we have
\begin{align*}
[\mono_A :\mono_{A_u}] & = [\mono_A \cap \Sp_{2g}(\wh{\ZZ}) : \rho_{A_u}(\Gal(\ol{K}/K^{\on{cyc}}))].
\intertext{In the case that $K = \QQ$, the Kronecker-Weber Theorem tells us that $\bq^{\cyc} = \bq^{\ab}$, so we have}
[\mono_A : \mono_{A_u}] & = \delta_{\QQ} \cdot [[\mono_A, \mono_A] : \rho_{A_u}(\Gal(\ol{\bq}/\bq^{\ab}))],
\end{align*}
where $\delta_{\QQ}$ is the index of $[\mono_A, \mono_A]$ in $\mono_A \cap \Sp_{2g}(\wh{\ZZ})$. Then Theorem~\ref{theorem:main} follows immediately from point (c) of Section~\ref{toomanynotes} and the following proposition.
\end{proof}
	\begin{proposition}\label{lemma:ab-cyc}
		Let $B, n > 0$. We have the following asymptotic statements, where the implied constants depend only on $U$ and $n$:
		\begin{enumerate}
			\item[\customlabel{asymptotic-commutator}{(1)}] For every number field $K$,
			\[
				\frac{ |\{ u\in U(K) \cap \mc{O}^r_K: \lVert u \rVert \le B,\, \rho_{A_u}(\on{Gal}(\ol{K} / K^{\on{ab}})) = [\mono_A, \mono_A] \}| }{ |\{ u\in U(K) \cap \mc{O}^r_K: \lVert u \rVert \le B \}| } = 1 + O(( \log B )^{-n}).
			\]
		\item[\customlabel{asymptotic-symplectic}{(2)}] Furthermore, if $K \neq \bq$,
			\[
				\frac{ |\{ u\in U(K) \cap \mc{O}^r_K: \lVert u \rVert \le B,\, \rho_{A_u}(\on{Gal}(\ol{K} / K^{\on{cyc}})) = \mono_A \cap \Sp_{2g}(\zh) \}| }{ |\{ u\in U(K) \cap \mc{O}^r_K: \lVert u \rVert \le B \}| } = 1 + O(( \log B )^{-n}).
			\]
		\end{enumerate}
	\end{proposition}
    \begin{remark}
 	Proposition~\ref{lemma:ab-cyc} is a generalization of~\cite[Proposition 7.9]{zywina2010hilbert} from the case $g = 1$ to all dimensions. We shall prove it assuming Proposition~\ref{proposition:applying-cohen-serre} and Proposition~\ref{corollary:applying-wallace}. The basic idea behind the argument is to reduce the problem of studying the (global) monodromy groups to one of studying the mod-$M'$ and mod-$\ell$ monodromy groups.
    \end{remark}
	\begin{proof}[Proof assuming Proposition~\ref{proposition:applying-cohen-serre} and Proposition~\ref{corollary:applying-wallace}]
		Assuming point (1), the proof of point (2) is completely analogous to the proof of~\cite[Proposition 7.9(ii)]{zywina2010hilbert}, which consists of two key steps. The first is the fact that $[\mono_A, \mono_A]$ is an open normal subgroup of $\mono_A \cap \Sp_{2g}(\wh{\ZZ})$, which follows from Proposition~\ref{theorem:commutator-open}. The second is~\cite[Proposition 7.7]{zywina2010hilbert}, which is a variant of Hilbert's Irreducibility Theorem and does not depend in any way on the context of elliptic curves (with which~\cite[Section 7]{zywina2010hilbert} is concerned). It therefore suffices to prove point (1).

Since $\Gal(\ol{K}/K^{\ab}) = [G_K, G_K]$, it follows by the continuity of $\rho_{A_u}$ and the compactness of profinite groups that $\rho_{A_u}(\Gal(\overline K/K^{\ab})) = [\mono_{A_u}, \mono_{A_u}]$.
Thus $\rho_{A_u}(\Gal(\overline K/K^{\ab}))$ is a closed subgroup of $\left[\mono_A, \mono_A \right]$. Moreover, by Proposition~\ref{theorem:commutator-open}, $[ \mono_A, \mono_A ]$ is an open subgroup of $\Sp_{2g}(\zh)$, so we may apply Proposition~\ref{theorem:adelic-surjective-subset} with $G = [ \mono_A, \mono_A ]$ and $H = \rho_{A_u}(\Gal(\overline K/K^{\ab}))$. In so doing, we obtain a positive integer $M$ so that the only closed subgroup of $[\mono_A, \mono_A]$ whose mod-$M$ reduction equals $[\mono_A, \mono_A](M) = [\mono_A(M), \mono_A(M)]$ and whose mod-$\ell$ reduction equals $\Sp_{2g}(\bz/\ell \ZZ)$ for every prime number $\ell \nmid M$ is $[\mono_A, \mono_A]$ itself. The same property is true when $M$ is replaced by any multiple $M'$ of $M$, and we choose a multiple $M'$ which is divisible by all primes less than $C$, where $C$ is defined as in point (b) of Section~\ref{toomanynotes}. The defining property of $M'$ then implies that
\begin{align}
			& \frac{|\{ u \in U(K) \cap \mathcal O_K^r : \|u \| \leq B,\, \rho_{A_u}(\Gal(\overline K/K^{\ab})) \neq [ \mono_A, \mono_A ] \}|}{|\{ u \in U(K) \cap \mathcal O_K^r : \|u \| \leq B \} |} \leq \nonumber \\
			& \qquad  \frac{|\{ u \in U(K) \cap \mathcal O_K^r : \|u \| \leq B,\, \rho_{A_u, M'}(\Gal(\overline K/K^{\ab})) \neq [ \mono_A(M'), \mono_A(M') ] \}|}{|\{ u \in U(K) \cap \mathcal O_K^r : \|u \| \leq B \} |}\,\,+ \raisetag{-0.6cm}{\label{equation:goursat-m-bound}}\\
			& \qquad \frac{|\{ u \in U(K) \cap \mathcal O_K^r : \|u \| \leq B,\, \rho_{A_u, \ell}(\Gal(\overline K/K^{\ab})) \neq \Sp_{2g}(\mathbb Z/\ell \mathbb Z) \text{ for some } \ell \nmid M' \}|}{|\{ u \in U(K) \cap \mathcal O_K^r : \|u \| \leq B \} |} \raisetag{-0.6cm}{\label{equation:goursat-l-bound}}.
\end{align}
The rest of this section is devoted to finding upper bounds for~\eqref{equation:goursat-m-bound} and~\eqref{equation:goursat-l-bound}. To bound~\eqref{equation:goursat-m-bound}, notice that we have
$$\rho_{A_u, M'}(\Gal(\overline K/K^{\ab})) \neq [\mono_A(M'), \mono_A(M')] \Longrightarrow \mono_{A_u}(M') \neq \mono_A(M').$$ It then follows from Proposition~\ref{proposition:applying-cohen-serre} that~\eqref{equation:goursat-m-bound} is bounded by $O( (\log B)/B^{\left[ K:\mathbb Q \right]/2} )$. To bound~\eqref{equation:goursat-l-bound}, notice that for $\ell \geq 3$ we have
$$\rho_{A_u, \ell}(\Gal(\ol{K}/K^{\ab})) \neq \Sp_{2g}(\bz / \ell \ZZ) \Longrightarrow \mono_{A_u}(\ell) \not\supset \Sp_{2g}(\bz / \ell \ZZ) ,$$ because~\cite[Proposition 1(a)]{landesman-swaminathan-tao-xu:lifting-symplectic-group} tells us that $\Sp_{2g}(\bz /\ell \ZZ)$ has trivial abelianization for $\ell \ge 3$. Since $C \geq 3$ by definition, it follows from Proposition~\ref{corollary:applying-wallace} that~\eqref{equation:goursat-l-bound} is $O(( \log B )^{-n} )$, since $\ell \nmid M'$ implies that $\ell > C$.
Combining the above estimates completes the proof of point (1).
\end{proof}
It now remains to bound the terms~\eqref{equation:goursat-m-bound} and~\eqref{equation:goursat-l-bound}.

\subsection{Bounding the Contribution of~\eqref{equation:goursat-m-bound}} \label{subsection:cohen-serre}

The next result is the means by which we bound~\eqref{equation:goursat-m-bound}; it is an immediate corollary of the Cohen-Serre version of Hilbert's Irreducibility Theorem (see~\cite[Theorem 1.2]{zywina2010hilbert}) since
the set in the numerator of~\eqref{equation:cohen-serre}
is a ``thin set.''

\begin{proposition} \label{proposition:applying-cohen-serre}
		For every integer $M' \ge 2$, we have
		\begin{align}
			\label{equation:cohen-serre}
			\frac{|\{u \in U(K) \cap \mc{O}^r_K : \lVert u \rVert \le B,\, \mono_{A_u}(M') \neq \mono_A(M') \} |}{ |\{ u\in U(K) \cap \mc{O}^r_K: \lVert u \rVert \le B \}| } \ll \frac{ \log B}{B^{[K:\bq]/2}},
		\end{align}
where the implied constant depends only in $U$ and $M'$.\footnote{For functions $f,g$ in the variable $B$, we say that $f(B) \ll g(B)$ if there exists a constant $c > 0$ such that $|f(B)| \leq c \cdot |g(B)|$ for all sufficiently large $B$.}
\end{proposition}
	
\subsection{Bounding the Contribution of~\eqref{equation:goursat-l-bound}} \label{subsection:applying-wallace}
	
To complete the proof of Theorem~\ref{theorem:main}, it remains to bound~\eqref{equation:goursat-l-bound}. We do this in Proposition~\ref{corollary:applying-wallace}, which relies on a strong version of Hilbert's Irreducibility Theorem due to Wallace, namely~\cite[Theorem 3.9]{scoopdedoo}. Before we can state and apply Wallace's result, we must introduce the various conditions upon which it depends. The setup detailed in~\cite[Section 3.2]{scoopdedoo} applies in a more general context than the one described below, but we specialize our discussion for the sake of brevity.

\subsubsection{Setup and Statement of~\cite[Theorem 3.9]{scoopdedoo}}
\label{subsubsection:setup-and-statement-of-wallace}

We start by introducing some notation to help us count points $u \in U(K)$ whose associated monodromy groups $\mono_{A_u}$ are not maximal. Let $B > 0$, and make the following two definitions:
		\begin{align*}
		E_\ell(B) & \defeq \{u \in U(K) : \on{Ht}(u) \le B,\, \mono_A^{\on{geom}}(\ell) \not\subset \mono_u(\ell) \}, \text{ and}\\
		E(B) &\defeq \bigcup_{\text{ prime }\ell > C} E_{\ell}(B),
	\end{align*}
where $C$ is defined as in point (b) of Section~\ref{toomanynotes}. The set $E_\ell(B)$ should be thought of as the set of exceptional points of height bounded by $B$ for the $\ell$-adic representation, and the set $E(B)$ should likewise be thought of as the set of points of height bounded by $B$ that are exceptional for some $\ell > C$.
Note in particular that for any $\ell > C$ we have $\mono_A(\ell) / \mono_A^{\on{geom}}(\ell) \simeq (\ZZ/\ell \ZZ)^\times$; this condition is important for the proof of~\cite[Theorem 3.9]{scoopdedoo} to go through, so we impose the following restriction:
\vspace*{0.1cm}
\begin{align}\label{thatswhatpeoplesay}
\text{\emph{For the rest of this section, we will maintain $\ell > C$ as a standing assumption.}}
\end{align}
For ease of notation, we redefine the set $S \subset \Sigma_K$ of ``bad'' primes, defined in Section~\ref{subsection:notation-for-families}, by adjoining to it all primes $\ell < C$.

\begin{remark}\label{thanksdavide}
Note that our definition of the exceptional set $E(B)$ differs slightly from that given in~\cite[Theorem 1.1]{scoopdedoo}, where it is defined to be the union over \emph{all} primes $\ell$ of the $\ell$-adic exceptional sets $E_{\ell}(B)$. This difference is inconsequential, as we can always deal with a finite collection of primes using Proposition~\ref{proposition:applying-cohen-serre}. Indeed, this is exactly why we replace $M$ by a multiple $M'$ divisible by all primes $\ell < C$ in the proof of Proposition~\ref{lemma:ab-cyc}.
\end{remark}

Now that we have introduced the setup needed for stating~\cite[Theorem 3.9]{scoopdedoo}, we declare the four criteria required for the theorem to be applied. For this, it will now be crucial to recall notation from the geometric setup detailed in Section~\ref{subsection:notation-for-families}.
\begin{condition} \label{definition:assumptions}
Recall from Section~\ref{subsection:notation-for-families} that $P_m$ denotes the set of primes of $\mathcal{O}_K$ dividing an integer $m$ and that $V_\ell$ denotes the connected Galois \'etale cover of $U$ giving rise to the monodromy group $H_A(\ell)$ for a prime $\ell$. In order to apply~\cite[Theorem 3.9]{scoopdedoo}, we need to verify the following geometric condition on the covers $V_\ell \to U$ as $\ell$ ranges through the primes greater than $C$:
\begin{enumerate}
	\item[\customlabel{assumption-4}{(G)}] Let $\zeta_\ell$ denote a primitive $\ell^{\mathrm{th}}$ root of unity. Each connected component of the base-change $(V_{\ell})_{K(\zeta_\ell)}$ is geometrically irreducible.
\end{enumerate}
We also need the following three asymptotic conditions concerning the monodromy groups $\mono_A(\ell)$, $\mono_A^{\on{geom}}(\ell)$, and $\mono_{A, \pp}(\ell)$ for~\cite[Theorem 3.9]{scoopdedoo} to be applied:
	\begin{enumerate}
		\item[\customlabel{property-a1}{(A1)}] There exist constants $\beta_1, \beta_2 > 0$ such that
		\[
			| \mono_A(\ell) | \ll \ell^{\beta_1} \quad \text{and} \quad | \{\text{conjugacy classes of } \mono_A(\ell) \} | \ll \ell^{\beta_2},
		\]
        where the implied constants depend only on $U$.
		\item[\customlabel{property-a2}{(A2)}] There exists a constant $\beta_3 > 0$ such that
		\[
			\geometricprimes \ell \defeq |\{ \text{prime } \mf{p} \subset \OO_K : \mf{p} \in S \cup P_\ell \text{ or } \mono_{A, \pp}^{\on{geom}}(\ell) \not\simeq \mono_A^{\on{geom}}(\ell) \} | \ll \ell^{\beta_3},
		\]
        where the implied constant depends only on $A \rightarrow U$.
		\item[\customlabel{property-a3}{(A3)}] For each $B > 0$, there exists a subset
		\[
			F(B) \subset \{u \in U(K) : \on{Ht}(u) \le B\}
		\]
		and constants $c, \gamma > 0$ depending only on $A \rightarrow U$ such that
		\begin{align*}
		\lim_{B \to \infty} \frac{|F(B)|}{|\{u \in U(K) : \on{Ht}(u) \le B\}|} = 1	\quad \text{and} \quad F(B) \cap E(B) \subset \bigcup_{\ell \le c (\log B)^\gamma} E_{ \ell}(B).
		\end{align*}
	\end{enumerate}
\end{condition}

We are now in a position to state Wallace's main result:
\begin{theorem}[\protect{\cite[Theorem 3.9]{scoopdedoo}}] \label{theorem:wallace-hit}
	Suppose that condition \ref{assumption-4} holds and that conditions \ref{property-a1}--\ref{property-a3} hold with the values $\beta_1, \beta_2, \beta_3, \gamma$.\footnote{The constant $c$ from condition~\ref{property-a3} is absorbed into the implied constant in~\eqref{thepressure'sonyoufeelityouvegotitallbelieveitthistimeforafricawhatever}.} Then
we have the following bound on the proportion of exceptional points of height bounded by $B$: 	\begin{equation}\label{thepressure'sonyoufeelityouvegotitallbelieveitthistimeforafricawhatever}
			\frac{|E(B)|}{|\{u \in U(K) : \on{Ht}(u) \le B\}|} \ll \frac{|\{u \in U(K) : \on{Ht}(u) \le B\} \setminus F(B)|}{|\{u \in U(K) : \on{Ht}(u) \le B\}|} + \frac{(\log B)^{(\beta_1 + \beta_2 + 2)\gamma +1}}{B^{1/2}},
		\end{equation}
        where the implied constant depends only on $U$.
	\end{theorem}
	
	\subsubsection{Bounding~\ref{equation:goursat-l-bound}, Conditional on Verifying \ref{assumption-4}, \ref{property-a2}, and \ref{property-a3}}

We have not yet determined that Conditions~\ref{definition:assumptions} hold in our setting. We defer the verification of these conditions to Sections~\ref{ver1},~\ref{ver2}, and~\ref{ver3}. Nevertheless, assuming that these conditions hold, we obtain the following consequence:
	\begin{proposition} \label{corollary:applying-wallace}
		Let $n>0$. Then we have
		\begin{equation}\label{cuzbabynowwe'vegotbadblood}
			\frac{\left|\left\{ u \in U(K) \cap \mc{O}_K^r : \|u \| \leq B,\, \mono_{A_u}(\ell) \not \subset \Sp_{2g}(\mathbb Z/ \ell \ZZ) \text{ for some } \ell  > C \right\}\right|}{\left|\left\{ u \in U(K) \cap \mathcal O_K^r : \|u \| \leq B \right\} \right|} \ll \left( \log B \right)^{-n},
		\end{equation}
        where the implied constant depends only on $U$ and $n$.
	\end{proposition}
	\begin{proof}[Proof assuming Propositions~\ref{proposition:verifying-assumptions},~\ref{proposition:check-B}, and~\ref{proposition:check-C}]
Note that condition~\ref{property-a1} holds trivially in our setting, because
\[
	\max\{|\mono_A(\ell)|, |\{\text{conjugacy classes of $\mono_A(\ell)$}\}|\} \leq |\GSp_{2g}(\ZZ/\ell \ZZ)|,
\]
and $|\GSp_{2g}(\ZZ/\ell \ZZ)| = O(\ell^\beta)$ for some positive constant $\beta$ depending only on $g$ because $\GSp_{2g}(\bz/\ell\bz) \subset \GL_{2g}(\bz/\ell\bz)$.

Condition \ref{assumption-4} holds by Proposition~\ref{proposition:verifying-assumptions}, and condition \ref{property-a2} holds by Proposition~\ref{proposition:check-B}. Proposition~\ref{proposition:check-C} constructs $F(B)$ that not only satisfy condition \ref{property-a3}, but also have the property that
		\begin{equation*}
			\frac{|\{u \in U(K) : \on{Ht}(u) \le B\} \setminus F(B) |}{|\{u \in U(K) : \on{Ht}(u) \le B\}|} \ll \left( \log B \right)^{-n}
            \end{equation*}
            for every $n >0$. Upon applying the argument in point (c) of Section~\ref{toomanynotes}, which relates the left-hand-sides of~\eqref{thepressure'sonyoufeelityouvegotitallbelieveitthistimeforafricawhatever} and~\eqref{cuzbabynowwe'vegotbadblood}, the proposition follows from Theorem~\ref{theorem:wallace-hit}.		
	\end{proof}

The rest of this section is devoted to verifying the conditions necessary for the proof of Proposition~\ref{corollary:applying-wallace}.
	
\subsection{Verifying Condition \ref{assumption-4}}\label{ver1}
In this section, we will consider the base-change of the setting established in~\ref{subsection:notation-for-families} from $K$ to a finite extension $L \subset \ol{K}$ of $K$; in this setting, we obtain a family $A_L \to U_L$ and a (not necessarily connected) finite Galois \'{e}tale cover $(V_\ell)_L \to U_L$. To verify condition \ref{assumption-4}, we employ the following lemma:

\begin{lemma}
	\label{lemma:geometrically-connected-components}
	Let $L \subset \ol{K}$ be a finite extension of $K$. We have that $\mono_{A_L}(m) \simeq \mono_{A_L}^{\on{geom}}(m)$
	if and only if all connected components of $(V_m)_L$ are geometrically connected over $L$.
\end{lemma}
\begin{proof}
	Observe that $(V_m)_L$ and $(V_m)_{\ol{K}}$ are finite Galois \'etale covers of $U_L$ and $U_{\ol{K}}$, which need not be connected.

	Let $W \subset (V_m)_L$ be a connected component, and let $\wt{W} \subset (V_m)_{\ol{K}}$ be a connected component mapping to $W$. By construction, $W \to U_L$ is the connected Galois \'etale cover corresponding to the surjection $\pi_1(U_L) \twoheadrightarrow \mono_{A_L}(m)$. Likewise, $\wt{W} \to U_{\ol{K}}$ corresponds to $\pi_1(U_{\ol{K}}) \twoheadrightarrow \mono_{A}^{\on{geom}}(m) = \mono_{A_L}^{\on{geom}}(m)$. This implies that:
	\begin{itemize}
		\item The degree $d_1$ of $W \to U_L$ equals $|\mono_{A_L}(m)|$.
		\item The degree $d_2$ of $\wt{W} \to U_{\ol{K}}$ equals $|\mono_{A_L}^{\on{geom}}(m)|$.
	\end{itemize}
	On the other hand, the maps $(V_m)_L \to U_L$ and $(V_m)_{\ol{K}} \to U_{\ol{K}}$ have equal degrees. Therefore $d_1 = d_2$ if and only if all connected components of $(V_m)_L$ are geometrically connected.
\end{proof}

We are now in position to prove condition \ref{assumption-4}.

\begin{proposition} \label{proposition:verifying-assumptions}
Condition \ref{assumption-4} holds in the setting of Section~\ref{subsection:notation-for-families}.
\end{proposition}
\begin{proof}
Let $L = K(\zeta_\ell)$, and recall the assumption~\eqref{thatswhatpeoplesay}. Since $(V_\ell)_L \to U_L$ is \'etale and $U_L$ is smooth over $L$, it follows that $(V_\ell)_L$ is smooth over $L$. Therefore $(V_\ell)_L$ is geometrically irreducible over $L$ if and only if it is geometrically connected over $L$. Now, by Lemma~\ref{lemma:geometrically-connected-components}, it suffices to show that $\mono_{A_L}(\ell) = \mono_A^{\on{geom}}(\ell)$.
Since we always have $\mono_{A_L}(\ell) \supset \mono_A^{\on{geom}}(\ell)$, it suffices to prove the reverse inclusion $\mono_{A_L}(\ell) \subset \mono_A^{\on{geom}}(\ell) = \Sp_{2g} (\ZZ/\ell \ZZ)$.
Since $\chi_\ell$ is trivial on $G_L = \pi_1(\spec K(\zeta_\ell))$, it follows from
Remark~\ref{remark:det-rho-is-chi}
that $\mono_{A_L}(\ell) \subset \Sp_{2g}(\mathbb Z/\ell \ZZ)$.
\end{proof}

\subsection{Verifying Condition~\ref{property-a2}}\label{ver2}

Before we carry out the verification of condition \ref{property-a2} in Proposition~\ref{proposition:check-B},
we need to introduce a modified version of the geometric setup developed in~\cite[Subsection 5.2]{zywina2010hilbert}
and in the proof of~\cite[Theorem 5.3]{zywina2010hilbert}.

\subsubsection{Geometric Setup from~\cite{zywina2010hilbert}}
\label{subsubsection:geometric-setup}
Fix the following notation: for a prime $\pp \subset \OO_K$, let $K_\pp$ be the completion of $K$ at $\pp$, let $K_\pp^{\on{un}}$ be the maximal unramified extension of $K_\pp$, let $\OO_\pp$ be the ring of integers of $K_\pp$, and let $\OO_\pp^{\on{un}}$ be the ring of integers of $K_\pp^{\on{un}}$. For a ring $R$, define $\gr R(1,r)$ to be the Grassmannian of lines in $\mathbb P^r_{R}$ and let $\mathscr L_R \subset \mathbb P^r_R \times \gr R(1,r)$ denote the universal
line over $\gr R(1,r)$. Let $Z$ and $\mathcal Z$ be as defined in Section~\ref{subsection:notation-for-families}.

We now construct a closed subscheme $\mathcal W$ of the Grassmannian parameterizing all lines whose intersections with $\mathcal Z$ are not \'etale over the base.
Define the projection $p: \mathscr L_{\mathcal O_K} \cap ( \mathcal Z \times \gr {\mathcal O_K}(1,r) ) \rightarrow \gr {\mathcal O_K}(1,r)$.
Let $\mathcal X_1$ be the open subscheme of $\mathscr L_{\mathcal O_K} \cap ( \mathcal Z \times \gr {\mathcal O_K}(1,r))$
on which $p$ is \'etale with nonempty fibers.
Define $\mathcal W \defeq  p(\mathscr L_{\mathcal O_K} \cap (\mathcal Z \times \gr {\mathcal O_K}(1,r) ) \setminus \mathcal X_1)$ with reduced subscheme structure and define
$\mathcal X \defeq \gr{\mathcal O_K}(1,r) \setminus \mathcal W$. Note that $\mathcal W$ is closed because $p$ is proper.
Considering $\mathcal W$ and $\mathcal X$ as schemes over $\mathcal O_K$, let $W$ and $X$ denote their fibers over $K$.

\begin{lemma}
	\label{lemma:nonempty-etale-locus}
	The scheme $\mathcal W$, as defined above, is a proper closed subscheme of $\gr {\mathcal O_K}(1,r)$.
\end{lemma}
\begin{proof}
It suffices to show that $\mathcal X$ is nonempty. In turn, it suffices to show $X$ is nonempty.
Since $X$ is the set of points in $\gr K(1,r)$ over which $p$ is \'etale,
by generic flatness, we need only verify that there is an open
subscheme of $\gr K(1,r)$ on which the fibers of $p_K$ are \'etale.
Since $Z$ is reduced, hence generically smooth,
and the fiber of $p_K$ over $[L]$ is identified with
$Z \cap L$, a Bertini theorem (specifically
\cite[Theoreme I.6.10(2)]{jouanolou1982theoremes}
applied to the smooth locus of $Z$ over $K$)
implies that $Z \cap L$ is indeed \'etale over $\kappa([L])$ for $[L]$ general in $\gr K(1,r)$.
\end{proof}

\begin{remark}
	\label{remark:exists-line}
By Lemma~\ref{lemma:nonempty-etale-locus}, $\mathcal W$ is a proper closed subscheme of $\gr {\mathcal O_K}(1,r)$.
	Observe that for any line $[L] \in (\gr {\mathcal O_K}(1,r) \setminus \mathcal W)(\mathbb F_\pp)$, there exists a lift $[\mathcal L] \in (\gr {\mathcal O_K}(1,r) \setminus \mathcal W)(\mathcal O_\pp)$.
The purpose of the above construction is to ensure that $\mathcal L \cap \mathcal Z_{\mathcal O_\pp}$ is \'etale over $\mathcal O_\pp$,
which we use in the proof of Proposition~\ref{proposition:check-B}.
\end{remark}

\subsubsection{Applying the Setup to Check~\ref{property-a2}}
In the following proposition, we use the Grothendieck Specialization Theorem to verify that condition \ref{property-a2} holds in our situation:
	\begin{proposition}	\label{proposition:check-B}
For a prime ideal $\pp \subset \OO_K$ let $\N(\pp)$ denote its norm and define $\nonEtalePrimes$ to be the finite set of primes over which the fiber of $\mathcal W$ is empty. Then,
 		\begin{align*}
			\geometricprimes \ell \leq |\nonEtalePrimes \cup P_\ell| + |\{\mathrm{primes}\,\,\, \pp \subset \mathcal O_K : \gcd(\N(\pp),\, | \Sp_{2g}(\mathbb Z/ \ell \ZZ)|) \neq 1 \}|.
		\end{align*}
		In particular, we have that $\geometricprimes \ell$ is bounded by a fixed power of $\ell$, so condition \ref{property-a2} holds in the setting of Section~\ref{subsection:notation-for-families}.
	\end{proposition}
\begin{remark}
	\label{remark:}
	In fact, it is true that $\geometricprimes \ell \ll \log \ell$. Apart from a finite number of primes depending only on the family $A \rightarrow U$, we need only throw out those primes whose norms are not coprime to $|\GSp_{2g}(\mathbb Z/ \ell \ZZ)|$. Since $|\GSp_{2g}(\mathbb Z/ \ell \ZZ)|$ grows polynomially in $\ell$,
	the number of distinct primes dividing $|\GSp_{2g}(\mathbb Z/ \ell \ZZ)|$
	is at most logarithmic in $\ell$.
\end{remark}
\begin{proof}[Proof of Proposition~\ref{proposition:check-B}]
Take a prime ideal $\mf{p} \notin \nonEtalePrimes \cup P_\ell$ so that $\gcd(\N(\pp), |\GSp_{2g}(\bz / \ell \bz)|) = 1$.
It suffices to show
$\mono_{A,\pp}^{\on{geom}}(\ell) = \Sp_{2g}(\mathbb Z/ \ell \ZZ) = \mono_A^{\on{geom}}(\ell).$

Choose $[\mathcal L] \in (\gr {\mathcal O_K}(1,r) \setminus \mathcal W)(\mathcal O_\pp)$, which exists by Remark~\ref{remark:exists-line}. Furthermore, define $\mathcal D \defeq \mathcal L \cap \mathcal Z_{\mathcal O_\pp}$ and $\mathcal Y \defeq \mathcal L \setminus \mathcal D$.
We have the commutative diagram
	\begin{equation}
		\label{equation:}
		\nonumber
		\begin{tikzcd}
			\mathcal Y_{\overline K} \ar {rr} \ar {dr} & & \mathcal U_{\overline K} \ar {dr} \\
			& \mathcal Y_{\mathcal O_\pp^{\text{un}}} \ar {r} & \mc U_{\OO_\pp^{\on{un}}} \ar{r} & \mathcal U \\
			\mathcal Y_{\overline {{\mathbb F}}_\pp} \ar{ur}\ar{rr} & & \mathcal U_{\overline {{\mathbb F}}_\pp} \ar{ur}
		\end{tikzcd}\end{equation}
	where all of the horizontal arrows are embeddings. 
	Let $\pi_1^{(p)}$ denote the largest prime to $p$ quotient of the fundamental group.
	Note that $\rho_{A,\ell}$ factors through $\pi_1^{(p)}(\mathcal U)$ because we are assuming
	$\gcd(\N(\pp), |\GSp_{2g}(\bz / \ell \bz)|) = 1$.
	By applying the prime to $\N(\pp)$ \'{e}tale fundamental group functor to the above diagram, we obtain
	\begin{equation}
		\label{equation:pi-1-property-3.3}
		\begin{tikzcd}
			\pi_1^{(\N(\pp))}(\mathcal Y_{\overline K}) \ar{rr}{\iota_{\overline K}} \ar{dr}{\alpha_{\overline K}} \ar{dd}{\phi} & & \pi_1^{(\N(\pp))}(\mathcal U_{\overline K}) \ar{dr}{\beta_{\overline K}}& \\
			& \pi_1^{(\N(\pp))}(\mathcal Y_{\mathcal O_\pp^{\text{un}}}) \ar {r}{\iota_{\mathcal O_\pp^{\text{un}}}} & \pi_1^{(\N(\pp))}(\mathcal U_{\OO_\pp^{\on{un}}}) \ar{r}{\beta_{\OO_\pp^{\on{un}}}} \ar{r} &  \pi_1^{(\N(\pp))}(\mathcal U) \ar{r}{\rho_{A,\ell}}& \GSp_{2g}(\mathbb Z/\ell \ZZ)\\
			\pi_1^{(\N(\pp))}(\mathcal Y_{\overline {{\mathbb F}}_\pp}) \ar{ur}{\alpha_{\overline {{\mathbb F}}_\pp}}\ar[swap]{rr}{\iota_{\overline {{\mathbb F}}_\pp}} & & \pi_1^{(\N(\pp))}(\mathcal U_{\overline {{\mathbb F}}_\pp}) \ar[swap]{ur}{\beta_{\overline {\FF}_\pp}} &
		\end{tikzcd}\end{equation}
	By Remark~\ref{remark:exists-line}, $\mathcal D$ is \'etale over $\mathcal O_\pp$. By the Grothendieck Specialization Theorem,~\cite[Th\'eor\`eme 4.4]{orgogozo2000theoreme},
there is a map
$\phi \colon \pi_1^{(\N(\pp))}(\mc{Y}_{\ol{K}}) \xrightarrow{\sim} \pi_1^{(\N(\pp))}(\mc{Y}_{\ol{K}_\pp}) \rightarrow \pi_1^{(\N(\pp))}(\mc{Y}_{\ol{{\FF}}_\pp})$
which makes the triangle on the left in~\eqref{equation:pi-1-property-3.3} commute and induces an isomorphism on the largest prime-to-$\N(\pp)$ quotients of the source and target.
Note that $\pi_1^{(\N(\pp))}(\mc{Y}_{\ol{K}}) \xrightarrow{\sim} \pi_1^{(\N(\pp))}(\mc{Y}_{\ol{K}_\pp})$ is an isomorphism by \cite[Expos\'e XIII, Proposition 4.6]{SGA1Grothendieck1971}.
Since the rest of the diagram~\eqref{equation:pi-1-property-3.3} commutes, the \mbox{entire diagram commutes.}

Now, observe that we have
\begin{align*}
(\rho_{A,\ell} \circ \beta_{\overline K})(\pi_1^{(\N(\pp))}(\mc U_{\ol{K}})) = \mono_A^{\on{geom}}(\ell) = \Sp_{2g}(\mathbb Z/ \ell \ZZ)
\intertext{where the last step follows from the assumption~\ref{thatswhatpeoplesay}. By \cite[Lemma 5.2]{zywina2010hilbert}, (since the scheme $W$ used in \cite[Lemma 5.2]{zywina2010hilbert} is contained in the scheme $W$ we have constructed above) we have that}
(\rho_{A, \ell} \circ \beta_{\overline K} \circ \iota_{\overline K})(\pi_1^{(\N(\pp))}(\mc Y_{\ol{K}})) = \mono_A^{\on{geom}}(\ell) = \Sp_{2g}(\mathbb Z/ \ell \ZZ).
\end{align*}
Since $\phi$ is an isomorphism, we deduce that
\begin{align*}
	(\rho_{A,\ell}\circ \beta_{\overline {\FF}_\pp} \circ \iota_{\overline{{\mathbb F}}_\pp})(\pi_1^{(\N(\pp))}(\mc Y_{\ol{{\mathbb F}}_\pp})) &=
	(\rho_{A,\ell}\circ \beta_{\overline {\FF}_\pp} \circ \iota_{\overline{{\mathbb F}}_\pp} \circ \phi)(\pi_1^{(\N(\pp))}(\mc Y_{\ol{K}})) \\
	&=	(\rho_{A, \ell} \circ \beta_{\overline K} \circ \iota_{\overline K})(\pi_1^{(\N(\pp))}(\mc Y_{\ol{K}})) \\
&= \Sp_{2g}(\mathbb Z/ \ell \ZZ).
\end{align*}
Therefore, $\Sp_{2g}(\mathbb Z/ \ell \ZZ) \subset (\rho_{A,\ell} \circ \beta_{\overline {\FF}_\pp})(\pi_1^{(\N(\pp))}(\mc U_{\ol{\FF}_\pp})) = \mono_{A,\pp}^{\on{geom}}(\ell)$.
Since $\ell \nmid \N(\pp)$, we have that $\overline {{\mathbb F}}_\pp$ contains nontrivial $\ell^{\mathrm{th}}$ roots of unity. Thus, the
mod-$\ell$ cyclotomic character is trivial on $\pi_1^{(\N(\pp))}(\mathcal U_{\overline {{\mathbb F}}_\pp})$,
and so $\Sp_{2g}(\mathbb Z/ \ell \ZZ) \supset \mono_{A,\pp}^{\on{geom}}(\ell).$
Hence, we have that
\[
	\mono_{A,\pp}^{\on{geom}}(\ell) = \Sp_{2g}(\mathbb Z/ \ell \ZZ) = \mono_A^{\on{geom}}(\ell). \qedhere
\]
\end{proof}

\subsection{Verifying Condition \ref{property-a3}}\label{ver3} \label{subsection:lombardo}
It remains to check that condition \ref{property-a3} is satisfied in our setting. As usual, before carrying out the argument, we must fix some notation. Let $\Sigma_K$ denote the set of nonzero prime ideals of $\OO_K$, and for a prime $\pp \in \Sigma_K$ of good reduction, let $\frob_\pp \in G_K$ denote a corresponding Frobenius element.

Given a PPAV $A/K$, let $\charpoly_A(\frob_\pp)$ denote the characteristic polynomial of $\rho_A(\frob_\pp) \in \GSp_{2g}(\wh{\ZZ})$, and observe that $\charpoly_A(\frob_\pp)$ has coefficients in $\ZZ$.
Finally, let $h(A)$ denote the absolute logarithmic Faltings height of $A$, obtained by passing to any field extension over which $A$ \mbox{has semi-stable reduction.}

\subsubsection{Applying Lombardo's Result}

The key input for our proof of this condition is the following theorem of Lombardo, which is an effective version of the Open Image Theorem:

\begin{theorem}[\protect{\cite[Theorem 1.2]{lombardo2015explicit} and Proposition~\ref{prop:Main} in Appendix~\ref{lombardstreet}}]
\label{theorem:lombardo}
		Let $A / K$ be a PPAV of dimension $g \ge 2$. Suppose that we have the following two conditions:
		\begin{enumerate}[(1)]
			\item $\End_{\ol{K}}(A) = \bz$.
			\item There exists a prime $\mf{p} \in \Sigma_K$ at which $A$ has good reduction and such that the splitting field of $\charpoly_A(\frob_{\mf{p}})$ has Galois group isomorphic to $\gee$.
		\end{enumerate}
		Then there are constants $c_1, c_2 > 0$ and $\gamma_1, \gamma_2$,
		depending only on $g$ and $K$, for which the following statement is true: For every prime $\ell$ unramified in $K$ and strictly larger than
		\[
		\max \{ c_1 (\N (\mf{p}))^{\gamma_1}, c_2 (h(A))^{\gamma_2}\},
		\]
		the $\ell$-adic Galois representation surjects onto $\GSp_{2g}(\bz_\ell)$.
	\end{theorem}

	\begin{remark}\label{remark:gee}
		The group structure of $\gee$ is defined by how $S_g$ acts on $(\bz / 2\bz)^g$, namely by permuting the $g$ factors. This group appears because it is the largest possible Galois group of a reciprocal polynomial, by which we mean a polynomial $P(T)$ satisfying $P(T) = P(1/T) \cdot T^{\deg P}$.
	\end{remark}
	
Now, the proof of condition \ref{property-a3} will follow from Theorem~\ref{theorem:lombardo} once we know that the two hypotheses of Theorem~\ref{theorem:lombardo} hold for a density-$1$ subset of the $K$-valued points of the family. We shall first check condition \ref{property-a3} under the assumption that these hypotheses hold most of the time.
To this end, it will be convenient to introduce notation to help us count the points that fail to satisfy one of the hypotheses in Theorem~\ref{theorem:lombardo}. For a given family $A \to U$, define the following two sets:
\begin{align*}
D_1(B) & \defeq \{u \in U(K) : \on{Ht}(u) \le B,\,\text{$A_u$ fails hypothesis (1)}\}, \text{ and} \\
D_2(B) & \defeq \{u \in U(K) : \on{Ht}(u) \le B,\, \text{$A_u$ fails hypothesis (2) for all $\mf{p}$ with $\N(\mf{p}) \le (\log B)^{n+1}$}\}.
\end{align*}
In the next proposition, we verify condition~\ref{property-a3}, conditional upon the assumptions that sets $D_1(B)$ and $D_2(B)$ are sufficiently small (these assumptions are proven in Lemma~\ref{lemma:davide-1} and Lemma~\ref{lemma:davide-2} respectively):
\begin{proposition} \label{proposition:check-C}
Let $n>0$. There are constants $c, \gamma$ depending only on $U$ such that the following holds: if we define
		\[
		F(B) \defeq \{u \in U(K) : \on{Ht}(u) \le B,\, \mono_{A_u}(\ell) \supset \Sp_{2g}(\bz / \ell \ZZ) \text{ for all } \ell > c(\log B)^\gamma\},
		\]
		then we have
		\begin{equation} \label{yousayyourefineIknowyoubetterthanthat}
			\frac{|F(B)|}{|\{u \in U(K) : \on{Ht}(u) \le B\}|} = 1 + O\left(\left( \log B \right)^{-n}\right),
		\end{equation}
        where the implied constant depends only on $U$ and $n$.
	\end{proposition}
	\begin{proof}[Proof assuming Lemma~\ref{lemma:davide-1}, Lemma~\ref{lemma:davide-2}, and Lemma~\ref{lemma:height}]

		Let $c_1, c_2$ and $\gamma_1, \gamma_2$ be as in Theorem~\ref{theorem:lombardo}. There exist constants $c_2', \gamma_2'$, chosen appropriately in terms of the constants $c_0, d_0$ provided by Lemma~\ref{lemma:height}, such that the following holds: for $u \in U(K)$ with $\on{Ht}(u) > B_0$, where $B_0$ is a positive constant depending only on $U$, we have that
		\[
			c_2(h(A_u))^{\gamma_2} \le c_2'(\log \on{Ht}(u))^{\gamma_2'}.
		\]
		The requirement that $\on{Ht}(u)$ be sufficiently large is insignificant because
		\begin{equation}\label{ifyoucouldseethatimtheoonewhounderstandsyoubeenhereallalongsowhycantyouseeeeyoubelongwithmeeee}
			\frac{|\{ u \in U(K): \on{Ht}(u) \le B_0\}|}{|\{ u \in U(K): \on{Ht}(u) \le B\}|} \ll \frac{1}{B^{[K : \bq](r+1)}},
		\end{equation}
		and the right-hand-side of~\eqref{ifyoucouldseethatimtheoonewhounderstandsyoubeenhereallalongsowhycantyouseeeeyoubelongwithmeeee} is dominated by the right-hand-side of~\eqref{yousayyourefineIknowyoubetterthanthat}.
If we take
		\[
			c = \max(c_1, c_2') \quad \text{and} \quad \gamma = \max((n+1)\gamma_1, \gamma_2'),
		\]
		Theorem~\ref{theorem:lombardo} tells us that
		\[
			\{u \in U(K) : \on{Ht}(u) \le B\} \setminus F(B) \subset D_1(B) \cup D_2(B).
		\]
		The desired result follows from Lemmas~\ref{lemma:davide-1} and \ref{lemma:davide-2}, from which we deduce that
		\[
			\frac{|D_1(B) \cup D_2(B)|}{|\{u \in U(K) : \on{Ht}(u) \le B\}|} \ll \left( \log B \right)^{-n}. \qedhere
		\]
	\end{proof}

In what follows, we prove the results upon which the above proof of Proposition~\ref{proposition:check-C} depends. To begin with, we check that hypotheses (1) and (2) from Theorem~\ref{theorem:lombardo} hold in our setting by bounding $D_1$ in Lemma~\ref{lemma:davide-1} (thus verifying hypothesis (1)) and bounding $D_2$ in Lemma~\ref{lemma:davide-2} (thus verifying hypothesis (2)).

\subsubsection{Verifying Hypothesis $(1)$}
We check that hypothesis (1) holds in our setting via the following lemma:
	\begin{lemma}\label{lemma:davide-1}
		 We have that
		 \begin{align}
			 \label{equation:bound-d1}
			\frac{|D_1(B)|}{|\{u \in U(K) : \on{Ht}(u) \le B\}|} \ll \frac{\log B}{B^{[K:\mathbb Q]/2}},
		 \end{align}
		
        where the implied constant depends only on $U$.
	\end{lemma}
	\begin{proof}
Choose $\ell > \max\{C, \ell_1(g)\}$, where $C$ is defined in~\eqref{thatswhatpeoplesay} and $\ell_1(g)$ is the constant, depending only on the dimension $g$, given in
\cite[Proposition 4]{ellenbergEHK:non-simple-abelian-varieties-in-a-family}.
By
\cite[Proposition 4]{ellenbergEHK:non-simple-abelian-varieties-in-a-family}, we have that $|D_1(B)|$ is bounded above by $\left| \left\{ u \in U(K) : \mono_{A_u}(\ell) \supset \Sp_{2g}(\mathbb Z/\ell \mathbb Z) \right\}\right|.$
The lemma then follows from
Proposition~\ref{proposition:applying-cohen-serre},
where we are using point (c) of Section~\ref{toomanynotes} to pass from lattice points to $K$-valued points.
	\end{proof}

\subsubsection{Verifying Hypothesis $(2)$}

We complete the verification of hypothesis $(2)$ in Lemma~\ref{lemma:davide-2} by means of an argument involving the large sieve, which lets one bound a set in terms of its reduction modulo primes. The large sieve is stated as follows:
\begin{theorem}[Large Sieve,~{\cite[Theorem 4.1]{zywina2010elliptic}}] \label{theorem:large-sieve}
Let $\lVert - \rVert$ be a norm on $\br \otimes_{\bz} \mc{O}_{K}^r$, and fix a subset $Y \subset \mc{O}_K^r$. Let $B \ge 1$ and $Q > 0$ be real numbers, and for every prime ${\mf{p}} \in \Sigma_K$, let $0 \le \omega_{\mf{p}} < 1$ be a real number. Suppose that we have the following two conditions:
	\begin{enumerate}
		\item The image of $Y$ in $\br \otimes_{\bz} \mc{O}_{K}^r$ is contained in a ball of radius $B$.
		\item For every ${\mf{p}} \in \Sigma_K$ with $\N (\mf{p}) < Q$, we have
		$
			|Y_{\mf{p}}| \le (1 - \omega_{\mf{p}})\cdot \N (\mf{p})^r
		$,
		where $Y_{\mf{p}}$ is the image of $Y$ under reduction modulo $\mf{p}$.
	\end{enumerate}
	Then we have that
	\[
		|Y| \ll \frac{B^{[K:\bq]r} + Q^{2r}}{L(Q)}, \quad \text{where} \quad L(Q) \defeq \sum_{\substack{\mf{a} \subset \mc{O}_K\,\,\, \mathrm{ squarefree} \\[0.1cm] \N (\mf{a}) \le Q}}\,\,\, \prod_{\mathrm{prime}\,\,\,\mf{p} | \mf{a}} \frac{\omega_{\mf{p}}}{1 - \omega_{\mf{p}}},
	\]
    and the implied constant depends only on $K$, $r$, and $||-||$.
\end{theorem}
We must now specialize the abstract setup in Theorem~\ref{theorem:large-sieve} to our setting. To do so, we define the various objects at play in the large sieve as follows:
\begin{definition}
	\label{definition:sieve-sets}
Introduce the following notation:
    \begin{itemize}[leftmargin=0.2in]
    \item Let $||-||$ be the norm defined in Section~\ref{weaintevergonnaberoyals}.
    \item Let $B \geq 1$, take $Q \defeq (\log B)^{n+1}$.
    \item Let $m$ be the positive integer produced by Proposition~\ref{proposition:good-cover}, let $\zeta_m$ denote a primitive $m^{\mathrm{th}}$ root of unity, and let $\Sigma^m_K \subset \Sigma_K$ be the set of ${\mf{p}} \in \Sigma_K$ which split completely in $K(\zeta_m)$.
Now, with $\galSlope, \galMin$ as in Lemma~\ref{lemma:bounding-Y-reduction}, we may take $\omega_{\mf{p}} = \galSlope$ for all $\mf{p} \in \Sigma^m_K$ with $\N(\mf{p}) > \galMin$ and $\omega_{\mf{p}} = 0$ for all other ${\mf{p}} \in \Sigma_K$.
    \item We take $Y$ to be the following set:
    $$Y \defeq \{u \in U(K) \cap \OO_K^r : ||u|| \le B,\, \text{$A_u$ fails hypothesis (2) for all $\mf{p}$ with $\N(\mf{p}) \le (\log B)^{n+1}$}\}.$$
	As above, $Y_\pp$ denotes the mod-$\pp$ reduction of $Y$.
\item Define $T_{\pp}$ by
	$$\goodGaloisSet_{\mf{p}} \defeq \{x \in \mc{U}_{\mathbb F_\pp} : \text{splitting field of }\charpoly_A(\frob_\pp) \text{ has Galois group } \gee \}.$$
    The motivation for defining $T_{\mf{p}}$ is that its complement contains $Y_\pp$.
    \end{itemize}
\end{definition}

To ensure that the choices made in Definition~\ref{definition:sieve-sets} are suitable, we must prove Proposition~\ref{proposition:good-cover} and Lemma~\ref{lemma:bounding-Y-reduction}, which when taken together assert that there exist a positive integer $m$ and $\sigma, \tau> 0$ so that $|Y_\pp| \leq \left( 1-\galSlope \right) \cdot \N (\pp)^r$ for all $\pp \in \Sigma_K^m$. However, the proof of this result is rather laborious, and stating it now would serve to distract the reader from the primary thrust of the argument. We therefore defer the proof of Lemma~\ref{lemma:bounding-Y-reduction} to Section~\ref{thoughtweweregoinstrong}, and conditional upon this, we now use the large sieve to check that hypothesis (2) \mbox{holds in our setting.}

\begin{proposition}\label{lemma:davide-2}
		For $n >0$, we have that
		\[
			\frac{|D_2(B)|}{|\{u \in U(K) : \on{Ht}(u) \le B\}|} \ll (\log B)^{-n}.
		\]
	\end{proposition}
	\begin{proof}[Proof assuming Proposition~\ref{proposition:good-cover} and Lemma~\ref{lemma:bounding-Y-reduction}]
Theorem~\ref{theorem:large-sieve} yields the estimate
\begin{align*}
|Y| & \ll \frac{B^{[K : \bq]r} + (\log B)^{2n(n+1)}}{L((\log B)^{n+1})}, \\
\intertext{whose denominator is bounded below by}
L((\log B)^n) &> \sum_{\substack{{\mf{p}} \in \Sigma^m_K \\ \tau < \N(\mf{p}) < (\log B)^{n+1}}} \frac{\galSlope}{1 - \galSlope} \\
&> \galSlope \cdot |\{{\mf{p}} \in \Sigma^m_K : \tau < \N(\mf{p}) \le (\log B)^{n+1} \}|.
\end{align*}
Applying the Chebotarev Density Theorem yields that
\begin{align*}
|\{{\mf{p}} \in \Sigma^m_K : \tau < \N ({\mf{p}}) \le (\log B)^{n+1} \}| &\gg |\{{\mf{p}} \in \Sigma_K : \tau < \N ({\mf{p}}) \le (\log B)^{n+1} \}|. \\
\intertext{Applying the Prime Number Theorem yields that}
|\{{\mf{p}} \in \Sigma_K : \tau < \N ({\mf{p}}) \le (\log B)^{n+1} \}|&\gg \frac{(\log B)^{n+1}}{\log ((\log B)^{n+1})}.
\end{align*}
Combining the above estimates, we deduce that
\begin{align*}
\frac{|Y|}{|\{u \in U(K) \cap \OO_K^r : ||u|| \le B\}|}
& \ll \frac{B^{[K:\mathbb Q]r} + (\log B)^{2n(n+1)}}{\frac{(\log B)^{n+1}}{\log ((\log B)^{n+1})}} \cdot \frac{1}{B^{[K:\mathbb Q]r}} \\
& \ll\frac{\log (( \log B )^{n+1})}{(\log B)^{n+1}}
\ll (\log B)^{-n}.
\end{align*}
Finally, employing point (c) of Section~\ref{toomanynotes} to translate the above estimate from lattice points to $K$-valued points yields the desired result.
	\end{proof}

\subsubsection{Validating the Sieve Setup}\label{thoughtweweregoinstrong}

This section is devoted to proving Proposition~\ref{proposition:good-cover} and Lemma~\ref{lemma:bounding-Y-reduction}, which together verify that the sieve setup introduced in Definition~\ref{definition:sieve-sets} satisfies the necessary conditions for applying the large sieve as we did in the proof of Proposition~\ref{lemma:davide-2}. We start by constructing the value of $m$ that we use in our application of the large sieve:

	\begin{proposition}\label{proposition:good-cover}
		There is a positive integer $m$ and a subset $\mc C \subset \Sp_{2g}(\bz / m \bz)$ invariant under conjugation in $\Sp_{2g}(\bz / m \bz)$, and hence in $\GSp_{2g}(\bz / m \bz)$, such that the following holds:
		\begin{enumerate}
			\item We have $\mono_A(m) = \GSp_{2g}(\bz /  m \ZZ)$ and $\mono_A^{\on{geom}}(m) = \Sp_{2g}(\bz /  m \ZZ)$.
			\item For any $\mf{p} \notin S$ and any closed point $x \in \mc{U}_{\mathbb F\pp}$, if $\rho_{A, m}(\frob_x) \in \mc C$, then the splitting field of $\charpoly(\frob_x)$ has Galois group $\gee$.\footnote{For the definition of $S$, see the sentence immediately preceding Remark~\ref{thanksdavide}.}
		\end{enumerate}
	\end{proposition}
Note that it is easy to construct many $m$ satisfying (a) by the big monodromy hypothesis. The main point of this proposition is to show there is an $m$
which also satisfies (b).
	\begin{proof}
We construct the desired $m$ as a product of four appropriate primes,
depending on the family $A \to U$.
By, for example, Hilbert irreducibility, (or more precisely \cite[\S 9.2, Proposition 1]{serre1989lectures} in conjunction with \cite[\S 13.1, Theorem 3]{serre1989lectures} applied to the extension $\mathbb Q(x_1, \ldots, x_g)[T]/(T^{2g} + \sum_{i=1}^{g-1} (-1)^i x_i (T^{2g-i} + T^i) + (-1)^g x_g T^g + 1)$ over $\mathbb Q(x_1, \ldots, x_g),$) there exists a degree-$2g$ polynomial $P(T) \in \bz[T]$ satisfying $P(T) = P(1/T) \cdot T^{\deg P}$ with Galois group $\gee$.
		It is easy to exhibit elements of $\gee$ whose left-action on $\gee$ is described by one of the following four cycle types:
		\begin{equation} \label{equation:splittings}
			\begin{array}{c}
				2 + 1 + \cdots + 1, \\
				4 + 1 + \cdots + 1, \\
				(2g-2) + 1 + 1, \\
				2g.
			\end{array}
		\end{equation}
We choose these cycle types because any subgroup of $\gee$ containing an element with each of these cycle types is in fact all of $\gee$ by \cite[Lemma 7.1]{kowalski2006large}.
		For each such partition, the Chebotarev Density Theorem tells us that there are infinitely many primes $\ell$ such that $P(T) \pmod{\ell}$ splits according to the chosen partition. For $\ell > C$ we have $\rho_{A, \ell}(\pi_1(U)) = \GSp_{2g}(\bz /  \ell \ZZ)$ and $\rho_{A, \ell}(\pi_1(U_{\ol{K}})) = \Sp_{2g}(\bz /  \ell \ZZ)$.
So, for $i \in \left\{ 1,2,3,4 \right\}$ we can find $\ell_i > C$
so that $P(T) \pmod{\ell_i}$ splits according to the $i^{\mathrm{th}}$ partition above. 		
By the Chinese remainder theorem, $(a)$ holds.

To complete the proof, we construct $\mc C$ and verify $(b)$.
Since characteristic polynomials are conjugacy-invariant, the set
		\[
			\mc C \defeq \{M' \in \GSp_{2g}(\bz / m \ZZ) : \charpoly(M')\ (\on{mod}\, \ell_i) \text{ splits as in \eqref{equation:splittings} for all $i \in \left\{ 1,2,3,4 \right\}$}\}
		\]
		is a union of conjugacy classes of $\GSp_{2g}(\bz /  m \ZZ)$. By \cite[Theorem A.1]{rivin2008walks} there exists an $M \in \Sp_{2g}(\bz)$ such that $\charpoly(M)(T) = P(T)$, which shows that $\mc C$ is nonempty. For this choice of $\mc C$, conclusion (2) follows from \cite[Lemma 7.1]{kowalski2006large}, which says that any subgroup of $\gee$ that contains elements realizing all four cycle types in \eqref{equation:splittings} must actually equal all of $\gee$.
	\end{proof}
    The reason why we constructed $m$ in Proposition~\ref{proposition:good-cover} in the way that we did is that it allows us to apply the following theorem, which is a crucial tool for bounding the set of Frobenius elements with certain Galois groups modulo each prime.

\begin{theorem}[\protect{\cite[Lemma 1.2]{ekedahl1988effective}}] \label{theorem:ekedahl}
	Let $X$ be a scheme, and let $\pi\colon X \ra \spec \OO_K$ be a morphism of finite type. Let $\phi\colon Y \ra X$ be a connected finite Galois \'{e}tale cover with Galois group $G$, and let $\rho \colon \pi_1(X) \to G$ denote the corresponding finite quotient. Suppose that $\pi \circ \phi$ has a geometrically irreducible generic fiber, and let $\mc C$ be a conjugacy-invariant subset of $G$. For every $\pp \in \Sigma_K$, we have
	\[
		\frac{\lvert \{ x \in X({\mathbb F_\pp}) :  \rho(\frob_x) \in \mc C\} \rvert }{\lvert X({\mathbb F_\pp}) \rvert} = \frac{|\mc C|}{|G|} + O((\N (\mf{p}))^{-1/2}),
	\]
with implicit constants depending only on the family $Y \rightarrow X$. By $\frob_x$ we mean the Frobenius element in $\pi_1(X)$ corresponding to $x \in X$.
\end{theorem}

We now apply Theorem~\ref{theorem:ekedahl} to the
	conjugacy-invariant set $\mc C$ from Proposition~\ref{proposition:good-cover} in order to obtain a lower bound on $|T_\pp|$, the number of points $u \in U(K)$ with the splitting field of $\charpoly_{A_u}(\frob_\pp)$ having Galois group equal to $\gee$.
\begin{proposition}\label{proposition:good-v}
	As $\mf{p}$ ranges through the elements of $\Sigma_K^m$, where $m$ is defined as in Proposition~\ref{proposition:good-cover}, we have that
	\(
	|\goodGaloisSet_{\mf{p}}| \gg (\N ({\mf{p}}))^r.
	\)
\end{proposition}
\begin{proof}
 Let $L \defeq K(\zeta_m)$. As in Section~\ref{subsection:notation-for-families}, let $\mc{V}_m \to \mc{U}_{\OO_{P_m}}$ be the connected Galois \'etale cover associated to the mod-$m$ Galois representation $\rho \colon \pi_1(\mc{U}_{\OO_{P_m}}) \to \GSp_{2g}(\bz / m \bz)$, and
let $\mc{X}$ be one of the connected components of $(\mc{V}_m)_{L}$. The map $\mc{X} \to (\mc{U}_{\OO_{P_m}})_{L}$ is the connected Galois \'etale cover associated to the map
	\[
		\begin{tikzcd}
			\rho' \colon \pi_1((\mc{U}_{\OO_{P_m}})_{L}) \ar{r} & \pi_1(\mc{U}_{\OO_{P_m}}) \ar{r}{\rho} & \GSp_{2g}(\bz / m \bz);
		\end{tikzcd}
	\]
	note that the image of this composite map equals $\rho(\pi_1(\mc{U}_{\OO_{P_m}})) \cap \Sp_{2g}(\bz / m \bz)$ by Remark~\ref{remark:det-rho-is-chi},
	since $\chi_m$ is trivial on $K(\zeta_m)$. By Proposition~\ref{proposition:good-cover}(a), we have $\rho(\pi_1(\mc{U}_{\OO_{P_m}})) = \GSp_{2g}(\bz / m \bz)$, so we conclude that $\rho'(\pi_1((\mc{U}_{\OO_{P_m}})_{L})) = \Sp_{2g}(\bz / m \bz)$.
	
	We seek to apply Theorem~\ref{theorem:ekedahl} with
	\[
		\mc{X} \to (\mc{U}_{\OO_{P_m}})_{L} \to \spec \mc{O}_{L} \quad \text{in place of} \quad Y \to X \to \spec \mc{O}_K.
	\]
	To do so, we must check that this composition has geometrically irreducible generic fiber, which follows from the second part of Proposition~\ref{proposition:good-cover}(a) in conjunction with Lemma~\ref{lemma:geometrically-connected-components}.
	
	Now let $\mc C \subset \Sp_{2g}(\bz / m \bz)$ be as in Proposition~\ref{proposition:good-cover}(b). For any $\mf{p} \in \Sigma^m_K \setminus S$
	and $\pp' \in \Sigma_{L}$ lying over $\pp$, we have
$(\mc{U}_{L})_{\FF_{\pp'}} \simeq \mc{U}_{\FF_\pp}$, and so there
is a bijection between
	\begin{align*}
		\{x \in \mc{U}_{L}({\FF_{\pp'}}) : \rho'(\frob_x) \in \mc C\}
		\quad \text{and} \quad \{x \in \mc{U}({\FF_{\pp}}) : \rho(\frob_x) \in \mc C\}.
	\end{align*}
	By Proposition~\ref{proposition:good-cover}(b), $\goodGaloisSet_{\mf{p}}$ contains the latter set, so we have
	\begin{align*}
		|\goodGaloisSet_{\mf{p}}| \ge |\{x \in \mc{U}({\FF_{\pp'}}) : \rho(\frob_x) \in \mc C\}| &= |\{x \in \mc{U}_{L}({\FF_{\pp'}}) : \rho'(\frob_x) \in \mc C\}| \\
		&= \left( \frac{|\mc C|}{|G|} + O((\N(\mf{p}'))^{-1/2}) \right) \cdot |\mc{U}_{L}({\FF_{\pp'}})|,
	\end{align*}
	where the last step above follows from Theorem~\ref{theorem:ekedahl}. Now, we have the estimate
    \[
		|\mc{U}_{L}({\mathbb F_{\mf{p}'}})| \gg (\N(\mf{p}'))^r,
	\]
	because the complement of $(\mc{U}_{L})_{\mathbb F_{\pp'}}$ in $(\bp^r_{\mc{O}_{L}})_{\mathbb F_{\pp'}}$ has codimension at least 1, since $\pp \notin S$. Combining our results, and using that $S$ is a finite set, we find that
	\begin{align*}
		|T_\pp| \geq \left( \frac{|\mc C|}{|G|} + O\left(\N(\mf{p}')^{-1/2}\right) \right) \cdot |\mc{U}_{L}({\FF_{\pp'}})| &\gg \N(\mf{p}')^r = \N(\mf{p})^r. \qedhere
	\end{align*}
\end{proof}

The following lemma completes our verification of the sieve setup by constructing the necessary constants $\sigma, \tau$.

\begin{lemma}
	\label{lemma:bounding-Y-reduction}
	There are constants $\galSlope,\galMin > 0$ so that for all $\pp \in \Sigma^m_K$ with $\N(\pp) > \tau$, we have
	$|Y_\pp| \leq \left( 1-\galSlope \right) \cdot \N (\pp)^r.$
\end{lemma}
\begin{proof}
By Proposition~\ref{proposition:good-v}, there are constants $\galSlope', \galMin' > 0$ so that, for all $\pp \in \Sigma^m_K$ with $\N(\pp) > \galMin'$,
we have $|\goodGaloisSet_{\mf{p}}| \geq \galSlope' \cdot(\N(\pp))^r$.
For such $\mf{p}$, we have that
\[
	|Y_{\mf{p}}| \le (1 - \galSlope') \cdot (\N(\mf{p}))^r + O((\N(\mf{p}))^{r-1}),
\]
where the error term is on order of $\N(\pp)$ smaller than the main term because $\mc{Z}$ has codimension at least $1$ in $\bp^r_{\mc{O}_K}$. By replacing $\galSlope'$ with a slightly smaller $\galSlope$ and $\galMin'$ with a slightly larger $\galMin$, we may write
\[
	|Y_{\mf{p}}| \le (1 - \galSlope) \cdot (\N(\mf{p}))^r. \hfill \qedhere
\]
\end{proof}

\subsubsection{Discussion of Heights}
\label{subsubsection:appendix-on-heights}
In this section, we prove a result that describes the relationship between the absolute multiplicative height on projective space and the absolute logarithmic Faltings height. Let $\on{Ht}$ be the height on $\bp_K^r$ as defined in~\ref{weaintevergonnaberoyals}, and let $h$ be the Faltings height. Let $\log\on{Ht}$ be the absolute logarithmic height on $\bp^r(\overline{K})$, and note that $\log\on{Ht}$ naturally restricts to a logarithmic height function defined on the open subscheme $U \subset \bp_K^r$.

Let $\ag$ be the moduli stack of $g$-dimensional PPAVs, and let $p \colon \ug \ra\ag$ be the universal family of abelian varieties. Let $\pi\colon \ag\ra \coarseag$ be its coarse moduli space, and let $j(A)\in \coarseag(K)$ be the closed point represented by $A$. As in \cite[Section 2]{FalFinite}, we choose $n\in\bn$ such that the line bundle $\mathscr{L}=((\pi \circ p)_*\omega_{\ug/\ag})^{\otimes n}$ is very ample, where $\omega_{\ug/\ag}$ is the canonical sheaf of $p \colon \ug \ra \ag$. Fix an embedding $i\colon\coarseag\hookrightarrow\bp^N$ with $i^*\so_{\mathbb P^N}(1) \simeq \mathscr{L}$. The modular height $\log \on{Ht}(j(A))$ of $A$ is then the restriction along $i$ of the absolute logarithmic height (i.e.,~the absolute logarithmic height of $j(A)$ considered as a point of $\bp^N(K)$). On the other hand, $\so_{\mathbb P^N}(1)$ is a metrized line bundle and restricts to give a metric on $\mathscr{L}$ \cite[p.~36]{FalRP}; we denote by $\log\on{Ht}_{\mathscr{L}}$ the corresponding height function on $\coarseag$.

We now relate the height on projective space and the Faltings height by piecing together results from the literature on heights:

\begin{lemma} \label{lemma: height-0}
Let $g$ be a positive integer, $K$ a number field, and let $n\in \bn$ be as in the definition of the modular height. Then there exist constants $\alpha$ and $\beta$ such that for every principally polarized abelian variety $A$ over $K$, we have
\[|n\cdot h(A)-\log\on{Ht}(j(A))|\leq \alpha \cdot \log\max\{1, \log\on{Ht}(j(A))\}+\beta.\]
\end{lemma}
\begin{proof}
By~\cite[Proof of Lemma 3]{FalFinite}, there exist constants $\alpha_1$ and $\beta_1$ such that for all abelian varieties $A/K$, we have
\[|n\cdot h(A)-\log\on{Ht}_{\mathscr{L}}(j(A))|\leq \alpha_1 \cdot \log(\log\on{Ht}_{\mathscr{L}}(j(A)))+\beta_1.\]
By \cite[B.3.2(b)]{afraidofheights}, there is a constant $\beta_2$ such that
\[|\log\on{Ht}_{\mathscr{L}}(j(A))-\log\on{Ht}(j(A))|\leq\beta_2. \qedhere \]
\end{proof}

\begin{lemma} \label{lemma:height}
		There exist constants $c_0$ and $d_0$ depending only on $A \to U$ such that
		\[
			h(A_u) \le c_0 \log \on{Ht}(u) + d_0
		\]
		for all $u \in U(K)$.
	\end{lemma}
	\begin{proof}
	By \cite[p.~19, Section 2.6, Theorem]{serre1989lectures},
$\on{Ht}(j(A_u)) \ll \on{Ht}(u)$ and $\on{Ht}(u) \ll \on{Ht}(j(A_u))$ for all $u \in U$. The result then follows from Lemma~\ref{lemma: height-0}.
    \end{proof}

\section{Applications of Theorem~\ref{theorem:main}}\label{iknewyouweretrouble}

The purpose of this section is to demonstrate that the main result, Theorem~\ref{theorem:main}, can be applied to a number of interesting families of PPAVs, such as families containing a dense open substack of the locus of Jacobians of hyperelliptic curves, trigonal curves, or plane curves. In Section~\ref{crit}, we prove a general tool that is needed to guarantee big monodromy for the loci in our applications, and in Section~\ref{androidbeatsios}, we examine each of these \mbox{applications in detail.}

\subsection{Finite-Index Criterion}\label{crit}

In this section we prove Proposition~\ref{proposition:stacky-dominant-map-surjective-fundamental-group}, which will be applied
in the setting of Theorem~\ref{theorem:main} to determine
that $U$ has big monodromy when its image in the moduli stack of
abelian varieties has big monodromy. We begin by recalling an elementary criterion giving surjectivity for the map on \'{e}tale fundamental groups induced by a morphism of Deligne-Mumford stacks. By {\em Deligne-Mumford stack}, we mean a stack in the \'etale topology with representable diagonal (i.e., representable by algebraic spaces),
which has an \'etale surjective morphism from a scheme.
For a general reference on stacks, see~\cite{olsson2016algebraic} or~\cite{laumon-bailly:champs-algebriques}; also, see~\cite{stacks-project} for a more comprehensive reference.

\begin{lemma}
       \label{lemma:surjectivity-criterion-pi1}
       Suppose $f\colon X \rightarrow Y$ is a map of Deligne-Mumford stacks.
       The fiber product $U \times_Y X$ is connected
       for all finite connected \'etale maps $U \rightarrow Y$
       if and only if
       the induced map $\pi_1(X) \rightarrow \pi_1(Y)$ is surjective.
       In particular, if $X$ and $Y$ are normal, integral, and Noetherian, and $f: X \rightarrow Y$ is a flat map with connected geometric
       generic fiber, then the induced map $\pi_1(X) \rightarrow \pi_1(Y)$ is surjective.
\end{lemma}
\begin{proof}
The first part holds in greater generality as a statement about Galois categories; see~\cite[\href{https://stacks.math.columbia.edu/tag/0BN6}{Tag 0BN6}]{stacks-project}.
As for the second part,
we only need verify that a connected finite \'etale cover $U \rightarrow Y$ pulls back to a connected cover of $X$.
Note that because $X$ and $Y$ are normal and integral, \'etale covers of $X$ and $Y$ are connected if and only if they are irreducible. Here, we are using that
normal and connected implies irreducible and that
normality is local in the \'etale topology over Noetherian stacks.
To see why normality is local in the \'etale topology over a Deligne-Mumford
stack, 
note first that normality is local in the \'etale topology
over any base scheme by
~\cite[\href{http://stacks.math.columbia.edu/tag/03E7}{Tag 03E7}]{stacks-project}.
Using this, one defines a Deligne-Mumford
stack to be normal if any \'etale cover by a scheme is normal.
From this definition,
it follows that
normality of a Deligne-Mumford stack is equivalent to normality of any \'etale cover.

Thus, we only need show that if $U \rightarrow Y$ is any irreducible finite \'etale cover, then so is $X \times_Y U \rightarrow X$.
But this follows from the assumptions that $f$ is flat and $U$ is integral,
which implies all generic points of $X \times_Y U$ map to the generic point of $U$.
So, if $X \times_Y U$ were reducible, the geometric generic fiber over $U$ would also be
reducible, which contradicts the assumption that $f$ has connected geometric generic fiber, since a geometric generic fiber of
$X \times_Y U$ is also a geometric generic fiber of $f$.
\end{proof}

\begin{proposition} \label{proposition:stacky-dominant-map-surjective-fundamental-group}
Let $k$ be an arbitrary field of characteristic $0$. Suppose $X$ is a scheme and $Y$ is a Deligne-Mumford stack over $k$, both of which are normal, integral, separated, and finite type over $k$, and let $f\colon X \rightarrow Y$ be a dominant map.
Then, the image of the induced map
       $\pi_1(X) \rightarrow \pi_1(Y)$ has finite index in $\pi_1(Y)$. If, in addition, the geometric
       generic fiber of $f$ is connected, then the map $\pi_1(X) \rightarrow \pi_1(Y)$ is surjective.
\end{proposition}
\begin{proof}
To begin, we reduce to the case in which $f$ is smooth.
By generic smoothness, we may replace $X$
by a dense open $X' \subset X$ so that $f|_{X'}$ is smooth.
Since, $\pi_1(X') \rightarrow \pi_1(X)$ is a surjection by Lemma~\ref{lemma:surjectivity-criterion-pi1}, in order to prove the proposition, we may replace $X$ by $X'$.

The last sentence of this Proposition follows from Lemma~\ref{lemma:surjectivity-criterion-pi1} (here we only needed that the map be $f$ be flat, but we have already reduced to the case it is smooth). To conclude, we only need prove that the image of $\pi_1(X)\rightarrow \pi_1(Y)$ has finite index in $\pi_1(Y)$, without the assumption that the geometric generic fiber of $f$ is connected. Since $f$ is smooth and $Y$ is Deligne-Mumford, we can find a scheme $U$ and a dominant \'etale map $U \rightarrow X$ so that $U \rightarrow Y$ factors through $\mathbb A^N_Y$ where $N$ is the dimension of the geometric generic fiber of $f$ and $U \rightarrow \mathbb A^N_Y$ \'etale.
So, after passing to a dense open substack of $W \subset \mathbb A^N_Y$ and a dense open subscheme $U' \subset U$, we may assume that $U' \rightarrow W$ is a finite \'etale cover:
To see why, take a smooth cover of $\mathbb A^N_Y$ by a scheme. The pullback to $U$ is a separated algebraic space, so it has a dense open subspace that is a scheme by \cite[Theorem 6.4.1]{olsson2016algebraic}. The finiteness claim then follows because the resulting \'etale morphism of schemes is locally quasi-finite, of finite type, and quasi-separated, hence generically finite on the target.
Since $U' \rightarrow W$ is finite \'etale, $\pi_1(U') \rightarrow \pi_1(W)$ has finite index.
Because the maps $\pi_1(W) \rightarrow \pi_1(\mathbb A^N_Y)$ and $\pi_1(\mathbb A^N_Y) \rightarrow \pi_1(Y)$ are surjective by Lemma~\ref{lemma:surjectivity-criterion-pi1}, the composition $\pi_1(U') \rightarrow \pi_1(Y)$ has finite index in $\pi_1(Y)$, and hence so does $\pi_1(X) \rightarrow \pi_1(Y)$.
\end{proof}

\subsection{Applications}\label{androidbeatsios}
Let $K$ be a number field with fixed algebraic closure $\ol{K}$, let $\mg$ denote the moduli stack of curves of genus $g$ over $K$,
and let $\ag$ denote the moduli stack of PPAVs of dimension $g$ over $K$. We have a natural map $\tau_g \colon \mg \rightarrow \ag$ given by the Torelli map, which sends a curve to its Jacobian. Let $\ug$ denote the universal family over $\ag$. Note that if $U$ is any scheme and $A \to U$ is a family of PPAVs, then there exist maps $A \to \ug$ and $U \to \ag$ so that $A$ equals the fiber product $U \times_\ag \ug$.

We will also be interested in the locus of smooth hyperelliptic curves of genus $g$, $\hyperelliptic g \subset \mg$, and locus of trigonal curves of genus $g$, $\trigonal g \subset \mg$.
If a curve $C$ is trigonal, there exists a unique nonnegative integer $M$, called the Maroni invariant, with the property that there is a canonical embedding into the Hirzebruch surface $\mathbb F_M \defeq \mathbb P_{\mathbb P^1} (\mathcal O_{\mathbb P^1} \oplus \mathcal O_{\mathbb P^1}(M))$. As mentioned in \cite{patel2015chow}, the Maroni invariant takes on all integer values between $0$ and $\frac{g+2}{3}$ with the same parity as $g$. Let $\trigonal g(M) \subset \mg$ denote the substack of trigonal \mbox{curves of Maroni invariant $M$.}

In order to more easily utilize Proposition~\ref{proposition:stacky-dominant-map-surjective-fundamental-group} for the purpose of giving interesting examples of Theorem ~\ref{theorem:main}, we record the following easy consequence of Proposition~\ref{proposition:stacky-dominant-map-surjective-fundamental-group}:
\vspace*{-0.2cm}
\begin{corollary} \label{corollary:criterion-for-applying-main}
       Let $U \subset \mathbb P^r_K$ be an open subscheme, and let $A \rightarrow U$ be a family of $g$-dimensional PPAVs. Let $\phi \colon U \rightarrow \ag$ be the map induced by the universal property of $\ag$.
Let $V$ be the smallest locally closed substack of $\ag$ through which $U$ factors, and let $W \subset \ag$ be a normal integral substack.
Suppose further that $W \cap V$ is dense in $W$ and that $V$ is normal.
Then, if $W$ has big monodromy, so do $V$ and $U$.
Furthermore, if the geometric generic fiber of $\phi$ is irreducible, then the monodromy of $V$ agrees with that of $U$. In particular, the conclusion of Theorem ~\ref{theorem:main}
holds for $U$.
\end{corollary}
\vspace*{-0.4cm}
\begin{proof}
	By Lemma~\ref{lemma:surjectivity-criterion-pi1}, if $W$ has big monodromy so does the dense open subset $W \cap V \subset W$.
	Therefore, $V$ has big monodromy, because it contains $W \cap V$, which has big monodromy.
	The result then follows from Proposition~\ref{proposition:stacky-dominant-map-surjective-fundamental-group},
once we verify that both $U$ and $V$ are normal, irreducible, separated, and finite type over $K$, with $V$ Deligne-Mumford.
All of these conditions are immediate except possibly that $V$
is generically smooth, which holds by generic smoothness on a smooth \mbox{cover of $V$ by a scheme.}
\end{proof}
\vspace*{-0.2cm}
Before stating the main theorem of this section, we pause to describe more precisely what we mean
by ``the locus of plane curves.''
\begin{remark}
	\label{remark:locus-of-plane-curves}
	In Theorem \ref{corollary:abbreviated-examples}(c) and Theorem \ref{corollary:examples}\autoref{big-plane},
	we refer to the ``substack of Jacobians of plane curves of degree $d$,'' for $d \geq 3$,
	and we now make more precise what we mean
	by this locus.
	When $d=3$, all $1$-dimensional abelian varieties can be realized as the Jacobian of a degree $3$ plane curve,
so in this case we take the locus to be all of $\mathscr M_{1,1}$.
	For $d \geq 4$, we will define a locally closed substack of $\mg$, where $g = \binom{d-1}{2}$,
	and the locus of Jacobians of plane curves of degree $d$ will denote the image of this under the Torelli map.
	For $d \geq 4$, let $\pi_d\colon \scv_d \ra \bp^{\binom{d+2}{2}-1}$ denote the universal family over 
	the Hilbert scheme of plane curves of degree $d$,
	and let $U_d \subset \bp^{\binom{d+2}{2}-1}$ denote the dense open subscheme 
	over which $\pi_d$ is smooth.
	Since $\scv_d|_{U_d} \subset U_d \times \bp^2$,
	the action of $\PGL_3$ on $\bp^2$ induces an action on $\scv_d|_{U_d}$ and hence on $U_d$.
	Then, we define the substack of Jacobians of plane curves of degree $d$
	to be the stack theoretic quotient $[U_d/\PGL_3]$.
	
	Note that there is a natural map $[U_d/\PGL_3] \ra \mg$.
	It can be verified that this map is a locally closed immersion of stacks. Further, one
	can show $[U_d/\PGL_3]$ represents the functor associating to any base scheme $T$
	projective flat morphisms $f \colon C \ra T$ where each geometric fiber is a proper smooth curve of genus $g := \binom{d-1}{2}$
	with a degree $d$ invertible sheaf on $C$ which commutes with base change.	
	In this sense, $[U_d/\PGL_3]$ may naturally be referred to as ``the locus of plane curves of degree $d$''
	and it is evidently smooth, since $U_d$ is smooth, being a dense open subscheme of projective space.

	Let us now briefly sketch the proof of the two facts claimed above.
	First, one can first see that $[U_d/\PGL_3]$ represents the claimed
	functor by defining natural maps both ways and verifying they are mutually inverse.
	To show $[U_d/\PGL_3] \ra \mg$ is a locally closed immersion,
	one can factor $[U_d/\PGL_3] \ra \mg$ through the stack $G^2_d$ parameterizing $g^2_d$'s on 
	the universal curve over $\mg$,
	via a natural generalization of the definition given in \cite[Chapter XXI, Definition 3.12]{ACMG:geometryOfCurves}.
	One can check the map $[U_d/\PGL_3] \ra G^2_d$ is an open immersion from the definitions.
	Finally, one can verify that the map $G^2_d \ra \mg$ is a locally closed immersion, using that
	every smooth plane curve of degree at least $4$ has a unique $g^2_d$, see \cite[Appendix A, Exercises 17 and 18]{ACGH:I},
	and the valuative criterion for locally closed immersions \cite[Chapter 1, Corollary 2.13]{mochizuki2014foundations}.
\end{remark}

We are now in position to state and prove the main theorem of this section:
\vspace*{-0.2cm}
\begin{theorem} \label{corollary:examples}
Suppose $A \rightarrow U$ is a rational family of principally polarized abelian varieties
and define $V$ to be the smallest locally closed substack of $\ag$ through which $U$ factors.
The conclusion of Theorem~\ref{theorem:main} holds whenever $V$ is normal and contains a dense open substack of one of the following loci:
\begin{enumerate}
	       \item[\customlabel{big-hyperelliptic}{(a)}] The locus $\tau_g(\hyperelliptic g)$ for any $g \geq 1$.
		      For every $g \geq 1$, there exists a $U$ dominating $\tau_g(\hyperelliptic g)$ because $\hyperelliptic g$ is unirational.
               \item[\customlabel{big-maroni}{(b)}] The locus $\tau_g(\trigonal g(M))$ of Jacobians of trigonal curves with Maroni invariant $M < \frac{g}{3}-1$ for any $g \geq 5$. In this case, there exists $U$ dominating $\tau_g(\trigonal g(M))$ because \mbox{$\trigonal g(M)$ is unirational.}
		\item[\customlabel{big-trigonal}{(c)}] The locus of trigonal curves $\trigonal g$ in any $g \geq 3$.
We can take $U$ to be any open subscheme of $\trigonal g$, as $\trigonal g$ is rational.
               \item[\customlabel{big-plane}{(d)}] The locus of Jacobians of degree-$d$ plane curves for any $d \geq 3$. In this case, the open subscheme of the Hilbert scheme
of degree-$d$ plane curves parameterizing smooth curves is rational and dominates the locus of Jacobians of degree-$d$ plane curves.
\item[\customlabel{big-mg}{(e)}] The locus $\tau_g(\mg)$ for any $g \geq 1$. In this case, when $1 \leq g \leq 14$, $\mg$ is unirational, so there exists a $U$ dominating $\mg$. Moreover, when $3 \leq g \leq 6$, $\mg$ is rational, and so we may take $U$ to be any open subscheme of $\mg$.
               \item[\customlabel{big-ag}{(f)}] The locus $\ag$ for any $g \geq 1$. When $1 \leq g \leq 5$, $\ag$ is unirational, so such a $U$ exists.
\end{enumerate}
\end{theorem}
\begin{proof}
By Corollary~\ref{corollary:criterion-for-applying-main}, it suffices to check that each of the families enumerated above has a dense open substack which has big monodromy, is irreducible, and is normal, and to verify the rationality and unirationality claims made above.
Irreducibility of these loci is well-known.
Note that in the first five cases, if we denote the locus in question by $\tau_g(W) \subset \ag$, it suffices to verify that $W \subset \mg$ is smooth as a substack of $\mg$, as we now explain.
First, $\tau_g(W) \subset \ag$ is generically smooth because it is reduced, since it is the image of $W$, which is reduced.
Taking a smooth dense open $Z' \subset \tau_g(W)$, we have that $\tau_g^{-1}(Z') \subset W$ is a dense open substack, hence it is also
smooth and has big monodromy. This implies $Z'$ also has big monodromy since the monodromy of a locus in $\mg$ agrees with the monodromy of its image in $\ag$ under $\tau_g$, as both can be identified with the monodromy action on
the first cohomology group.
We now conclude the proof by verifying that each locus in $\mg$ (in the first five cases) is normal, has big monodromy, and is rational or unirational when claimed. In fact, we just show the
substack has big geometric monodromy, since this implies it has big monodromy by
Proposition~\ref{proposition:big-geometric-monodromy-reduction}.

\begin{enumerate}[(a)]
              \item The hyperelliptic locus, $\hyperelliptic g$, has big geometric monodromy as was shown independently in~\cite[Lemma 8.12]{mumford:tata-lectures-on-theta-ii} and~\cite[Th\'eor\`eme 1]{acampo:tresses-monodromie-et-le-groupe-symplectique}.
	       The hyperelliptic locus $\hyperelliptic g$ is smooth and unirational because it is the quotient of an open subscheme of
	       $\mathbb P^{2g+2}_K$ by the smooth action of $\PGL_2$.
\item By~\cite[Theorem, p.~2]{bolognesi2016mapping}, $\trigonal g (M)$ has big geometric monodromy when $M < \frac{g}{3}-1$.
	      Additionally, $\trigonal g(M)$ is smooth and unirational because it can be expressed as a quotient
	      $[U/G]$ of a smooth rational scheme $U$ by a smooth group scheme $G$.
	      Here, $G$ is the group of automorphisms of the Hirzebruch surface $\mathbb F_M$
	      and $U$ is an open subscheme of the projectivization of the linear system of class $3e + \left(\frac{g + 3M + 2}{2}\right) f$ on $\mathbb F_M$, where $f$ is the class of the fiber over $\mathbb P^1$ and $e$
	      is the unique section with negative self-intersection (see ~\cite[p.~8]{bolognesi2016mapping} for an explanation
		of this description of $U$).
	       Note that in this application, we are implicitly translating between
the topological monodromy representation of $\mg$ described in ~\cite[Theorem, p.~2]{bolognesi2016mapping} and the algebraic Galois representation in
$\ag$, but these two representations are compatible, essentially because
both are given by the action of the fundamental group on the first cohomology
group.
\item In the case that $g \geq 5$, we have $\trigonal g(g \bmod 2)$ is birational to $\trigonal g$, so $\trigonal g$
	has a smooth dense open with big geometric monodromy by the previous part.
	Next, $\trigonal g$ is rational for $g \geq 5$ by~\cite[Theorem, p.~1]{ma2014rationality}.
	The cases $g = 3, 4$ hold because for such $g$, $\trigonal g$ forms a dense open in $\mg$, which is itself rational and smooth,
	as shown in the proof of part~\ref{big-mg} below.
\item By Remark \ref{remark:locus-of-plane-curves}, the locus of plane curves (as was also defined in Remark \ref{remark:locus-of-plane-curves}) in $\mg$ is smooth.
By~\cite[Th\'{e}or\`{e}me 4]{beauville1986groupe}, the locus of smooth degree-$d$ plane curves in the Hilbert scheme has big geometric monodromy. 
It follows from Lemma~\ref{lemma:surjectivity-criterion-pi1}
that the locus of plane curves has big monodromy.
The locus of \emph{smooth} degree-$d$ plane curves in the Hilbert scheme
is certainly rational,
       as it is an open subscheme of the Hilbert scheme of degree-$d$ plane curves, which is itself isomorphic to $\mathbb P_K^{\binom{d+2}{2}-1}$.
\item By~\cite[5.12]{deligne1969irreducibility}, the geometric monodromy of $\mg$ is all of $\Sp_{2g}(\widehat {\mathbb Z})$ for every $g \geq 1$. (Alternatively, the fact that $\mg$ has big geometric monodromy follows immediately from the corresponding fact for any one of parts (a)--(d).) Next, $\mg$ is smooth by \cite[Theorem 5.2]{deligne1969irreducibility}. We have that $\mg$ is unirational for $1 \leq g \leq 14$ by~\cite{verra2005unirationality}. Moreover, when $3 \leq g \leq 6$, we have that $\mg$ is rational; see~\cite[p.~2]{casnati2007rationality} for comprehensive references.
\item Note that $\ag$ has geometric big monodromy because $\ag$ contains $\mg$ and $\mg$ has monodromy $\Sp_{2g}(\widehat {\mathbb Z})$, as argued in point (d).
Further,
		       $\ag$ is smooth by \cite[Theorem 2.4.1]{oort:finite-group-schemes-local-moduli-for-abelian-varieties-and-lifting-problems}.
	We have that $\ag$ is unirational for $1 \leq g \leq 5$ \mbox{as shown in~\cite[p.~1]{verra2005unirationality}.}\qedhere
\end{enumerate}
\end{proof}

\begin{remark}
       \label{remark:mmmm}
       In most of the cases enumerated in Theorem~\ref{corollary:examples}, we actually know that the geometric monodromy is not only big, but also equal to $\Sp_{2g}(\widehat{\mathbb Z})$.
       By Corollary~\ref{corollary:criterion-for-applying-main},
       this occurs when $U$ has irreducible geometric generic fiber
       over any of the following loci:
       \begin{enumerate}
               \item The locus $\trigonal g (M)$ for any $M < \frac{g}{3}-1$, by~\cite[Theorem, p.~2]{bolognesi2016mapping};
               \item The locus of plane curves of degree $d$ with $d$ even, by~\cite[Theoreme 4(i)]{beauville1986groupe};
               \item The locus $\mg$ for any $g$, by~\cite[5.12]{deligne1969irreducibility};
	       \item The locus $\ag$ for any $g$, because $\mg \subset \ag$ and $\mg$ has full monodromy by point (d).
       \end{enumerate}
\end{remark}

\begin{remark}
       If $A \to U$ is a family with $\mono_A^{\on{geom}} = \Sp_{2g}(\zh)$, then the group $\mono_A$ can be determined as follows. The intersection $K \cap \bq^{\on{cyc}}$ is of the form $\bq(\zeta_n)$ for some $n \ge 2$. Let $r_n : \zh \to \bz / n \bz$ be the reduction map. Then
    \[
       \mono_A = \ker (r_n \circ \on{mult}) = \{M \in \Sp_{2g}(\zh) : \on{mult} M \equiv 1 \pmod{n} \},
    \]
       which follows from Remark~\ref{remark:det-rho-is-chi}.
Thus, when the conclusion of the preceding remark holds, Theorem~\ref{theorem:main} tells us the following:
    \begin{itemize}
           \item If $K\neq \bq$, or if $K = \bq$ and $g \ge 3$, then most $u \in U(K)$ have $\mono_{A_u} = \ker (r_n \circ \on{mult})$.
	   \item If $K = \bq$ and $g \in \left\{ 1,2 \right\}$, then most $u \in U(K)$ are such that $[\GSp_{2g}(\zh) : \mono_{A_u}] = 2$.
       \end{itemize}
\end{remark}

\begin{remark} \label{remark:}
	Theorem~\ref{corollary:examples}~\ref{big-hyperelliptic} tells us that if $U$ dominates $\hyperelliptic g$, then the conclusion of Theorem~\ref{theorem:main} holds for $U$. In the case where $U$ has irreducible geometric generic fiber, we can say explicitly what the monodromy group of the family is and what its commutator is. For example, let $\mathscr{Y}_{2g+2,K}$ denote the family of genus-$g$ hyperelliptic curves over $K$ with Weierstrass equation given by $y^2 = x^{2g+2} + a_{2g+1}x^{2g+1} + \dots + a_0$. We show in~\cite[Theorem 1.2]{landesman-swaminathan-tao-xu:hyperelliptic-curves} that most members of $\mathscr{Y}_{2g+2, K}$ have monodromy equal to $\mono_{\mathscr Y_{2g+2,K}}$ (which we explicitly compute) over $K \neq \QQ$, and have index-$2$ monodromy when $K = \QQ$. We neither prove nor state this result precisely here, but a complete statement and proof is given in~\cite{landesman-swaminathan-tao-xu:hyperelliptic-curves}.
\end{remark}

\appendix
\section{Explicit Surjectivity for Abelian Surfaces \\ By Davide Lombardo} \label{lombardstreet}

Let $K$ be a number field and $A/K$ be an abelian surface such that $\operatorname{End}_{\overline{K}}(A)=\mathbb{Z}$. For every place $w$ of $K$ at which $A$ has good reduction, let $\operatorname{Frob}_w$ be the corresponding Frobenius element of $\operatorname{Gal}\left( \overline{K}/K\right)$ and let $f_w(x)$ be the characteristic polynomial of $\operatorname{Frob}_w$ acting on $T_\ell A$, where $\ell$ is any prime different from the residual characteristic of $w$ (as it is well known, this definition is well-posed). Let $F(w)$ be the splitting field over $\mathbb{Q}$ of $f_w(x)$. 
By \autoref{remark:gee},
the Galois group of $F(w)/\mathbb{Q}$ is isomorphic to a subgroup of $(\bz/2\bz)^2 \rtimes S_2 \simeq D_4$, the dihedral group on 4 points.

To state our result we need the following function:
\begin{definition}\label{def:bFunction}
Let $\alpha(g)=2^{10}g^3$ and set
$
b(d,g,h)=\left( (14g)^{64g^2} d \max\left(h, \log d,1 \right)^2 \right)^{\alpha(g)}.
$
\end{definition}

We shall show the following result, which extends \cite[Theorem 1.2]{lombardo2015explicit} to the case of abelian surfaces:
\begin{proposition}\label{prop:Main}
Let $v$ be a place of $K$, of good reduction for $A$, such that the Galois group of $f_v(x)$ is isomorphic to $D_4$. Let $q_v$ be the order of the residue field at $v$. \mbox{For all primes $\ell$, let}
\[
\rho_{\ell^\infty} : \operatorname{Gal}\left( \overline{K}/K \right) \to \operatorname{Aut}(T_\ell A) \cong \operatorname{GL}_4(\mathbb{Z}_\ell)
\]
be the natural $\ell$-adic Galois representation attached to $A/K$. We have $\operatorname{Im} \rho_{\ell^\infty}=\operatorname{GSp}_4(\mathbb{Z}_\ell)$ for all primes $\ell$ that are unramified in $K$ and strictly larger than
\[
\max\{b(2[K:\mathbb{Q}],4,2h(A))^{1/4}, (2q_v)^8 \}.
\]
\end{proposition}

From now on, let $v$ be a place as in the statement of proposition \ref{prop:Main}. 
Notice that $f_v(x)$ is irreducible by assumption, hence all its roots are simple. Moreover, $f_v(x)$ doesn't have any real roots, because (by the Weil conjectures) every root of $f_v(x)$ has absolute value $\sqrt{q_v}$, hence its only possible real roots are $\pm \sqrt{q_v}$. But these are algebraic numbers of degree at most 2 over $\mathbb{Q}$, while $f_v(x)$ is irreducible of degree 4, contradiction.
In particular, the roots of $f_v(x)$ come in complex conjugate pairs, so we shall denote them by $\mu_1, \mu_2, \iota(\mu_1), \iota(\mu_2)$, where $\iota : \mathbb{C} \to \mathbb{C}$ is complex conjugation.  We shall need the following lemma:

\begin{lemma}\label{lemma:Nonzero}
Let $x,y,z$ be three distinct eigenvalues of $\operatorname{Frob}_v$. We have $y^2 \neq xz$.
\end{lemma}
\begin{proof}
Suppose first that $z=\iota(x)$. Then $y^2=x\iota(x)=q_v$, which implies that $y=\pm \sqrt{q_v}$ is a root of $f_v(x)$. As we have already seen, this is a contradiction. Hence, up to renaming the eigenvalues of $\operatorname{Frob}_v$ if necessary, we can assume $x=\mu_1, z=\mu_2$ and $y=\iota(\mu_1)$. Since $\operatorname{Gal}(F(v)/\mathbb{Q})$ is isomorphic to $D_4$ by assumption, there is a $\sigma \in \operatorname{Gal}(F(v)/\mathbb{Q})$ such that $\sigma(\mu_1)=\mu_1, \sigma(\iota(\mu_1))=\iota(\mu_1), \sigma(\mu_2)=\iota(\mu_2)$ and $\sigma(\iota(\mu_2))=\mu_2$. Applying $\sigma$ to the equality $y^2=xz$, that is, $\iota(\mu_1)^2=\mu_1\mu_2$, we get $\iota(\mu_1)^2=\mu_1\iota(\mu_2)$, whence $\iota(\mu_2)=\mu_2$. 
But this implies that $\mu_2$ is real, which is once again a contradiction.
\end{proof}

\begin{proof}[Proof of Proposition \ref{prop:Main}]
Let $\ell$ be a prime unramified in $K$ and strictly larger than $b(2[K:\mathbb{Q}],4,2h(A))^{1/4}$.
Let $\rho_\ell : \operatorname{Gal}\left(\overline{K}/K \right) \to \operatorname{Aut} A[\ell]$ be the natural Galois representation associated with the $\ell$-torsion of $A$. 

Much of the proof of~\cite[Theorem 3.19]{lombardoGL2type} still applies in the current setting, and shows that one of the following holds:
\begin{enumerate}
\item $\operatorname{Im}(\rho_{\ell^\infty}) = \operatorname{GSp}_{4}(\mathbb{Z}_\ell)$
\item the image of $\rho_\ell$ is contained in a maximal subgroup of $\operatorname{GSp}_4(\mathbb{F}_\ell)$ of type (2) in the sense of Theorem 3.3 in~\cite{lombardoGL2type}.
\end{enumerate}
If we are in case (1) we are done, so assume we are in case (2). To conclude the proof, we shall show that $\ell \leq (2q_v)^8$. If $\ell$ is equal to the residual characteristic of $v$ this inequality is obvious, so we can assume that $v \nmid \ell$. In this case, the characteristic polynomial of the action of $\operatorname{Frob}_v$ on $T_\ell A$ is $f_v(x)$.
 By \cite[Lemma 3.4]{lombardoGL2type}, the eigenvalues of any $x \in \operatorname{Im}(\rho_\ell)$ can be written as $\lambda \cdot \lambda_1^3, \lambda \cdot  \lambda_1^2\lambda_2, \lambda \cdot \lambda_1\lambda_2^2,\lambda \cdot \lambda_2^3$ for some $\lambda, \lambda_1, \lambda_2 \in \mathbb{F}_{\ell^2}^\times$. Taking $g\defeq\rho_\ell(\operatorname{Frob}_v)$, we may assume the four eigenvalues $\nu_1, \ldots, \nu_4$ of $g$ satisfy $\nu_2^2=\nu_1\nu_3$. 

Let $\lambda$ be a place of $F(v)$ of characteristic $\ell$ and identify $\lambda$ with a maximal ideal of $\mathcal{O}_{F(v)}$. Since $f_v(x)$ splits completely in $F(v)$ by definition, its four roots $\mu_1, \mu_2, \iota(\mu_1), \iota(\mu_2)$ all belong to $\mathcal{O}_{F(v)}$. Upon reduction modulo $\lambda$, these four roots yield four elements of $\mathcal{O}_{F(v)}/\lambda$, which is a finite field of characteristic $\ell$. Moreover, as $\{\mu_1, \mu_2, \iota(\mu_1), \iota(\mu_2)\}$ is a Galois-stable set, its image in $\overline{\mathbb{F}_\ell}$ independent of the choice embedding of $\mathcal{O}_{F(v)}/\lambda$ into $\overline{\mathbb{F}_\ell}$, hence well defined. Denote by $\overline{\mu_1}, \overline{\mu_2}, \overline{\iota(\mu_1)}, \overline{\iota(\mu_2)}$ the images of $\mu_1, \mu_2, \iota(\mu_1), \iota(\mu_2)$ in $\overline{\mathbb{F}_\ell}$.

Now observe that the characteristic polynomial of $g$ is the reduction modulo $\ell$ of $f_v(x)$, so its roots $\nu_1, \ldots, \nu_4 \in \overline{\mathbb{F}_\ell}^\times$ must coincide with $\overline{\mu_1}, \overline{\mu_2}, \overline{\iota(\mu_1)}, \overline{\iota(\mu_2)}$ in some order.
Given that $\nu_2^2=\nu_1\nu_3$, there are three (necessarily distinct) eigenvalues of $\operatorname{Frob}_v$, call them $x,y,z$, that satisfy $y^2- xz \equiv 0 \pmod{ \lambda}$.  
By Lemma~\ref{lemma:Nonzero}, $N_{F(v)/\mathbb{Q}}(y^2-xz)$ is a nonzero integer. 
 Therefore, $N_{F(v)/\mathbb{Q}}(y^2-xz)$ has positive valuation at $\lambda$, hence it is divisible by $\ell$. In turn, this gives
 \[
 \ell \leq |N_{F(v)/\mathbb{Q}}(y^2-xz)| = \prod_{\sigma \in \operatorname{Gal}(F(v)/\mathbb{Q})}|\sigma(y)^2-\sigma(x)\sigma(z)| \leq (2q_v)^8,
 \]
where the inequality $|\sigma(y)^2-\sigma(x)\sigma(z)| \leq 2q_v$ follows immediately from the triangle inequality and the Weil conjectures.
\end{proof}

\section*{Acknowledgments}

\noindent This research was supervised by Ken Ono and David Zureick-Brown at the Emory University Mathematics REU and was supported by the National Science Foundation (grant number DMS-1557960). We would like to thank David Zureick-Brown for suggesting the problem that led to the present article and for offering us his invaluable advice and guidance; in particular, we acknowledge David Zureick-Brown for providing a detailed outline of the material on heights in Section~\ref{subsubsection:appendix-on-heights}.
and for much help proving
Proposition~\ref{proposition:stacky-dominant-map-surjective-fundamental-group}. We would like to acknowledge Brian Conrad for his meticulous efforts in providing us with enlightening comments, corrections, and suggestions on nearly every part of this paper. In addition, we would like to thank Davide Lombardo for writing an appendix for the present article and for many fruitful conversations. 
We thank Daniel Litt for much help proving Proposition~\ref{proposition:big-geometric-monodromy-reduction}. We also thank the anonymous referees for their helpful comments and suggestions.
Finally, we would like to thank Jeff Achter, Jarod Alper, Michael Aschbacher, Anna Cadoret, Alina Cojocaru, John Cullinan, Dougal Davis, Anand Deopurkar, Noam Elkies, Jordan Ellenberg, Tony Feng, Nick Gill, Jack Hall, Joe Harris, Eric Katz, Mark Kisin, Ben Moonen, Jackson Morrow, Anand Patel, Bjorn Poonen, Jeremy Rickard, Eric Riedl, Simon Rubinstein-Salzedo, David Rydh, Jesse Silliman, Jacob Tsimerman, Evelina Viada, Erik Wallace, and Alex Wright for their helpful advice. We used {\tt magma} and \mbox{\emph{Mathematica} for explicit calculations.}

\bibliographystyle{alpha}
\bibliography{bibfile}

\newcommand{\etalchar}[1]{$^{#1}$}
\begin{thebibliography}{AdRAK{\etalchar{+}}16}

\bibitem[A'C79]{acampo:tresses-monodromie-et-le-groupe-symplectique}
N.~A'Campo.
\newblock Tresses, monodromie et le groupe symplectique.
\newblock {\em Comment. Math. Helv.}, 54(2):318--327, 1979.

\bibitem[ACG11]{ACMG:geometryOfCurves}
Enrico Arbarello, Maurizio Cornalba, and Phillip~A. Griffiths.
\newblock {\em Geometry of algebraic curves. {V}olume {II}}, volume 268 of {\em
  Grundlehren der Mathematischen Wissenschaften [Fundamental Principles of
  Mathematical Sciences]}.
\newblock Springer, Heidelberg, 2011.
\newblock With a contribution by Joseph Daniel Harris.

\bibitem[ACGH85]{ACGH:I}
E.~Arbarello, M.~Cornalba, P.~A. Griffiths, and J.~Harris.
\newblock {\em Geometry of algebraic curves. {V}ol. {I}}, volume 267 of {\em
  Grundlehren der Mathematischen Wissenschaften [Fundamental Principles of
  Mathematical Sciences]}.
\newblock Springer-Verlag, New York, 1985.

\bibitem[AdRAK{\etalchar{+}}16]{arias2015large}
Sara Arias-de Reyna, C\'ecile Armana, Valentijn Karemaker, Marusia Rebolledo,
  Lara Thomas, and N\'uria Vila.
\newblock Large {G}alois images for {J}acobian varieties of genus 3 curves.
\newblock {\em Acta Arith.}, 174(4):339--366, 2016.

\bibitem[ALS16]{anni2016residual}
S.~Anni, P.~Lemos, and S.~Siksek.
\newblock Residual representations of semistable principally polarized abelian
  varieties.
\newblock {\em Res. Number Theory}, 2:2:1, 2016.

\bibitem[Bea86]{beauville1986groupe}
A.~Beauville.
\newblock Le groupe de monodromie des familles universelles d'hypersurfaces et
  d'intersections compl\`etes.
\newblock In {\em Complex analysis and algebraic geometry ({G}\"ottingen,
  1985)}, volume 1194 of {\em Lecture Notes in Math.}, pages 8--18. Springer,
  Berlin, 1986.

\bibitem[BL16]{bolognesi2016mapping}
M.~Bolognesi and M.~L{\"o}nne.
\newblock Mapping class groups of trigonal loci.
\newblock {\em Selecta Math. (N.S.)}, 22(1):417--445, 2016.

\bibitem[Cad15]{cadoret2015open}
Anna Cadoret.
\newblock An open adelic image theorem for abelian schemes.
\newblock {\em Int. Math. Res. Not. IMRN}, (20):10208--10242, 2015.

\bibitem[CF07]{casnati2007rationality}
G.~Casnati and C.~Fontanari.
\newblock On the rationality of moduli spaces of pointed curves.
\newblock {\em J. Lond. Math. Soc. (2)}, 75(3):582--596, 2007.

\bibitem[CGJ11]{cojocaruGJ:one-parameter-families-of-elliptic-curves}
A.~C. Cojocaru, D.~Grant, and N.~Jones.
\newblock One-parameter families of elliptic curves over {$\Bbb Q$} with
  maximal {G}alois representations.
\newblock {\em Proc. Lond. Math. Soc. (3)}, 103(4):654--675, 2011.

\bibitem[CH05]{cojocaruH:uniform-results-for-serres-theorem-for-elliptic-curves}
A.~C. Cojocaru and C.~Hall.
\newblock Uniform results for {S}erre's theorem for elliptic curves.
\newblock {\em Int. Math. Res. Not.}, (50):3065--3080, 2005.

\bibitem[CM15]{cadoret2015integral}
Anna Cadoret and Ben Moonen.
\newblock Integral and adelic aspects of the mumford-tate conjecture.
\newblock {\em arXiv preprint arXiv:1508.06426}, 2015.

\bibitem[CT12]{cadoretuniform-i}
Anna Cadoret and Akio Tamagawa.
\newblock A uniform open image theorem for {$\ell$}-adic representations, {I}.
\newblock {\em Duke Math. J.}, 161(13):2605--2634, 2012.

\bibitem[CT13]{cadoretuniform-ii}
Anna Cadoret and Akio Tamagawa.
\newblock A uniform open image theorem for {$\ell$}-adic representations, {II}.
\newblock {\em Duke Math. J.}, 162(12):2301--2344, 2013.

\bibitem[Die02]{dooleyfat}
L.~V. Dieulefait.
\newblock Explicit determination of the images of the {G}alois representations
  attached to abelian surfaces with {${\rm End}(A)=\Bbb Z$}.
\newblock {\em Experiment. Math.}, 11(4):503--512 (2003), 2002.

\bibitem[DM69]{deligne1969irreducibility}
P.~Deligne and D.~Mumford.
\newblock The irreducibility of the space of curves of given genus.
\newblock {\em Inst. Hautes \'Etudes Sci. Publ. Math.}, (36):75--109, 1969.

\bibitem[Duk97]{duke:elliptic-curves-with-no-exceptional-primes}
W.~Duke.
\newblock Elliptic curves with no exceptional primes.
\newblock {\em C. R. Acad. Sci. Paris S\'er. I Math.}, 325(8):813--818, 1997.

\bibitem[EEHK09]{ellenbergEHK:non-simple-abelian-varieties-in-a-family}
J.~S. Ellenberg, C.~Elsholtz, C.~Hall, and E.~Kowalski.
\newblock Non-simple abelian varieties in a family: geometric and analytic
  approaches.
\newblock {\em Journal of the London Mathematical Society}, page jdp021, 2009.

\bibitem[Eke90]{ekedahl1988effective}
T.~Ekedahl.
\newblock An effective version of {H}ilbert's irreducibility theorem.
\newblock In {\em S\'eminaire de {T}h\'eorie des {N}ombres, {P}aris
  1988--1989}, volume~91 of {\em Progr. Math.}, pages 241--249. Birkh\"auser
  Boston, Boston, MA, 1990.

\bibitem[Fal86]{FalFinite}
G.~Faltings.
\newblock Finiteness theorems for abelian varieties over number fields.
\newblock In {\em Arithmetic geometry ({S}torrs, {C}onn., 1984)}, pages 9--27.
  Springer, New York, 1986.
\newblock Translated from the German original [Invent. Math. {{\bf{7}}3}
  (1983), no. 3, 349--366; ibid. {{\bf{7}}5} (1984), no. 2, 381; MR
  85g:11026ab] by Edward Shipz.

\bibitem[FWG{\etalchar{+}}92]{FalRP}
G.~Faltings, G.~W{\"u}stholz, F.~Grunewald, Norbert Schappacher, and Ulrich
  Stuhler.
\newblock {\em Rational points}.
\newblock Aspects of Mathematics, E6. Friedr. Vieweg \& Sohn, Braunschweig,
  third edition, 1992.
\newblock Papers from the seminar held at the Max-Planck-Institut f{\"u}r
  Mathematik, Bonn/Wuppertal, 1983/1984, With an appendix by W{\"u}stholz.

\bibitem[GR71a]{noopsortSGA1Grothendieck1971}
A.~Grothendieck and M.~Raynaud.
\newblock {\em Rev\^etements \'etales et groupe fondamental}.
\newblock Springer-Verlag, Berlin-New York, 1971.
\newblock S{\'e}minaire de G{\'e}om{\'e}trie Alg{\'e}brique du Bois Marie
  1960--1961 (SGA 1).

\bibitem[GR71b]{SGA1Grothendieck1971}
A.~Grothendieck and M.~Raynaud.
\newblock {\em Rev\^etements \'etales et groupe fondamental}.
\newblock Springer-Verlag, Berlin-New York, 1971.
\newblock S{\'e}minaire de G{\'e}om{\'e}trie Alg{\'e}brique du Bois Marie
  1960--1961 (SGA 1).

\bibitem[Gra00]{grant:a-formula-for-the-number-of-elliptic-curves-with-exceptional-primes}
D.~Grant.
\newblock A formula for the number of elliptic curves with exceptional primes.
\newblock {\em Compositio Math.}, 122(2):151--164, 2000.

\bibitem[Gre10]{greasy}
A.~Greicius.
\newblock Elliptic curves with surjective adelic {G}alois representations.
\newblock {\em Experiment. Math.}, 19(4):495--507, 2010.

\bibitem[Gro66a]{EGAIV.3}
A.~Grothendieck.
\newblock \'{E}l\'ements de g\'eom\'etrie alg\'ebrique. {IV}. \'{E}tude locale
  des sch\'emas et des morphismes de sch\'emas. {III}.
\newblock {\em Inst. Hautes \'Etudes Sci. Publ. Math.}, (28):255, 1966.

\bibitem[Gro66b]{grothendieck:un-theoreme-sur-les-homomorphismes}
A.~Grothendieck.
\newblock Un th\'eor\`eme sur les homomorphismes de sch\'emas ab\'eliens.
\newblock {\em Invent. Math.}, 2:59--78, 1966.

\bibitem[Gro67]{EGAIV.4}
A.~Grothendieck.
\newblock \'{E}l\'ements de g\'eom\'etrie alg\'ebrique. {IV}. \'{E}tude locale
  des sch\'emas et des morphismes de sch\'emas {IV}.
\newblock {\em Inst. Hautes \'Etudes Sci. Publ. Math.}, (32):361, 1967.

\bibitem[HS00]{afraidofheights}
M.~Hindry and J.~H. Silverman.
\newblock {\em Diophantine geometry}, volume 201 of {\em Graduate Texts in
  Mathematics}.
\newblock Springer-Verlag, New York, 2000.
\newblock An introduction.

\bibitem[Jon10]{josofabank}
N.~Jones.
\newblock Almost all elliptic curves are {S}erre curves.
\newblock {\em Trans. Amer. Math. Soc.}, 362(3):1547--1570, 2010.

\bibitem[Jou82]{jouanolou1982theoremes}
Jean-Pierre Jouanolou.
\newblock Theoremes de bertini et applications.
\newblock 1982.

\bibitem[Kaw03]{ifyouseekamy}
T.~Kawamura.
\newblock The effective surjectivity of mod {$l$} {G}alois representations of
  1- and 2-dimensional abelian varieties with trivial endomorphism ring.
\newblock {\em Comment. Math. Helv.}, 78(3):486--493, 2003.

\bibitem[Kow06]{kowalski2006large}
E.~Kowalski.
\newblock The large sieve, monodromy and zeta functions of curves.
\newblock {\em J. Reine Angew. Math.}, 601:29--69, 2006.

\bibitem[LMB00]{laumon-bailly:champs-algebriques}
G.~Laumon and L.~Moret-Bailly.
\newblock {\em Champs alg\'ebriques}, volume~39 of {\em Ergebnisse der
  Mathematik und ihrer Grenzgebiete. 3. Folge. A Series of Modern Surveys in
  Mathematics [Results in Mathematics and Related Areas. 3rd Series. A Series
  of Modern Surveys in Mathematics]}.
\newblock Springer-Verlag, Berlin, 2000.

\bibitem[{L}om15]{lombardo2015explicit}
{D}. {L}ombardo.
\newblock {Explicit open image theorems for abelian varieties with trivial
  endomorphism ring}.
\newblock {\em arXiv:1508.01293v2}, August 2015.

\bibitem[Lom16]{lombardoGL2type}
D.~Lombardo.
\newblock Explicit surjectivity of {G}alois representations for abelian
  surfaces and {GL{$_2$}}-varieties.
\newblock {\em J. Algebra}, 460:26--59, 2016.

\bibitem[LSTX17]{landesman-swaminathan-tao-xu:lifting-symplectic-group}
Aaron Landesman, Ashvin~A. Swaminathan, James Tao, and Yujie Xu.
\newblock Lifting subgroups of symplectic groups over {$\Bbb{Z}/\ell\Bbb{Z}$}.
\newblock {\em Res. Number Theory}, 3:Art. 14, 12, 2017.

\bibitem[LSTX20]{landesman-swaminathan-tao-xu:hyperelliptic-curves}
Aaron Landesman, Ashvin~A. Swaminathan, James Tao, and Yujie Xu.
\newblock Hyperelliptic curves with maximal {G}alois action on the torsion
  points of their {J}acobians.
\newblock {\em Indiana Univ. Math. J.}, 69(7):2461--2492, 2020.

\bibitem[Ma15]{ma2014rationality}
S.~Ma.
\newblock The rationality of the moduli spaces of trigonal curves.
\newblock {\em Int. Math. Res. Not. IMRN}, (14):5456--5472, 2015.

\bibitem[Moc14]{mochizuki2014foundations}
Shinichi Mochizuki.
\newblock {\em Foundations of p-adic Teichm{\"u}ller Theory}, volume~11.
\newblock American Mathematical Soc., 2014.

\bibitem[Mor01]{morris:introduction-to-arithmetic-groups}
Dave~Witte Morris.
\newblock Introduction to arithmetic groups.
\newblock {\em arXiv preprint math/0106063v6}, 2001.

\bibitem[Mum07]{mumford:tata-lectures-on-theta-ii}
D.~Mumford.
\newblock {\em Tata lectures on theta. {II}}.
\newblock Modern Birkh\"auser Classics. Birkh\"auser Boston, Inc., Boston, MA,
  2007.
\newblock Jacobian theta functions and differential equations, With the
  collaboration of C. Musili, M. Nori, E. Previato, M. Stillman and H. Umemura,
  Reprint of the 1984 original.

\bibitem[Ols16]{olsson2016algebraic}
Martin Olsson.
\newblock {\em Algebraic spaces and stacks}, volume~62 of {\em American
  Mathematical Society Colloquium Publications}.
\newblock American Mathematical Society, Providence, RI, 2016.

\bibitem[O'M78]{omeara1978symplectic}
O.~T. O'Meara.
\newblock {\em Symplectic groups}, volume~16 of {\em Mathematical Surveys}.
\newblock American Mathematical Society, Providence, R.I., 1978.

\bibitem[Oor71]{oort:finite-group-schemes-local-moduli-for-abelian-varieties-and-lifting-problems}
F.~Oort.
\newblock Finite group schemes, local moduli for abelian varieties, and lifting
  problems.
\newblock {\em Compositio Math.}, 23:265--296, 1971.

\bibitem[OV00]{orgogozo2000theoreme}
F.~Orgogozo and I.~Vidal.
\newblock Le th\'eor\`eme de sp\'ecialisation du groupe fondamental.
\newblock In {\em Courbes semi-stables et groupe fondamental en g\'eom\'etrie
  alg\'ebrique ({L}uminy, 1998)}, volume 187 of {\em Progr. Math.}, pages
  169--184. Birkh\"auser, Basel, 2000.

\bibitem[PV15]{patel2015chow}
A.~{Patel} and R.~{Vakil}.
\newblock {On the Chow ring of the Hurwitz space of degree three covers of
  ${\bf P}^{1}$}.
\newblock {\em arXiv:1505.04323v1}, May 2015.

\bibitem[Rib76]{ribbit}
K.~A. Ribet.
\newblock Galois action on division points of {A}belian varieties with real
  multiplications.
\newblock {\em Amer. J. Math.}, 98(3):751--804, 1976.

\bibitem[Riv08]{rivin2008walks}
I.~Rivin.
\newblock Walks on groups, counting reducible matrices, polynomials, and
  surface and free group automorphisms.
\newblock {\em Duke Math. J.}, 142(2):353--379, 2008.

\bibitem[Ser72]{causalrelationship}
J.-P. Serre.
\newblock Propri\'et\'es galoisiennes des points d'ordre fini des courbes
  elliptiques.
\newblock {\em Invent. Math.}, 15(4):259--331, 1972.

\bibitem[Ser97]{serre1989lectures}
J.-P. Serre.
\newblock {\em Lectures on the {M}ordell-{W}eil theorem}.
\newblock Aspects of Mathematics. Friedr. Vieweg \& Sohn, Braunschweig, third
  edition, 1997.
\newblock Translated from the French and edited by Martin Brown from notes by
  Michel Waldschmidt, With a foreword by Brown and Serre.

\bibitem[Ser98]{serre1989abelian}
J.-P. Serre.
\newblock {\em Abelian {$l$}-adic representations and elliptic curves},
  volume~7 of {\em Research Notes in Mathematics}.
\newblock A K Peters, Ltd., Wellesley, MA, 1998.
\newblock With the collaboration of Willem Kuyk and John Labute, Revised
  reprint of the 1968 original.

\bibitem[{Sta}16]{stacks-project}
The {Stacks Project Authors}.
\newblock {\it Stacks Project}.
\newblock \url{http://stacks.math.columbia.edu}, 2016.

\bibitem[Ste61]{eliotsteinclub}
R.~Steinberg.
\newblock Automorphisms of classical {L}ie algebras.
\newblock {\em Pacific J. Math.}, 11:1119--1129, 1961.

\bibitem[Ver05]{verra2005unirationality}
A.~Verra.
\newblock The unirationality of the moduli spaces of curves of genus 14 or
  lower.
\newblock {\em Compos. Math.}, 141(6):1425--1444, 2005.

\bibitem[Wal14]{scoopdedoo}
E.~Wallace.
\newblock Principally polarized abelian surfaces with surjective {G}alois
  representations on {$l$}-torsion.
\newblock {\em J. Lond. Math. Soc. (2)}, 90(2):451--471, 2014.

\bibitem[Zyw10a]{zywina2010elliptic}
D.~Zywina.
\newblock Elliptic curves with maximal {G}alois action on their torsion points.
\newblock {\em Bull. Lond. Math. Soc.}, 42(5):811--826, 2010.

\bibitem[{Zyw}10b]{zywina2010hilbert}
D.~{Zywina}.
\newblock {Hilbert's irreducibility theorem and the larger sieve}.
\newblock {\em arXiv:1011.6465v1}, November 2010.

\bibitem[{Z}yw15]{seaweed}
{D}. {Z}ywina.
\newblock {An explicit Jacobian of dimension 3 with maximal Galois action}.
\newblock {\em arXiv:1508.07655v1}, August 2015.

\end{thebibliography}

\end{document}